\pdfoutput=1
\documentclass[12pt]{amsbook}
\usepackage{etex} 
\usepackage{amscd}
\usepackage{amssymb}
\usepackage[noadjust,space]{cite}
\usepackage{booktabs}
\usepackage{url}
\usepackage{hyphenat}
\usepackage{epsf}
\usepackage{mathtools}
\usepackage{mathrsfs, stmaryrd}
\usepackage{array}
\usepackage[curve]{xypic}
\usepackage[T1]{fontenc}
\usepackage{enumitem}
\usepackage[utf8]{inputenc}
\usepackage[pdftex,lmargin=1in,rmargin=1in,tmargin=1in,bmargin=1in,headheight=16pt]{geometry}
\usepackage{xr-hyper}
\usepackage[bookmarks=true, bookmarksopen=true,%
bookmarksdepth=3,bookmarksopenlevel=2,%
colorlinks=true,%
linkcolor=blue,%
citecolor=blue,%
filecolor=blue,%
menucolor=blue,%
urlcolor=blue]{hyperref}
\hypersetup{pdftitle={Notes on bordered Floer homology}}
\hypersetup{pdfauthor={Robert Lipshitz, Peter Ozsvath, Dylan Thurston}}
\usepackage[percent]{overpic}
\usepackage{tikz}
\usetikzlibrary{arrows}
\usepackage[stable]{footmisc}
\usepackage{fancyhdr}

\newread\testin

\graphicspath{{draws/}{}}

\def\mathcenter#1{%
  \vcenter{\hbox{$#1$}}%
}

\def\mfigb#1{
        \mathcenter{\includegraphics[trim=-1 -1 -1 -1]{#1}}
}

\DeclareRobustCommand{\widebar}[1]{\overline{#1}{}}

\makeatletter
\newcommand\mi@kern[1]{%
  \settowidth\@tempdima{$\mi@obj^{#1}$}
  \kern-\@tempdima
  #1
  \settowidth\@tempdima{$\mi@obj$}
  \kern\@tempdima
}

\newtoks\mi@toksp
\newtoks\mi@toksb
\DeclareRobustCommand{\manyindices}[5]{
  \def\mi@obj{#5}
  \mi@toksp\expandafter{\mi@kern{#2}}
  \mi@toksb\expandafter{\mi@kern{#1}}
  \@mathmeasure4\textstyle{#5_{#1}^{#2}}
  \@mathmeasure6\textstyle{#5_{#3}^{#4}}
  \dimen0-\wd6 \advance\dimen0\wd4
  \@mathmeasure8\textstyle{\hphantom{{}_{#1}^{#2}}#5^{\the\mi@toksp#4}_{\the\mi@toksb#3}}
  \hbox to \dimen0{}{\kern-\dimen0\box8}
}
\makeatother 

\usepackage{ifpdf}
\ifpdf
  
\else
  
\fi

\newcommand{\RR}{\mathbb R}
\newcommand{\CC}{\mathbb C}
\newcommand{\ZZ}{\mathbb Z}

\newcommand{\FF}{\mathbb F}
\newcommand{\NN}{\mathbb N}

\newcommand{\Image}{\mathrm{Im}}
\newcommand{\bD}{\mathbb{D}}

\newcommand{\co}{\nobreak\mskip2mu\mathpunct{}\nonscript
  \mkern-\thinmuskip{:}\penalty300\mskip6muplus1mu\relax}

\newcommand{\bdy}{\partial}

\newcommand{\lbracket}{[}
\newcommand{\rbracket}{]}

\newcommand{\Hyph}{\text{-}}
\def\hyph{-\penalty0\hskip0pt\relax} 

\DeclareMathOperator{\Sym}{Sym}

\DeclareMathOperator{\Aut}{Aut}
\DeclareMathOperator{\Hom}{Hom}

\DeclareMathOperator{\Ext}{Ext}

\DeclareMathOperator{\rank}{rank}

\DeclareMathOperator{\ind}{ind}

\DeclareMathOperator{\Inv}{Inv}
\DeclareMathOperator{\ev}{ev}
\DeclareMathOperator{\gr}{gr}

\numberwithin{figure}{chapter}

\theoremstyle{plain}
\numberwithin{equation}{chapter}
\newtheorem{theorem}[equation]{Theorem}

\newtheorem{proposition}[equation]{Proposition}
\newtheorem{lemma}[equation]{Lemma}
\newtheorem{corollary}[equation]{Corollary}

\newtheorem{definition}[equation]{Definition}
\newtheorem{construction}[equation]{Construction}

\theoremstyle{definition}
\newtheorem{exercise}{Exercise}
\numberwithin{exercise}{chapter}

\theoremstyle{remark}
\newtheorem{example}[equation]{Example}
\newtheorem{remark}[equation]{Remark}

\newcommand{\HF}{\mathit{HF}}
\newcommand{\HFa}{\widehat {\HF}}

\newcommand{\CFa}{\widehat {\mathit{CF}}}
\newcommand{\HFKa}{\widehat{\mathit{HFK}}}
\newcommand{\x}{\mathbf x}
\newcommand{\y}{\mathbf y}
\newcommand{\z}{\mathbf z}
\newcommand{\w}{\mathbf w}

\newcommand\HH{\mathit{HH}}

\newcommand\Hochschild\HH

\newcommand{\Ainf}{A_\infty}
\newcommand{\Alg}{\mathcal{A}}

\newcommand{\cB}{{\mathcal{B}}}

\newcommand{\alphas}{{\boldsymbol{\alpha}}}
\newcommand{\betas}{{\boldsymbol{\beta}}}

\newcommand{\rhos}{{\boldsymbol{\rho}}}

\newcommand{\cM}{\mathcal{M}}
\newcommand{\Mod}{\cM}

\newcommand{\ocM}{\widebar{\cM}{}} 

\newcommand{\tcM}{\widetilde{\mathcal{M}}}
\newcommand{\DD}{\textit{DD}}
\newcommand{\DA}{\textit{DA}}

\newcommand{\CFD}{\mathit{CFD}}
\newcommand{\CFDD}{\mathit{CFDD}}
\newcommand{\CFA}{\mathit{CFA}}

\newcommand{\CFDA}{\mathit{CFDA}}
\newcommand{\CFDAa}{\widehat{\CFDA}}
\newcommand{\CFAA}{\mathit{CFAA}}
\newcommand{\CFAAa}{\widehat{\CFAA}}
\newcommand{\CFDa}{\widehat{\CFD}}

\newcommand{\CFK}{\mathit{CFK}}
\newcommand{\CFKa}{\widehat{\CFK}}
\newcommand{\CFKm}{\CFK^-}

\newcommand{\CFDDa}{\widehat{\CFDD}}

\newcommand{\CFAa}{\widehat{\CFA}}

\newcommand{\cZ}{\mathcal{Z}}
\newcommand{\PtdMatchCirc}{\cZ}
\newcommand{\PMC}{\PtdMatchCirc}

\newcommand{\CircPts}{{\mathbf{a}}}

\newcommand{\dg}{\textit{dg} }

\newcommand\Id{\mathbb{I}}

\newcommand\Mas{\ind}
\newcommand\DTP{\mathop{\widetilde\otimes}\nolimits}
\newcommand\DT{\boxtimes}
\newcommand\Gen{\mathfrak{S}}

\newcommand\Zmod[1]{\mathbb{Z}/{#1}\mathbb{Z}}
\newcommand{\Field}{{\FF_2}}
\DeclareMathOperator{\nbd}{nbd}

\newcommand{\Heegaard}{\mathcal{H}}
\newcommand{\HD}{\Heegaard}

\DeclareMathOperator{\Fix}{Fix}

\DeclareMathOperator{\Mor}{Mor}

\newcommand\drHD{{\mathcal H}_{dr}}

\newcommand{\op}{\mathrm{op}}
\newcommand\PunctF{F^\circ}
\newcommand{\Chord}{\mathrm{Chord}}

\newcommand{\SetS}{\mathbf{s}}
\newcommand{\SetT}{\mathbf{t}}

\makeatletter
\newcommand\honestalg[3]{\bigl\lbracket
\begin{smallmatrix} #1\@ifempty{#3}{}{&#3} \\ #2 \end{smallmatrix}
\bigr\rbracket}

\makeatother
\newcommand{\lab}[1]{$\scriptstyle #1$}

\newcommand{\sos}[3]{\mathbin{{}_{#1}\mathord#2_{#3}}}

\newcommand{\lsup}[2]{{}^{#1}\mskip-.6\thinmuskip#2}

\newcommand{\arcz}{\mathbf{z}}
\newcommand{\nilCox}{\mathcal{N}}
\newcommand{\gCFKm}{\mathit{gCFK}^-}
\newcommand{\gHFKm}{\mathit{gHFK}^-}
\newcommand\piBig{{\widetilde\pi}_2}
\newcommand{\tpsi}{{\widetilde{\psi}}}
\newcommand{\HB}{\mathsf{H}}
\newcommand{\support}[1]{[#1]}
\newcommand{\NearChord}{\mathrm{NChord}}

 \makeatletter
     \def\revddots{\mathinner{\mkern1mu\raise\p@
       \vbox{\kern7\p@\hbox{.}}\mkern2mu
       \raise4\p@\hbox{.}\mkern2mu\raise7\p@\hbox{.}\mkern1mu}}
   \makeatother

\begin{document}
\title{Notes on bordered Floer homology}

\author{Robert Lipshitz}
\address{Department of Mathematics, Columbia University\\
  New York, NY 10027}
\email{lipshitz@math.columbia.edu}

\author[Ozsv\'ath]{Peter Ozsv\'ath}
\address {Department of Mathematics, Princeton University\\ New
  Jersey, 08544}
\email {petero@math.princeton.edu}

\author[Thurston]{Dylan~P.~Thurston}
\author[Thurston]{Dylan~P.~Thurston}
\address{Department of Mathematics,
        Indiana University\\
        Bloomington, IN 47405}
\email{dpthurst@indiana.edu}

\maketitle

\tableofcontents

\chapter*{Introduction}

Heegaard Floer homology is a kind of $(3+1)$-dimensional topological
field theory defined by the second author and Z.~Szab\'o. More
precisely, one variant of Heegaard Floer homology associates to each
connected, oriented $3$-manifold $Y$ an abelian group
$\HFa(Y)$~\cite{OS04:HolomorphicDisks} (see also~\cite{JT:Naturality}), and to each smooth, connected,
$4$-dimensional cobordism $W$ from $Y_1$ to $Y_2$ a group homomorphism
$\hat{F}\co \HFa(Y_1)\to\HFa(Y_2)$~\cite{OS06:HolDiskFour}. This
assignment is functorial: composition of cobordisms corresponds to
composition of maps. As the name suggests, the Heegaard Floer homology
groups are the homologies of chain complexes $\CFa(Y)$, defined via
Lagrangian-intersection Floer homology\footnote{Strictly speaking, in
  the original definition the manifolds were only totally-real, not
  Lagrangian. It was shown in~\cite{Perutz07:handleslides} that a
  K\"ahler form can be chosen making the relevant submanifolds
  Lagrangian.}.  The invariant is also multiplicative: the chain
complex $\CFa(Y_1\# Y_2)$ associated to the connected sum of $Y_1$ and
$Y_2$ is the tensor product $\CFa(Y_1)\otimes\CFa(Y_2)$ of the chain
complexes associated to $Y_1$ and $Y_2$. 
The other variants of Heegaard Floer homology---$\HF^+(Y)$, $\HF^-(Y)$
and $\HF^\infty(Y)$---are modules over $\ZZ[U]$, but otherwise behave
fairly similarly to $\HFa(Y)$ (but see point~(\ref{item:SW}) below).

Heegaard Floer homology has received widespread attention largely
because of its
striking topological
applications. Many of these applications
draw on the remarkable geometric content of the 
Heegaard Floer invariants:
\begin{enumerate}
\item The group $\HFa(Y)$ detects the Thurston norm of $Y$; similarly,
  the variant of Heegaard Floer homology $\HFKa(Y,K)$ associated to a
  nullhomologous knot $K$, called knot Floer
  homology~\cite{OS04:Knots,Rasmussen03:Knots}, detects the genus of
  $K$~\cite{OS04:ThurstonNorm}.  
\item The group $\HFa(Y)$ detects whether and how $Y$ fibers over $S^1$;
  similarly, $\HFKa(Y,K)$ detects whether $K$ is
  fibered~\cite{Ghiggini08:FiberedGenusOne,Ni09:FiberedMfld}.
\item The two previous properties are reminiscent of the Alexander
  polynomial, which gives partial information in each case. There is a
  precise relationship between $\HFKa$ and the Alexander
  polynomial. Specifically, if $K$ is a knot in $S^3$,
  then $\HFKa(K)$ is endowed with an integral bigrading 
  $\HFKa(K)=\bigoplus_{d,s\in\ZZ}
  \HFKa_{d}(K,s)$, and 
  \[
  \sum_{d}(-1)^d T^s\rank\HFKa_d(K,s)=\Delta_K(T)
  \]
  \cite{OS04:Knots,Rasmussen03:Knots}.
\item\label{item:SW} The Heegaard Floer homology groups of closed
  $3$-manifolds are now known to agree with the Seiberg-Witten Floer
  homology
  groups \cite{Taubes10:SW-ECH-I,Taubes10:SW-ECH-II,Taubes10:SW-ECH-III,Taubes10:SW-ECH-IV,Taubes10:SW-ECH-V,KutluhanLeeTaubes:HFHMI,KutluhanLeeTaubes:HFHMII,KutluhanLeeTaubes:HFHMIII,KutluhanLeeTaubes:HFHMIV,KutluhanLeeTaubes:HFHMV,ColinGhigginiHonda11:HF-ECH-1,ColinGhigginiHonda11:HF-ECH-2,ColinGhigginiHonda11:HF-ECH-3}. Moreover, one can use Heegaard Floer
  homology to define an invariant of smooth, closed
  $4$-manifolds \cite{OS06:HolDiskFour},
  with similar properties to the Seiberg-Witten
  invariant \cite{OS04:symplectic,Roberts08:blow-downs,JabukaMark08:product};
  it is expected that the two invariants agree.  Note, however, that
  to capture the analogue of the
  Seiberg-Witten invariant one needs to work with the $\HF^+$ and
  $\HF^-$ variants of Heegaard Floer homology.
\end{enumerate}

As mentioned above, Heegaard Floer homology is defined using
Lagrangian-intersection Floer homology, i.e., by counting holomorphic
curves. Consequently, it is in general hard to compute---though there
are now several algorithms for doing so; see
particularly~\cite{SarkarWang07:ComputingHFhat,MOS06:CombinatorialDescrip,MOST07:CombinatorialLink,MOT:grid,ManolescuOzsvath:surgery}. With
the goal of computing and better understanding Heegaard Floer homology in mind, we have been developing bordered Heegaard Floer
homology, a tool for understanding the behavior of the Heegaard Floer
homology group $\HFa(Y)$ under cutting and gluing of $Y$ along
surfaces. Roughly, bordered Floer homology is a $(2+1+1)$-dimensional
field theory. That is, roughly, it assigns to each connected, oriented
surface $F$ a differential graded algebra $\Alg(F)$ and to a cobordism
$Y$ from $F_1$ to $F_2$ an $(\Alg(F_1),\Alg(F_2))$-bimodule
$\CFDAa(Y)$. Composition of cobordisms corresponds to tensor product
of bimodules.

More precisely, like in Heegaard Floer homology, in bordered Floer
homology, the invariants are not associated directly to the
topological objects of interest---manifolds of dimensions $2$ through
$4$---but rather to certain combinatorial representations for these
objects, which we describe next.

The combinatorial representations of oriented surfaces which appear in
bordered Floer homology, the {\em pointed matched circles}, which we
denote by $\PMC$, consist essentially of a handle-decomposition of the
surface. (See Definition~\ref{def:PMC} below for a more precise formulation.) We
will let $F(\PMC)$ denote the surface underlying $\PMC$.
Bordered Floer homology associates to such a pointed matched
circle a differential-graded ($\dg$) algebra $\Alg(\PMC)$; the definition of $\Alg(\PMC)$ is
purely combinatorial.

The three-dimensional objects studied in the bordered theory are
cobordisms, i.e., three-manifolds with parameterized boundary.  More
precisely, a \emph{bordered $3$-manifold} consists of a compact,
oriented $3$-manifold-with-boundary $Y$ and a homeomorphism $\phi\co
F(\PMC)\to \bdy Y$, where $\PMC$ is some pointed matched circle.

Bordered Floer homology associates to a bordered $3$-manifold
$(Y,\phi\co F(\PMC)\to\bdy Y)$ a left $\dg$
$\Alg(-\PMC)$-module, which we denote $\CFDa(Y)$. (The minus sign in
front of $\PMC$ denotes a reversal of orientation.) Explicitly,
$\CFDa(Y)$ is a left module over the $\dg$ algebra
$\Alg(-\PMC)$; and $\CFDa(Y)$ is equipped with a differential which
satisfies the Leibniz rule \footnote{The ground ring for bordered Floer homology is
  $\Zmod{2}$; hence the signs usually appearing in the differential
  graded Leibniz rule become irrelevant.}
 with respect to the action by the algebra;
\[ \bdy_{\CFDa(Y)} (a\cdot x) = d_{\Alg(-\PMC)}(a)\cdot x +
a\cdot \bdy_{\CFDa(Y)}(x). \]

Like the algebras, the modules $\CFDa$ are also associated to
combinatorial representations of the underlying structure. In this
case, the combinatorial structure is called a {\em bordered
  Heegaard diagram} (Definition~\ref{def:bordered-HD} below). Unlike
the algebras, the definition of $\CFDa$ then depends on further analytic choices
(specifically, a family of complex structures on the underlying Heegaard
surface); but the quasi-isomorphism type of the module does not depend
on these further choices.

The modules $\CFDa$ can be used to reconstruct the Heegaard Floer
homology $\HFa$ via {\em pairing theorems}, which come in several
variants. For example, recall that if $M_1$ and $M_2$ are two $\dg$-modules
over some algebra $\Alg$, we can consider their chain complex of
morphisms
$\Mor_{\Alg}(M_1,M_2)$, which is to be thought of as the space of
$\Alg$-linear maps $\phi\co M_1\to M_2$, equipped with a
differential
\[d_{\Mor}(\phi)=d_{M_2}\circ \phi + \phi\circ d_{M_1}.\]

\begin{theorem}
  \label{thm:SimpleMorCFDaPairing}
  Let $Y_1$ and $Y_2$ be two $\PMC$-bordered three-manifolds.
  Then there is an isomorphism between the homology of the morphism
  space
  $\Mor_{\Alg(-\PMC)}(\CFDa(Y_1),\CFDa(Y_2))$ and the Heegaard Floer
  homology $\HFa(Y)$ of the three-manifold $Y=-Y_1\cup_{F(\PMC)}Y_2$
  obtained by gluing
  $-Y_1$ and $Y_2$ along their common boundary $F(\PMC)$ (according to the
  identifications specified by their borderings).
\end{theorem}

(This was not the original formulation of the pairing theorem; rather
it is a re-formulation appearing first in~\cite{AurouxBordered}; see
also~\cite{LOTHomPair}.)

The discussion above naturally raises the following questions:
\begin{enumerate}
\item\label{item:Q-inv-surf} To what extent is the algebra of a
  pointed matched surface an invariant of the underlying surface?
\item\label{item:Q-depend-param} In what way does the bordered
  invariant $\CFDa(Y)$ depend on the parameterization of the boundary
  of $Y$?
\end{enumerate}
Perhaps not too surprisingly, the answers to both of these questions are
governed by certain bimodules.  

Given a homeomorphism $\psi\co F(-\PMC_1)\to F(-\PMC_2)$, there is
an $\Alg(\PMC_1)\Hyph\Alg(\PMC_2)$-bimodule $\CFDDa(\psi)$ which allows one to change the framing of a
bordered three-manifold. There is a mild technical point which
becomes important when discussing these bimodules: as we will see,
$F(\PMC)$ contains a distinguished disk, and 
the homeomorphism~$\psi$
is required to fix this disk pointwise.

We can now state the dependence of the modules on the parameterization
in terms of these bimodules.
To state the dependence, recall that if $\Alg_1$
and $\Alg_2$ are two $\dg$ algebras, $B$ is an
$\Alg_1\Hyph\Alg_2$-bimodule and $M$ is a $\dg$ $\Alg_1$-module, 
then the space
$\Mor_{\Alg_1}(B,M)$ is naturally a left \dg $\Alg_2$-module.

\begin{theorem}
  \label{thm:Bimodules}
  If $(Y,\phi\co F(-\PMC_2)\to \bdy Y)$ is a bordered
  three-manifold and $\psi\co F(-\PMC_1)\to F(-\PMC_2)$ is a
  homeomorphism then there is a quasi-isomorphism:
  \[
  \CFDa(Y,\phi\circ\psi)
  \simeq 
  \Mor_{\Alg(\PMC_1)}(\CFDDa(\psi),\CFDa\left(Y,\phi)\right).
  \]
\end{theorem}

Theorem~\ref{thm:Bimodules} can be thought of as a
kind of pairing theorem, as well. The bimodule $\CFDDa(\psi)$
appearing above is the invariant associated to a very simple
bordered three-manifold with two boundary components: the underlying
three-manifold here is the product of an interval with the surface
$F(\PMC_2)$.  It is best to think of this as the special case of a
more general construction, involving bordered three-manifolds with two
boundary components. It turns out that these three-manifolds need to be equipped
with some additional structure, giving the {\em arced cobordisms} of
Definition~\ref{def:ArcedCobordism}
below. Theorem~\ref{thm:Bimodules} then becomes a special case of a
pairing theorem for gluing bordered three-manifolds to arced
cobordisms (Theorem~\ref{thm:bimod-mod-hom}, below); see Example~\ref{ex:Reparameterize}.

Theorem~\ref{thm:Bimodules} answers
Question~(\ref{item:Q-depend-param}) above.  The bimodules associated
to mapping classes also answer Question~(\ref{item:Q-inv-surf}): while
$\Alg(\PMC)$ is not an invariant of $F(\PMC)$, the (equivalence class
of the) derived category of modules over $\Alg(\PMC)$ is an invariant of
(the homeomorphism type of) $F(\PMC)$. For more details,
see~\cite[Theorem~\ref*{LOT2:thm:AlgebraDependsOnSurface}]{LOT2}.

Arguably more excitingly, Theorems~\ref{thm:SimpleMorCFDaPairing}
and~\ref{thm:Bimodules} are an effective tool for computing Heegaard
Floer homology.  They can be used to give an algorithm for computing
$\HFa(Y)$ for an arbitrary closed, oriented three-manifold
$Y$~\cite{LOT4}; the map $\hat{F}_W$ associated to any smooth
cobordism $W$~\cite{LOTCobordisms}; and the spectral
sequence~\cite{BrDCov} from Khovanov homology to $\HFa$ of the
branched double cover~\cite{LOT:DCov1,LOT:DCov2}. (We sketch the
algorithm for computing $\HFa(Y)$ in Lecture~\ref{lec:compute-HFa}.)
In a different direction, the torus boundary case of bordered Floer
homology has been particularly useful for practical computations; see
Lecture~\ref{lec:torus}.

Bordered Floer homology also associates another kind of module,
denoted $\CFAa(Y)$, to a bordered $3$-manifold $(Y,\phi\co F(\PMC)\to
\bdy Y)$. The module $\CFAa(Y)$ is a right $\Ainf$-module over
$\Alg(\PMC)$. To avoid digressing into $\Ainf$-algebra, we have
suppressed $\CFAa(Y)$, and will continue to do so throughout these
notes to the extent possible. (Another drawback of $\CFAa(Y)$ is that
its definition requires counting more holomorphic curves than
$\CFDa(Y)$, making $\CFAa(Y)$ typically harder to compute.)  There is
one place that $\CFAa(Y)$ seems unavoidable: in the proof of the
pairing theorem, which we sketch in Section~\ref{sec:prove-pairing}.

These notes are organized into five lectures. The first of these
focuses primarily on the combinatorial representations for manifolds
(pointed matched circles and Heegaard diagrams for bordered and arced
three-manifolds) which are used in the definitions of the modules. 
After a sufficient amount of the background is laid out, we give
a second, more detailed overview of the theory during the middle of the first
lecture. Finally, Lecture~\ref{lec:1} concludes by defining the
algebra $\Alg(\PMC)$ associated to a pointed matched circle $\PMC$.

The second lecture is devoted to defining the module $\CFDa(Y)$
associated to a bordered $3$-manifold $Y$, as well as its
generalization $\CFDDa(Y)$ to an arced cobordism. That lecture starts
by reviewing both the original definition and the cylindrical
reformulation of the invariant $\HFa(Y)$ for a closed $3$-manifold. 
The lecture
then turns to $\CFDa(Y)$ and the moduli spaces used to define it,
proves the surgery exact triangle for $\HFa$ (originally proved
in~\cite{OS04:HolDiskProperties}) and concludes by briefly defining
the extension $\CFDDa(Y)$.

In the third lecture, we describe the analysis which
underpins the theory. This allows us to sketch the proof that
the differential on $\CFDa$ is, in fact, a differential. It also
allows us to sketch a proof of the pairing theorem; in the process,
the invariant $\CFAa(Y)$, elsewhere absent from these notes, arises naturally.

The last two lectures are computational. The fourth lecture is devoted
to the torus-boundary case. After recalling some terminology about
knot Floer homology, it explains how one can recover the knot Floer
homology group $\HFKa(Y,K)$ from the bordered Floer homology of
$Y\setminus K$; indeed, this process also allows one to obtain, with a little
more work, the knot Floer homology of any satellite of $K$. The
lecture then discusses the other direction: for a knot $K$ in $S^3$,
one can recover the bordered Floer homology $S^3\setminus K$ from the
knot Floer complex $\CFK^-(K)$. Combining these results, one obtains a
theorem about the behavior of knot Floer homology under taking
satellites.

Finally, the last lecture describes an algorithm coming from bordered
Floer homology for computing $\HFa(Y)$ for closed three-manifolds
$Y$. 

There are a number of important aspects of the theory which are
missing from these notes. These include:
\begin{itemize}
\item Any discussion of the grading on bordered Floer homology. The
  grading takes a somewhat complicated form---the algebras are graded by
  a non-commutative group $G(\PMC)$ and the modules by
  $G(\PMC)$-sets---and we refer the reader to \cite[Chapter 10]{LOT1} for this
  part of the story.
\item A more thorough treatment of $\CFAa$. This would involve a
  lengthy algebraic digression which might distract from the
  underlying geometry in the theory.  Again, we refer the reader
  to~\cite{LOT1} to fill in this omission.
\item A discussion of  the proof of invariance of the bordered modules
  (Theorem~\ref{thm:invariance}). Most of the ideas in the proof of
  invariance, however, are present in the proof that $\bdy^2=0$ on
  $\CFDa$ and the proof of invariance in the closed
  case~\cite{OS04:HolomorphicDisks}.
\item A proof of the $\Mor$ versions of the pairing theorem
  (Theorems~\ref{thm:or-rev} and~\ref{thm:bimod-mod-hom}). We refer
  the reader to~\cite{LOTHomPair} for these proofs.
\item The connection between bordered Floer homology and
Juh{\'a}sz's {\em sutured Floer homology}~\cite{Juhasz06:Sutured}. This
connection
is given by Zarev's {\em bordered sutured theory}~\cite{Zarev09:BorSut}.
\end{itemize}

There are two other expository articles on bordered Heegaard Floer
homology, with somewhat different focuses, in which the reader might
be interested:~\cite{LOT0,LOT11:tour}. The paper~\cite{LOT13:faith} is
also intended to be partly expository.

\addtocontents{toc}{\protect\setcounter{tocdepth}{0}}
\section*{Acknowledgments}
\addtocontents{toc}{\protect\setcounter{tocdepth}{1}}

These are notes from a series of lectures by the first author during
the CaST conference in Budapest in the summer of 2012. We
thank Jennifer Hom for many helpful comments on and corrections to
earlier drafts of these notes, and the participants in the CaST summer
school for many further corrections. We also thank the
R{\'e}nyi Institute for providing a stimulating environment without
which these notes would not have been written, and the Simons Center
for providing a peaceful one, without which these notes would not have
been revised.

RL was supported by NSF Grant number DMS-0905796 and a Sloan
Research Fellowship. PSO was supported by NSF grant number
DMS-0804121. DPT was supported by NSF grant number DMS-1008049.

\chapter[Formal structure]{Combinatorial \texorpdfstring{$3$}{3}-manifolds with boundary. Formal structure of
  bordered Floer homology. The algebra associated to a surface.}\label{lec:1}

Much of this lecture lays out in detail the combinatorial
representations of the topological objects used in the definition of
bordered Floer homology. We start with surfaces (encoded by pointed
matched circles), and then move on to bordered three-manifolds (encoded by
Heegaard diagrams). With this material in place, we give a more
detailed overview of the formal structure of bordered Floer homology.
The lecture concludes with the definition of the algebra associated to
a pointed matched circle.

\section{Arc Diagrams and Surfaces}

\begin{definition}
  \label{def:PMC}
  A \emph{pointed matched circle} consists of an oriented circle $Z$, a point
  $z\in Z$, a finite set of points $\CircPts\subset Z$ disjoint from
  $z$, and a fixed-point free involution $M\co \CircPts\to
  \CircPts$. The map $M$ matches the points $\CircPts$ in pairs; that
  is, we can view $\CircPts$ as a union of $S^0\!$'s. We require that
  the result $Z'$ of doing surgery on $(Z,\CircPts)$ according to $M$
  be connected.  See Figure~\ref{fig:arc-diag}.
\end{definition}

\begin{figure}
  \centering
  \includegraphics[scale=.85714]{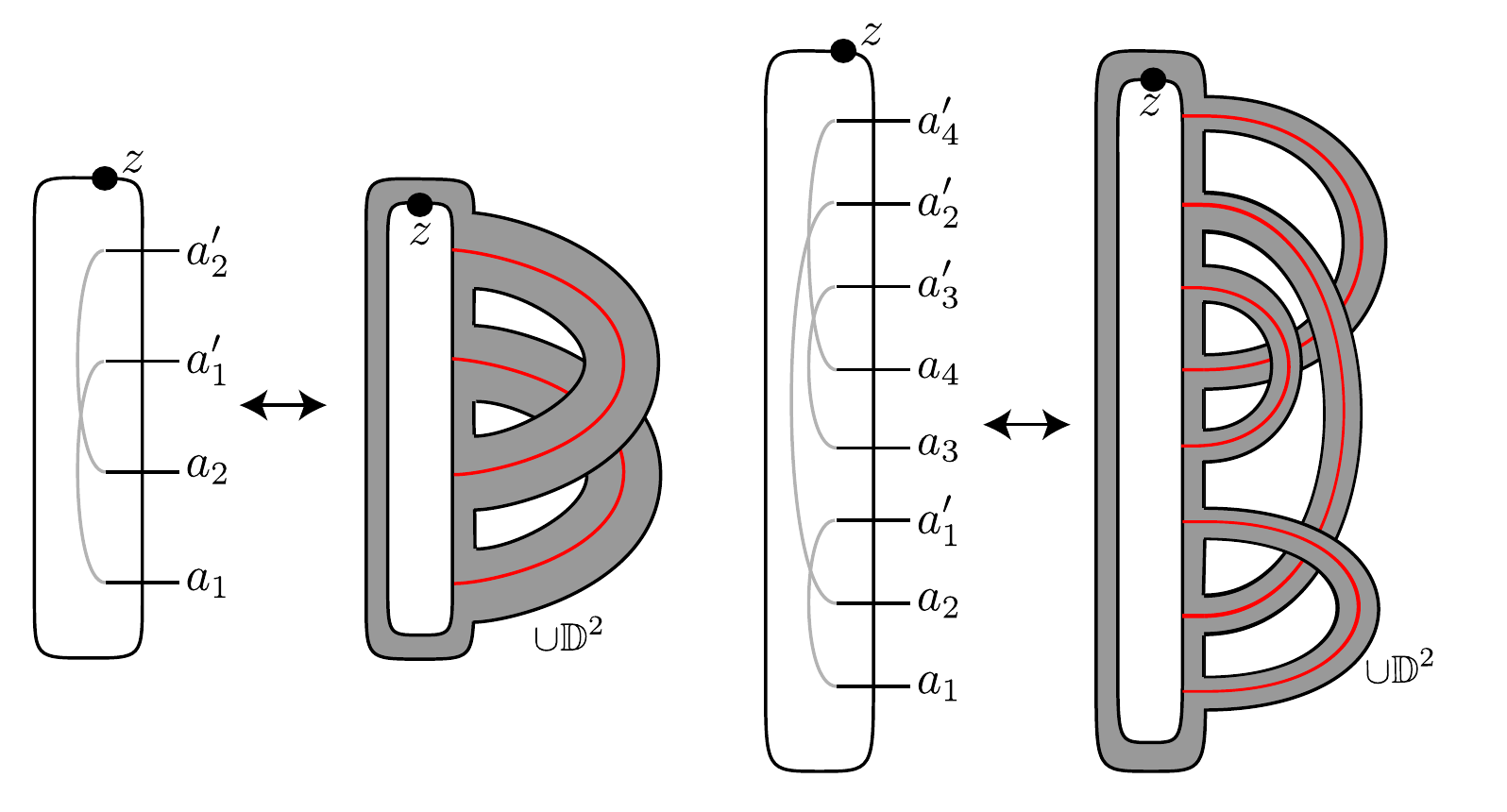}
  \caption{\textbf{Pointed matched circles and surfaces.} Left: a
    pointed matched circle specifying a once-punctured torus. Right: a
    pointed matched circle specifying a once-punctured genus $2$
    surface. In both cases, the involution $M$ exchanges $a_i$ and
    $a'_i$.}
  \label{fig:arc-diag}
\end{figure}

A pointed matched circle specifies a surface. There are a few essentially
equivalent constructions; here is one:
\begin{construction}
  Fix a pointed matched circle $\PMC=(Z,\CircPts,M,z)$. Build an
  oriented surface-with-boundary $\PunctF(\PMC)$ as follows. Start
  with $[0,1]\times Z$. Attach a strip ($2$-dimensional $1$-handle) to
  each pair of matched points in $\CircPts\times \{0\}$. The result
  has boundary $(Z\times \{1\})\amalg Z'$. Fill in $Z'$ with a copy of $\bD^2$. The result is $\PunctF(\PMC)$. Again,
  see Figure~\ref{fig:arc-diag}.
\end{construction}
As a slight variant, we could fill in the boundary of
$\PunctF(\PMC)$ with a disk. This gives a surface $F(\PMC)$ with a
distinguished disk in it---the disk $F(\PMC)\setminus
\PunctF(\PMC)$---and a distinguished basepoint on the boundary of this
disk. That is, $F(\PMC)$ is a \emph{strongly based surface}. 
(Papers in the subject  sometimes treat a pointed matched
circle as specifying a surface with boundary, and sometimes as
specifying a closed, strongly based surface; it makes no essential
difference.)

\begin{remark}
  Pointed matched circles are a special case of Zarev's \emph{arc
    diagrams}; any orientable surface with non-empty boundary can be
  represented by an arc diagram, and there is an associated algebra
  similar to the one we will describe in Section~\ref{sec:pmc}. Arc
  diagrams are, in turn, closely related to fat graphs and chord diagrams.
\end{remark}
\section{Bordered Heegaard diagrams for \texorpdfstring{$3$}{3}-manifolds}
We  start with $3$-manifolds with one boundary component:
\begin{definition}
  A \emph{bordered $3$-manifold} consists of a compact, oriented
  $3$-manifold-with-boundary $Y$ and a homeomorphism $\phi\co
  F(\PMC)\to \bdy Y$ for some pointed matched circle $\PMC$.

  Call two bordered $3$-manifolds $(Y_1,\phi_1\co F(\PMC)\to\bdy Y_1)$
  and $(Y_2,\phi_2\co F(\PMC)\to\bdy Y_2)$ \emph{equivalent} if there
  is a homeomorphism $\psi\co Y_1\to Y_2$ so that
  $\phi_2=\psi\circ\phi_1$, i.e., 
  \[
  \xymatrix{
    Y_1 \ar[rr]^\psi_\cong & &Y_2\\
    & F(\PMC)\ar[ul]^{\phi_1}\ar[ur]_{\phi_2}
  }
  \]
  commutes.
\end{definition}
We often drop the parametrization $\phi$ from the notation,
writing $Y$ to denote a bordered $3$-manifold, i.e., $Y=(Y,\phi)$.

We can represent bordered $3$-manifolds combinatorially, as follows:
\begin{definition}\label{def:bordered-HD}
  Let $\PMC$ be a pointed matched circle representing a surface of
  genus $k$.  A \emph{bordered Heegaard diagram with boundary $\PMC$}
  is a tuple
  \[
  \HD=(\Sigma_g,\overbrace{\overbrace{\alpha_1^a,\dots,\alpha_{2k}^a}^{\alphas^a},\overbrace{\alpha_1^c,\dots,\alpha_{g-k}^c}^{\alphas^c}}^{\alphas},\overbrace{\beta_1,\dots,\beta_g}^{\betas},z)
  \]
  where
  \begin{itemize}
  \item $\Sigma_g$ is a compact, oriented surface of genus $g$ with one boundary component.
  \item $\betas$ is a $g$-tuple of pairwise disjoint circles in the
    interior of ${\Sigma}$.
  \item $\alphas^c$ is a $(g-k)$-tuple of pairwise disjoint circles in the
    interior of ${\Sigma}$.
  \item $\alphas^a$ is a $(2k)$-tuple of pairwise disjoint arcs in
    $\Sigma$ with boundary in $\bdy\Sigma$.
  \item $z$ is a basepoint in $\bdy\Sigma\setminus\alphas^a$.
  \item $\alphas^a\cap\alphas^c=\emptyset$.
  \item $\Sigma\setminus(\alphas^c\cup\alphas^a)$ and $\Sigma\setminus\betas$ are both
    connected.
  \item $\PMC=(\bdy\Sigma, \alphas^a\cap\bdy\Sigma, M,z)$. Here, $M$
    matches (exchanges) the two endpoints of each $\alpha_i^a$.
  \end{itemize}
\end{definition}
Especially when we are considering holomorphic curves, we will abuse
notation and also use $\Sigma$ to denote $\Sigma\setminus\bdy\Sigma$;
and similarly for the $\alpha$-arcs.

\begin{construction}\label{const:HD-to-mfld}
  Let $\HD=(\Sigma,\alphas,\betas,z)$ be a bordered Heegaard
  diagram with boundary $\PMC$. There is a corresponding bordered $3$-manifold $Y(\HD)$
  constructed as follows.
  \begin{enumerate}
  \item Thicken ${\Sigma}$ to ${\Sigma}\times[0,1]$.
  \item Attach three-dimensional two-handles along the
    $\alpha$-circles in ${\Sigma}\times\{0\}$.
  \item Attach three-dimensional two-handles along the $\beta$-circles
    in ${\Sigma}\times\{1\}$.
  \end{enumerate}
  \begin{figure}
    \centering
    \includegraphics{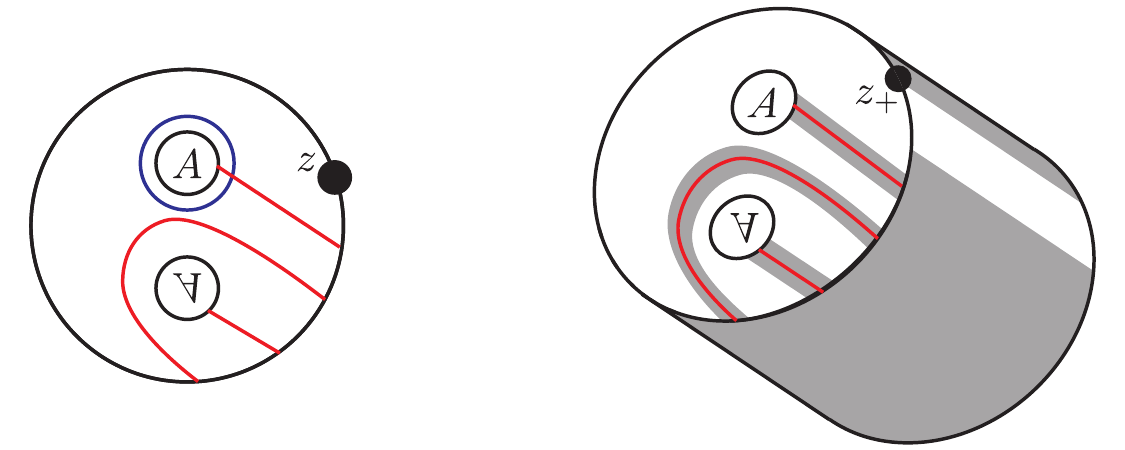}
    \caption{\textbf{A bordered Heegaard diagram and the
        associated $3$-manifold.} The picture on the left is a
      Heegaard diagram for the bordered solid torus shown on the
      right. (The labels $A$ indicate a handle between the
      corresponding circles.) The shaded part of the boundary is
      $\PunctF(\PMC)$. This figure is adapted from~\cite[Figure~\ref*{LOT2:fig:PBHD-3-mfld}]{LOT2}.}
    \label{fig:PBHD-3-mfld}
  \end{figure}
  A parameterization of the boundary is specified as follows.
  Consider the graph
  $$\left({\alphas}^a\cup \left(\bdy{\Sigma}\setminus\nbd(z)\right)\right)\times \{0\}\subset {\Sigma}\times\{0\},$$
  thought of as a subset of $\bdy Y$. The closure $\PunctF$ of a
  neighborhood of this graph is naturally identified with
  $\PunctF(\PMC)$. 
  The complement of $\PunctF$ in $\bdy Y$ is a disk, and is
  identified with $F(\PMC)\setminus\PunctF(\PMC)$. See
  Figure~\ref{fig:PBHD-3-mfld}.
\end{construction}
The orientations in Construction~\ref{const:HD-to-mfld} are confusing;
see~\cite[Construction~\ref*{LOT2:construct:bordered-HD-to-bordered-Y-one-bdy}]{LOT2} for a discussion of this point.

\begin{example}
  Figure~\ref{fig:PBHD-3-mfld} shows a Heegaard diagram for a solid
  torus. This is one of many Heegaard diagrams for bordered solid
  tori; see Section~\ref{sec:surgery} for more Heegaard diagrams for
  solid tori.
\end{example}

\begin{example}
  Figure~\ref{fig:GenusTwoBorderedDiagram}
  (page~\pageref{fig:GenusTwoBorderedDiagram}) shows a Heegaard diagram for a
  genus $2$ handlebody. Again, this is one among many.
\end{example}

\begin{example}\label{eg:bord-knot-compl}
  Fix an oriented surface $\Sigma$, equipped with a $g$-tuple of
  pairwise disjoint, homologically independent curves $\betas$ and a
  $(g-1)$-tuple of pairwise disjoint, homologically independent curves
  $\alphas^c=\{\alpha_1^c,\dots,\alpha_1^{g-1}\}$. Then
  $(\Sigma,\alphas^c,\betas)$ is a Heegaard diagram for a
  three-manifold with torus boundary, and indeed any such
  three-manifold~$Y$ can be described by a Heegaard diagram of this
  type. To turn such a diagram into a bordered Heegaard diagram, we
  proceed as follows.  Fix an additional pair of circles $\gamma_1$
  and $\gamma_2$ in $\Sigma$ so that:
  \begin{itemize}
  \item $\gamma_1$ and $\gamma_2$ are disjoint from
    $\alpha_1^c,\dots,\alpha_{g-1}^c$,
  \item $\gamma_1$ and $\gamma_2$ intersect, transversally, in a
    single point~$p$ and
  \item both of the homology classes
    $[\gamma_1]$ and $[\gamma_2]$ are homologically independent
    from $[\alpha_1^c],\dots,[\alpha_{g-1}^c]$.
  \end{itemize}
  Let $D$ be a disk around~$p$ which is disjoint from all the above
  curves, except for $\gamma_1$ and~$\gamma_2$, each of which it meets
  in a single arc. Then, the complement of $D$ specifies a bordered
  Heegaard diagram for $Y$, for some parametrization of $\bdy Y$.
  A bordered Heegaard diagram for the trefoil complement is
  illustrated in Figure~\ref{fig:TrefoilComplement}.
  \begin{figure}
    \centering
    \begin{overpic}[tics=10, scale=.85714]{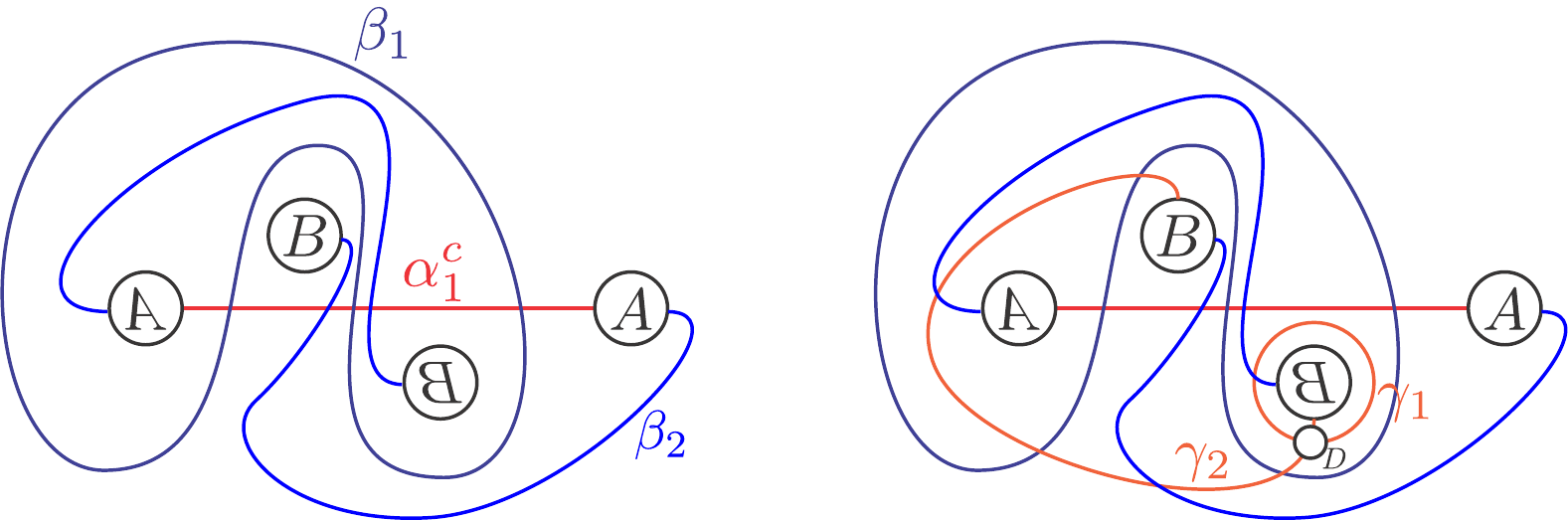}
      \put(94,2){\textcolor{blue}{$\beta_2$}}
    \end{overpic}
    \caption{\textbf{A bordered Heegaard diagram for the trefoil
        complement.} Left: a Heegaard diagram for the complement of
      the trefoil. The circles labeled $A$ (respectively $B$) denote a
      handle attached to the plane. Right: a bordered Heegaard
      diagram, obtained by adding the curves $\gamma_1$ and $\gamma_2$
      and deleting a disk. It may be instructive to compare this
      diagram with Figure~\ref{fig:trefoil}}
    \label{fig:TrefoilComplement}
  \end{figure}
  (This example is drawn
  from~\cite[Section~\ref*{LOT1:sec:Bestiary}]{LOT1}; see also the
  discussion around~\cite[Figure~\ref*{LOT1:fig:Twisting}]{LOT1}.)
\end{example}

We also consider $3$-dimensional cobordisms:
\begin{definition}\label{def:ArcedCobordism}
  Fix pointed matched circles 
  \begin{align*}
  \PMC_L&=(Z_L,\CircPts_L,M_L,z_L)\\
  \PMC_R&=(Z_R,\CircPts_R,M_R,z_R).
  \end{align*}
  An \emph{arced cobordism from $\PMC_L$ to $\PMC_R$} consists of:
  \begin{itemize}
  \item A compact, oriented $3$-manifold-with-boundary $Y$,
  \item an injection $\phi\co (-\PunctF(\PMC_L))\amalg
    \PunctF(\PMC_R)\to \bdy Y$ (where $-$ denotes orientation
    reversal) and
  \item a path $\gamma$ in $\bdy Y\setminus \Image(\phi)$
  \end{itemize}
  such that $Y\setminus (\Image(\phi)\cup \nbd(\gamma))$ is a disk.

  There is a natural notion of equivalence for arced cobordisms,
  similar to the notion of equivalence for bordered $3$-manifolds; we
  leave it as an exercise.
\end{definition}
As for bordered $3$-manifolds, we will typically denote all of the
data of an arced cobordism simply by $Y$. Also as with bordered
$3$-manifolds, there are several other essentially equivalent ways to
formulate the notion of an arced cobordism; see for
instance~\cite[Section~\ref*{LOT2:sec:Diagrams}]{LOT2}
and~\cite[Section~\ref*{HomPair:sec:ab}]{LOTHomPair}.

Again, a combinatorial representation of arced cobordisms will be
important to us:
\begin{definition}
  An \emph{arced Heegaard diagram} is a tuple
  \[
  \HD=(\Sigma_g,\overbrace{\overbrace{\alpha_1^{a,L},\dots,\alpha_{2k_L}^{a,L}}^{\alphas^{a,L}},\overbrace{\alpha_1^{a,R},\dots,\alpha_{2k_R}^{a,R}}^{\alphas^{a,R}},\overbrace{\alpha_1^c,\dots,\alpha_{g-k_L-k_R}^c}^{\alphas^c}}^{\alphas},\overbrace{\beta_1,\dots,\beta_g}^{\betas},\arcz)
  \]
  where
  \begin{itemize}
  \item ${\Sigma}_g$ is a
    compact, oriented surface of genus $g$ with two boundary components,
    $\bdy_L\Sigma$ and $\bdy_R\Sigma$;
  \item $\betas$ is a $g$-tuple of pairwise disjoint curves in the
    interior of ${\Sigma}$;
  \item $\alphas^{a,L}$ is a collection of pairwise-disjoint embedded
    arcs with boundary on $\bdy_L{\Sigma}$;
  \item $\alphas^{a,R}$ is a collection of pairwise-disjoint embedded arcs with
    boundary on $\bdy_R{\Sigma}$;
  \item $\alphas^c$ is a collection of pairwise-disjoint circles in
    the interior of ${\Sigma}$; and
  \item $\arcz$ is a path in
    ${\Sigma}\setminus(\alphas^{a,L}\cup\alphas^{a,R}\cup\alphas^c\cup\betas)$
    between $\bdy_L{\Sigma}$ and $\bdy_R{\Sigma}$.
  \end{itemize}
  These are required to satisfy:
  \begin{itemize}
  \item $\alphas^{a,L}$, $\alphas^{a,R}$ and $\alphas$ are all disjoint,
  \item ${\Sigma}\setminus{\alphas}$ and
    ${\Sigma}\setminus\betas$ are connected and
  \item ${\alphas}$ intersects $\betas$ transversely.
  \end{itemize}
  (Compare~\cite[Definition~\ref*{LOT2:def:ArcedBordered}]{LOT2}.)

  Observe that each boundary component of an arced Heegaard diagram is
  a pointed matched circle.
\end{definition}

\begin{figure}
  \centering
  \includegraphics[scale=.85714]{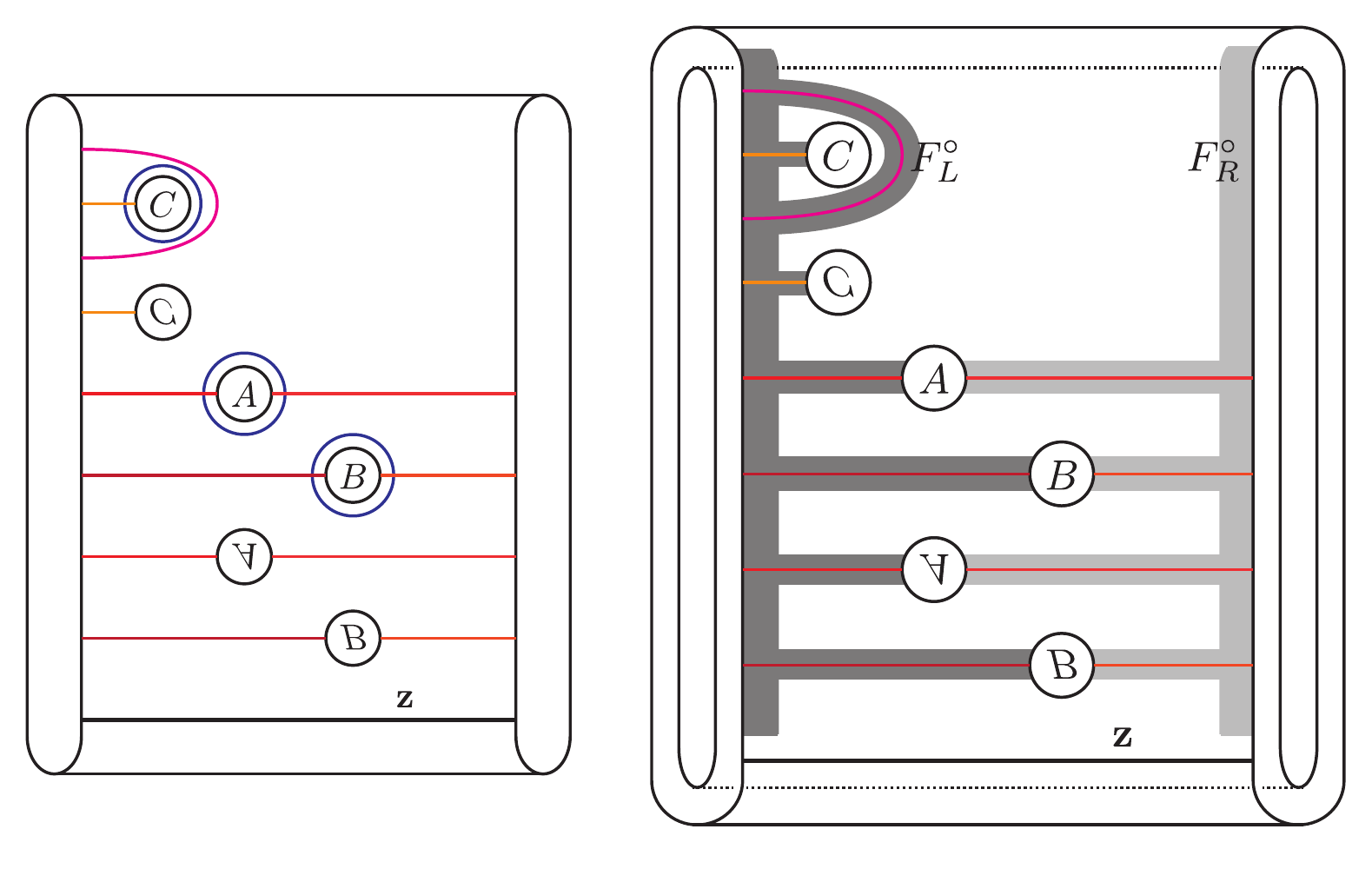}
  \caption{\textbf{Constructing a bordered 3-manifold with two
      boundary components from an arced bordered Heegaard diagram.}
    The Heegaard diagram on the left represents an elementary
    cobordism from the genus two surface to the genus one surface. On
    the right is a (somewhat schematic) depiction of the resulting
    $3$-manifold. The inside part of the boundary, which corresponds
    to $\Sigma\times\{1\}$, is a cylinder, since
    the
    $\beta$-circles caused the handles to be filled in. The outside
    part of the boundary, corresponding to $\Sigma\times\{0\}$, is a surface of genus $3$ with two boundary
    components. The region $F^\circ_L$ (respectively $F^\circ_R$) is
    darkly (respectively lightly) shaded.}
  \label{fig:2bdy-construct}
\end{figure}

\begin{construction}\label{const:arced-to-mfld}
  Fix an arced Heegaard diagram
  $\HD=(\Sigma_g,\alphas^{a,L},\alphas^{a,R},\alphas^c,\betas,\arcz)$
  with boundary $\PMC_L\amalg \PMC_R$. Build a $3$-manifold-with-boundary 
  $Y$ as follows:
  \begin{enumerate}
  \item Thicken ${\Sigma}$ to ${\Sigma}\times[0,1]$.
  \item Attach three-dimensional two-handles along the
    $\alpha$-circles in ${\Sigma}\times\{0\}$.
  \item Attach three-dimensional two-handles along the $\beta$-circles
    in ${\Sigma}\times\{1\}$.
  \end{enumerate}

  Consider the graphs
  \begin{align*}
    \Gamma_L&=\left({\alphas}^{a,L}\cup
      \left(\bdy_L{\Sigma}\setminus\nbd(\arcz)\right)\right)\times
    \{0\}\subset {\Sigma}\times\{0\}\\
    \Gamma_R&=\left({\alphas}^{a,R}\cup
      \left(\bdy_R{\Sigma}\setminus\nbd(\arcz)\right)\right)\times
    \{0\}\subset {\Sigma}\times\{0\}
  \end{align*}
  thought of as subsets of $\bdy Y$. The closure $\PunctF_L$
  (respectively $\PunctF_R$) of a neighborhood of $\Gamma_L$
  (respectively $\Gamma_R$) is naturally identified with
  $\PunctF(\PMC_L)$ (respectively $\PunctF(\PMC_R)$). Let $\phi$
  denote this identification $\PunctF(\PMC_L)\amalg \PunctF(\PMC_R)\to
  \PunctF_L\amalg \PunctF_R$. The arc $\gamma_\arcz=\arcz\times \{0\}$
  connects $\PunctF_L$ and $\PunctF_R$, and $\bdy Y\setminus
  (\PunctF_L\cup \PunctF_R\cup\nbd(\gamma_\arcz))$ is a disk. The data
  $(Y,\phi,\gamma_\arcz)$ is an arced cobordism; we call this
  cobordism the \emph{arced cobordism associated to $\HD$} and denote
  it by $Y(\HD)$. See Figure~\ref{fig:2bdy-construct}.
\end{construction}

\begin{example}\label{eg:mapping-cyl} Let $\psi\co F(\PMC_L)\to F(\PMC_R)$ be a
  homeomorphism taking the preferred disk to the preferred disk and
  the basepoint to the basepoint; that is, $\psi$ is a \emph{strongly
    based homeomorphism}. The \emph{mapping cylinder of $\psi$}, denoted
  $M_\psi$, is the arced cobordism from $\PMC_L$ to $\PMC_R$ given as
  follows. The underlying $3$-manifold is $[0,1]\times
  \PunctF(\PMC_R)$. The map $\phi\co -\PunctF(\PMC_L)\amalg
  \PunctF(\PMC_R)\to \bdy M_\psi$ is given by the identity map $\Id\co
  \PunctF(\PMC_R)\to \{1\}\times \PunctF(\PMC_R)$ and the map $\psi
  \co \PunctF(\PMC_L)\to \{0\}\times \PunctF(\PMC_R)$. The arc
  $\gamma$ is $[0,1]\times\{z\}$.

  Some examples of Heegaard diagrams for mapping cylinders are shown
  in Figure~\ref{fig:DehnTwistsGenusOne}.

  Gluing the mapping cylinder for $\psi$ to a bordered $3$-manifold
  $(Y,\phi)$ in the sense of Exercise~\ref{ex:glue-bord-to-cob} gives $(Y,\phi\circ\psi)$.
\end{example}

\begin{figure}
  \centering
  \includegraphics{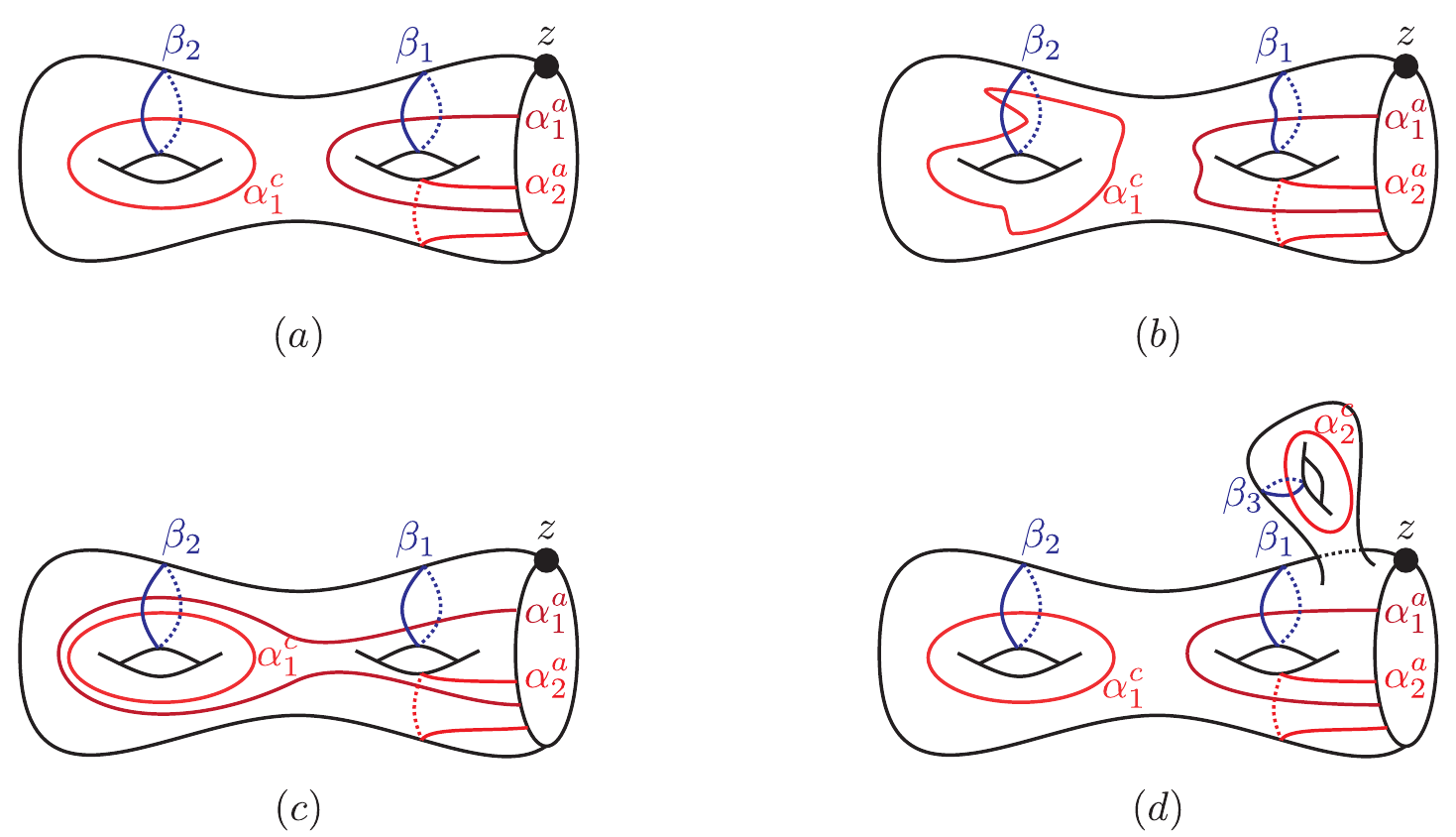}
  \caption{\textbf{Heegaard moves.} (a) A genus $2$ bordered Heegaard diagram for a solid
  torus. (b) The result of applying some isotopies to the $\alpha$-
  and $\beta$-curves. (c) The result of a handleslide of $\alpha_1^a$
  over $\alpha_1^c$. (d) The result of a stabilization.}
  \label{fig:heegaard-moves}
\end{figure}

As in the closed case, the key properties of bordered Heegaard
diagrams are that every bordered $3$-manifold can be represented by a
bordered Heegaard diagram, and any two such diagrams can be related by
certain elementary moves:

\begin{theorem}
  Let $(Y, \phi\co F(\PMC)\to \bdy Y)$ be a bordered
  $3$-manifold. Then $Y$ is represented by some bordered Heegaard
  diagram $\HD$. Similarly, let $(Y, \phi\co \PunctF(\PMC_L)\amalg
  \PunctF(\PMC_R)\to \bdy Y, \gamma)$ be an arced cobordisms. Then $Y$
  is represented by some arced Heegaard diagram $\HD$.
\end{theorem}
The case of bordered Heegaard diagrams
is~\cite[Lemma~\ref*{LOT1:lem:3mfld-heegaard}]{LOT1} while
the arced Heegaard diagram case is~\cite[Proposition~\ref*{LOT2:prop:diagrams-exist}]{LOT2}.

\begin{theorem} 
  Suppose that $\HD$ and $\HD'$ are bordered Heegaard diagrams
  representing equivalent bordered $3$-manifolds $Y(\HD)\cong
  Y(\HD')$. Then $\HD$ and $\HD'$ can be made diffeomorphic by a
  sequence of the following moves:
  \begin{itemize}
  \item Isotopies of the $\alpha$- and/or $\beta$-curves.
  \item Handleslides or $\alpha$-circles over $\alpha$-circles,
    $\alpha$-arcs over $\alpha$-circles, and $\beta$-circles over
    $\beta$-circles.
  \item Stabilizations and destabilizations of the diagram, i.e.,
    taking connected sums with the standard Heegaard diagram for $S^3$.
  \end{itemize}
  (See Figure~\ref{fig:heegaard-moves}.)

  An exactly analogous statement holds for arced Heegaard diagrams and
  arced cobordisms.
\end{theorem}
The case of bordered Heegaard diagrams is~\cite[Proposition~\ref*{LOT1:prop:heegaard-moves}]{LOT1} while
the arced Heegaard diagram case is~\cite[Proposition~\ref*{LOT2:prop:heegaard-moves}]{LOT2}.

\section{The structure of bordered Floer homology}
\subsection{The connected boundary case}
For simplicity, we begin with the connected boundary case. Bordered
Floer homology assigns:
\begin{center}
\begin{tabular}{l|l}
Pointed matched circle $\PMC$ & \dg algebra $\Alg(\PMC)$\\[3pt]
\hline
Bordered $3$-manifold $(Y,\phi\co F(\PMC)\to \bdy Y)$ & Right $\Ainf$
$\Alg(\PMC)$-module $\CFAa(Y)$\\
& Left \dg $\Alg(-\PMC)$-module $\CFDa(Y)$.
\end{tabular}
\end{center}

Actually, the modules $\CFAa(Y)$ and $\CFDa(Y)$ depend on a choice of
bordered Heegaard diagram $\HD$ for $Y$, as well as another auxiliary
choice---an almost-complex structure. However:
\begin{theorem}\label{thm:invariance}\cite[Theorems~\ref*{LOT1:intro:D-invariance} and~\ref*{LOT1:intro:A-invariance}]{LOT1}
  The quasi-isomorphism types of the modules $\CFAa(Y)$ and $\CFDa(Y)$
  depend only on the equivalence class of bordered $3$-manifold $Y$.
\end{theorem}

The utility of $\CFAa$ and $\CFDa$ comes from the fact that they
can be used to reconstruct the Heegaard Floer homology groups of
closed three-manifolds
$\HFa(Y)$, via what we call a \emph{pairing theorem}. Recall that
$\HFa(Y)$ is the homology of a chain complex $\CFa(Y)$.
\begin{theorem}\label{thm:pairing1}\cite[Theorem~\ref*{LOT1:thm:TensorPairing}]{LOT1} Suppose
  that $(Y_1,\phi_1\co F(\PMC)\to \bdy Y)$ and $(Y_2,\phi_2\co
  {-F(\PMC)}\to \bdy Y)$ are bordered $3$-manifolds with boundaries
  parameterized by $\PMC$ and $-\PMC$, respectively. Write
  $Y_1\cup_\bdy Y_2$ to mean $\bigl(Y_1\amalg Y_2)/(\phi_1(x)\sim
  \phi_2(x))$. Then
  \[
  \CFa(Y)\simeq \CFAa(Y_1)\DTP_{\Alg(F)}\CFDa(Y_2).
  \]
\end{theorem}
Here, $\DTP$ denotes the appropriate notion of tensor product given
that $\CFAa$ may be an $\Ainf$-module. In the case that $\CFAa$ is an
ordinary module, this reduces to the derived tensor product---which is
good, since $\CFAa$ is only well-defined up to quasi-isomorphism. But
this distinction is not so important: the module $\CFDa$ is
projective, so the derived and ordinary tensor products agree.

The modules $\CFAa(Y)$ and $\CFDa(Y)$ are defined using holomorphic
curves (though for certain kinds of diagrams the techniques
of~\cite{SarkarWang07:ComputingHFhat} can be used to compute them
combinatorially). By contrast, the algebras $\Alg(\PMC)$ are defined
combinatorially. A few further properties of the algebras:
\begin{itemize}
\item Each $\Alg(\PMC)$ is a finite-dimensional algebra over $\Field$.
\item The algebra $\Alg(\PMC)$ decomposes as a direct sum of subalgebras
  \[
  \Alg(\PMC)=\bigoplus_{i=-k}^k \Alg(\PMC,i).
  \]
  Here, $k$ is the genus of $F(\PMC)$.  The action of $\Alg(\PMC,i)$
  on $\CFAa(Y)$ and $\CFDa(Y)$ is trivial for $i\ne 0$, but the other
  summands come up for the cobordism invariants below.
\item The algebra $\Alg(\PMC,-k)$ is isomorphic to $\Field$ (with
  trivial differential). In particular, if $\PMC$ is the (unique)
  pointed matched circle for $S^2$ then $\Alg(\PMC)=\Field$. The
  algebra $\Alg(\PMC,k)$ is quasi-isomorphic to $\Field$.
\item If $\PMC$ is the unique pointed matched circle for the torus
  then $\Alg(\PMC,0)$ has no differential; in terms of generators and
  relations, $\Alg(\PMC,0)$ is given by
  \begin{equation}\label{eq:torus-alg}
  \xymatrix{
   \iota_0\bullet\ar@/^1pc/[r]^{\rho_1}\ar@/_1pc/[r]_{\rho_3} & \bullet\iota_1\ar[l]_{\rho_2}
  }/(\rho_2\rho_1=\rho_3\rho_2=0).
  \end{equation}
  This algebra is $8$-dimensional over $\Field$. It will appear
  frequently, so we name the rest of the elements in its standard
  basis: let $\rho_{12}=\rho_1\rho_2$, $\rho_{23}=\rho_2\rho_3$ and
  $\rho_{123}=\rho_1\rho_2\rho_3$.

  (Our notation for path algebras might be somewhat non-standard. The
  vertices $\iota_0$ and $\iota_1$ are, of course, idempotents. The
  arrow $\rho_1$ indicates that $\iota_0\rho_1\iota_1=\rho_1$.)
\end{itemize}

\subsection{Invariants of arced cobordisms}
To get a useful theory, we need to generalize to three-manifolds with
two boundary components. In fact, the invariants which come up in this
two-boundary-component case are associated to
three-manifolds equipped with some extra structure:
the arced cobordisms of Definition~\ref{def:ArcedCobordism}.

Suppose $Y$ is an arced cobordism from $\PMC_1$ to $\PMC_2$. Then
there are several kinds of bimodules we can associate to $Y$: we can
treat each boundary component of $Y$ in either a ``type $D$'' way or a
``type $A$'' way. (What this means will be clearer after
Lectures~\ref{lec:CFD} and~\ref{lec:analysis}.) This gives invariants
$\CFDDa(Y)$ (both boundaries viewed in a type $D$ way), $\CFDAa(Y)$
(one boundary, say $\PMC_1$, viewed in a type $D$ way and the other in
a type $A$ way), and $\CFAAa(Y)$ (both boundaries viewed in a type $A$
way).  The bimodule
$\CFDDa(Y)$ is an ordinary---indeed, bi-projective---\dg bimodule;
both of $\CFDAa(Y)$ and $\CFAAa(Y)$ are typically $\Ainf$-bimodules.

As with the modules associated to bordered $3$-manifolds, the bimodules
$\CFDDa(Y)$, $\CFDAa(Y)$ and $\CFAAa(Y)$ depend on the choices of
Heegaard diagrams and almost-complex structures. Again, up to
quasi-isomorphism they are invariants:
\begin{theorem}
  \label{thm:bimod-invariance}\cite[Theorem~\ref*{LOT2:thm:InvarianceOfBimodules}]{LOT2}
  The quasi-isomorphism types of the bimodules $\CFDDa(Y)$, $\CFDAa(Y)$ and
  $\CFAAa(Y)$ depend only on the equivalence class of arced cobordism
  $Y$.
\end{theorem}

By convention, we view $\CFDDa(Y)$ as having commuting left actions
by $\Alg(-\PMC_1)$ and $\Alg(-\PMC_2)$; $\CFDAa(Y)$ as having a left
action by $\Alg(-\PMC_1)$ and a right action by $\Alg(\PMC_2)$; and
$\CFAAa(Y)$ as having right actions by $\Alg(\PMC_1)$ and
$\Alg(\PMC_2)$. However, $\Alg(-\PMC)$ is the opposite algebra to
$\Alg(\PMC)$ (Exercise~\ref{ex:opposite}) so we can move actions from
one side to the other at the cost of introducing / deleting minus signs.
In the literature, we often find it convenient to decorate the
invariants with the algebras they are over, writing
\[
\lsup{\Alg(-\PMC_1),\Alg(-\PMC_2)}\CFDDa(Y) \qquad\qquad
\lsup{\Alg(-\PMC_1)}\CFDAa(Y)_{\Alg(\PMC_2)} \qquad\qquad \CFAAa(Y)_{\Alg(\PMC_1),\Alg(\PMC_1)}.
\]
The superscripts indicate that the module structure is projective, and
subscripts indicate the module structure may be $\Ainf$. This notation
leads to a kind of Einstein summation behavior for tensor products in
the pairing theorems:
\begin{theorem}\label{thm:bimod-mod-pair}\cite[Theorem~\ref*{LOT2:thm:GenReparameterization}]{LOT2}
  Let $Y_1$ be a bordered $3$-manifold with boundary $\PMC_1$ and
  $Y_2$ be an arced cobordism from $\PMC_1$ to $\PMC_2$. Let
  $Y_1\cup_{F(\PMC)} Y_2$ be the bordered $3$-manifold obtained by
  gluing $Y_1$ to $Y_2$ (Exercise~\ref{ex:glue-bord-to-cob}). Then
  there are quasi-isomorphisms
  \begin{align*}
    \CFAa(Y_1)\DTP_{\Alg(\PMC_1)}  \CFDAa(Y_{2}) & \simeq \CFAa(Y_{1}\cup_{F(\PMC_1)} Y_{2})\\
    \CFAAa(Y_{2}) \DTP_{\Alg(-\PMC_1)} \CFDa(Y_1)& \simeq \CFAa(Y_{1}\cup_{F(\PMC_1)} Y_{2}) \\
    \CFAa(Y_1)\DTP_{\Alg(\PMC_1)}  \CFDDa(Y_{2}) & \simeq \CFDa(Y_{1}\cup_{F(\PMC_1)} Y_{2}) \\
    \CFDAa(Y_{2}) \DTP_{\Alg(-\PMC_1)}  \CFDa(Y_1) & \simeq \CFDa(Y_{1}\cup_{F(\PMC_1)} Y_{2}).
  \end{align*}
\end{theorem}
\begin{theorem}\label{thm:bimod-bimod-pair}\cite[Theorem~\ref*{LOT2:thm:GenComposition}]{LOT2}
  Let $Y_{1}$ be an arced cobordism from $\PMC_1$ to $\PMC_2$ and
  $Y_{2}$ an arced cobordism from $\PMC_2$ and $\PMC_3$. Let
  $Y_1\cup_{F(\PMC_2)} Y_2$ be the result of gluing $Y_1$ to $Y_2$
  along $F(\PMC_2)$ (Exercise~\ref{ex:glue-bord-to-cob}). Then there
  are quasi-isomorphisms of bimodules:
  \begin{align*}
    \CFDAa(Y_{1})\DTP_{\Alg(\PMC_2)} \CFDAa(Y_{2}) & \simeq 
    \CFDAa(Y_{1}\cup_{F(\PMC_2)} Y_{2}) \\
    \CFAAa(Y_{1})\DTP_{\Alg(\PMC_2)} \CFDAa(Y_{2}) & \simeq 
    \CFAAa(Y_{1}\cup_{F(\PMC_2)} Y_{2}) \\
    \CFDAa(Y_{1})\DTP_{\Alg(\PMC_2)} \CFDDa(Y_{2}) & \simeq 
    \CFDDa(Y_{1}\cup_{F(\PMC_2)}
    Y_{2}) \\
    \CFAAa(Y_{1})\DTP_{\Alg(\PMC_2)} \CFDDa(Y_{2}) & \simeq 
    \CFDAa(Y_{1}\cup_{F(\PMC_2)} Y_{2})
  \end{align*}  
\end{theorem}
The compact way of stating Theorems~\ref{thm:bimod-mod-pair}
and~\ref{thm:bimod-bimod-pair} is that if you tensor type $A$
boundaries with type $D$ boundaries then you get what you expect.

\subsection{Pairing theorems without \texorpdfstring{$A$}{A} modules}
To avoid a long detour into $\Ainf$ formalism, in most of these
lectures we will avoid $\CFAa$. (The exception will be the discussion
of the pairing theorem in 
Lecture~\ref{lec:analysis}.) So, it will be useful to have versions of
the pairing
theorems---Theorems~\ref{thm:pairing1},~\ref{thm:bimod-mod-pair}
and~\ref{thm:bimod-bimod-pair}---making use only of type $D$
modules. We can accomplish this using certain dualities of bordered
Floer invariants:

\begin{theorem}\cite[Theorem~\ref*{HomPair:thm:or-rev}]{LOTHomPair}\label{thm:or-rev}
  Let $Y$ be a bordered $3$-manifold with boundary $F(\PMC)$. Let $-Y$
  denote $Y$ with its orientation reversed, which has boundary
  $F(-\PMC)$. Then there are quasi-isomorphisms:
  \begin{align}
  \Mor_{\Alg(-\PMC)}(\CFDa(Y),\Alg(-\PMC))&\simeq \CFAa(-Y)
  \label{eq:ReverseTypeD}\\
  \Mor_{\Alg(\PMC)}(\CFAa(Y),\Alg(\PMC))&\simeq \CFDa(-Y).
  \label{eq:ReverseTypeA}
  \end{align}
\end{theorem}
In Formula~\eqref{eq:ReverseTypeD}, $\Mor$ denotes the chain complex
of module homomorphisms from $\CFDa(Y)$ to $\Alg(-\PMC)$, with
differential given by 
\[
\bdy(f) = f\circ \bdy_{\CFDa(Y)}+d_{\Alg(-\PMC)}\circ f.
\]
So, for instance, the cycles in the $\Mor$ complex are the \dg module
homomorphisms, i.e., chain maps which respect the module structure.
In Formula~\eqref{eq:ReverseTypeA}, $\Mor$ denotes the chain complex
of $\Ainf$-morphisms.

\tolerance=400
\begin{corollary}\label{cor:mod-mod-hom}\cite[Theorem~\ref*{HomPair:thm:hom-pair}]{LOTHomPair}
  Suppose that $Y_1$ and $Y_2$ are bordered $3$-manifolds with
  boundary $F(\PMC)$. Then 
  \begin{align*}
    \CFa(-Y_1\cup_{F(\PMC)}Y_2) &\simeq
    \Mor_{\Alg(-\PMC)}(\CFDa(Y_1),\CFDa(Y_2))\\
    &\simeq
    \Mor_{\Alg(\PMC)}(\CFAa(Y_1),\CFAa(Y_2))\\
    \shortintertext{so} \HFa(-Y_1\cup_{F(\PMC)}Y_2) &\simeq
    \Ext_{\Alg(-\PMC)}(\CFDa(Y_1),\CFDa(Y_2))\\
    &\simeq
    \Ext_{\Alg(\PMC)}(\CFAa(Y_1),\CFAa(Y_2)).
  \end{align*}
\end{corollary}
\tolerance=200

For bimodules the situation is somewhat more subtle: there are a few
natural notions of ``dual'', and some versions introduce boundary Dehn
twists in the bimodules. The following result will be more than
sufficient for these lectures:
\begin{theorem}\label{thm:bimod-mod-hom}\cite[Corollary~\ref*{HomPair:cor:bimod-mod-hom-pair}]{LOTHomPair}
  If $Y_1$ is a bordered $3$-manifold with boundary $F(\PMC_1)$ and
  $Y_2$ is an arced cobordism from $-\PMC_1$ to $-\PMC_2$ then
  \begin{equation}
    \label{eq:MorDDtoD}
    \begin{split}
      \CFAa(Y_1\cup_{F(\PMC_1)}(-Y_2))&\simeq \Mor_{\Alg(-\PMC_1)}(\CFDDa(Y_2),\CFDa(Y_1))\\
      \CFDa(-Y_1\cup_{F(\PMC_1)}Y_2)&\simeq
      \Mor_{\Alg(-\PMC_1)}(\CFDa(Y_1),\CFDDa(Y_2)).
    \end{split}
  \end{equation}
\end{theorem}

\begin{example}
  \label{ex:Reparameterize}
  The bimodules $\CFDDa(\psi)$ discussed in the introduction are
  defined to be $\CFDDa(M_\psi)$ associated to the mapping cylinder of
  $\psi$ (Example~\ref{eg:mapping-cyl}). So,
  Theorem~\ref{thm:Bimodules} from the introduction is a special case
  of Theorem~\ref{thm:bimod-mod-hom}.
\end{example}

For further results like these, including some involving boundary Dehn
twists, see the introduction to~\cite{LOTHomPair}.

\section{The algebra associated to a pointed matched circle}\label{sec:pmc-alg}
We will define the algebras associated to pointed matched circles in
three steps. We start with a warm-up in Section~\ref{sec:perm}, discussing 
the group ring of the symmetric group $S_n$ and a deformation of it called
the nilCoxeter algebra. In Section~\ref{sec:Ank} we define a family of
algebras $\Alg(n,k)$ ($n,k\in\NN$), which are a kind of directed,
distributed version of the nilCoxeter algebra. The algebra
$\Alg(\PMC)$ associated to a pointed matched circle for a surface of
genus $k$ is defined as a subalgebra of
$\bigoplus_{i=0}^{2k}\Alg(2k,i)$; the definition is given in
Section~\ref{sec:pmc}. (It is also possible to give a more direct
definition of $\Alg(\PMC)$; see, for instance,~\cite[Section~\ref*{HFa:subsec:AlgPMC}]{LOT4}.)
\subsection{A graphical representation of permutations}\label{sec:perm}
Consider the symmetric group $S_n$ on $\underline{n}=\{1,\dots,n\}$. We can
represent elements of $S_n$ graphically as homotopy classes of maps 
\[\textstyle
\bigl(\coprod_{i=1}^n [0,1], \coprod_{i=1}^n\{0\}, \coprod_{i=1}^n\{1\} \bigr)\stackrel{\phi}{\longrightarrow}
\bigl([0,1]\times [0,n], \coprod_{i=1}^n\{0\}\times\underline{n},
\coprod_{i=1}^n\{1\}\times \underline{n}\bigr)
\]
such that the restrictions $\phi|_{\coprod_{i=1}^n\{0\}}$ and
$\phi|_{\coprod_{i=1}^n\{1\}}$ are injective. For example, the
permutation $
\left(\begin{smallmatrix}
  1 & 2 & 3 & 4 & 5\\
  3 & 1 & 2 & 5 & 4
\end{smallmatrix}\right)\in S_5$ is represented by the diagram
\begin{equation}\label{eq:perm-graph}
\mathcenter{\includegraphics[height=.75in]{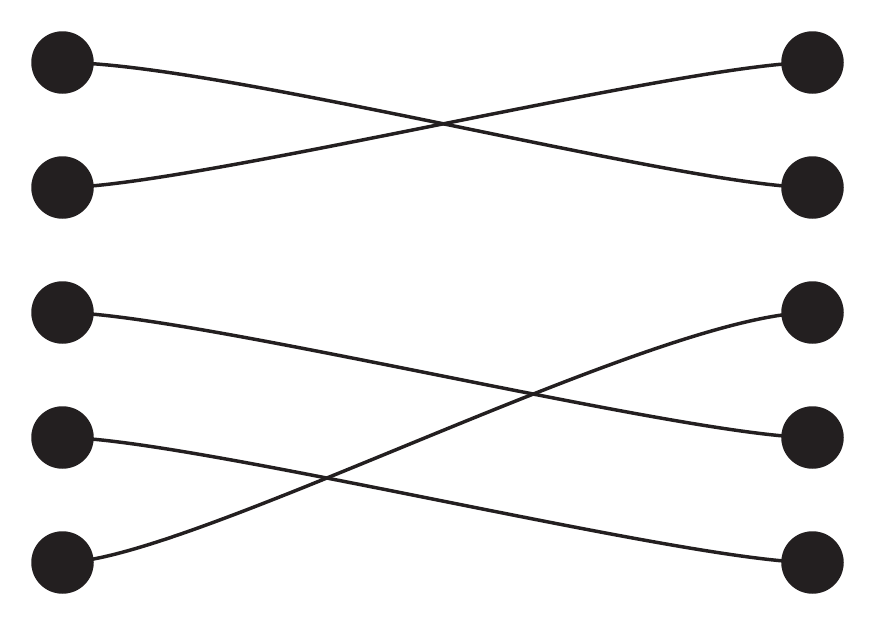}}
\end{equation}

In the graphical notation, multiplication corresponds to
juxtaposition. So, the group ring $\ZZ[S_n]$ of $S_n$ is given by formal linear
combinations of diagrams as in \eqref{eq:perm-graph}, with
product given by juxtaposition.  Moreover, notice that essential
crossings in diagrams like Formula~\eqref{eq:perm-graph} correspond to
\emph{inversions}, i.e., pairs $i,j\in\{1,\dots,n\}$ such that $i<j$
but $\sigma(j)<\sigma(i)$.

In $\ZZ[S_n]$, double-crossings can be undone via Reidemeister II-like
moves:
\begin{equation}\label{eq:double-cross-1}
\mathcenter{\includegraphics[height=.35in]{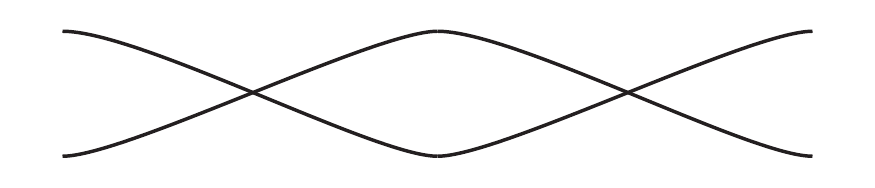}}=
\mathcenter{\includegraphics[height=.35in]{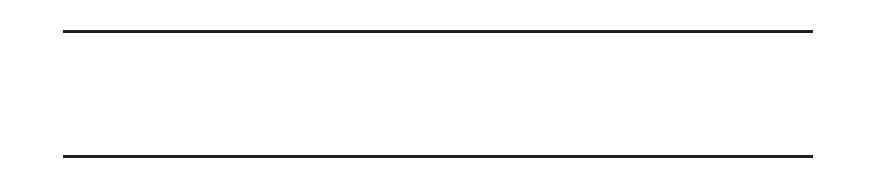}}
\end{equation}
If we replace this relation by the relation that double-crossings are $0$,
\begin{equation}\label{eq:double-cross-0}
\mathcenter{\includegraphics[height=.35in]{DoubleCross1}}=0
\end{equation}
we arrive at another algebra, the \emph{nilCoxeter algebra} $\nilCox_n$;
see, for instance~\cite{Khovanov01:Nilcoxeter}. Note that even though
$\nilCox_n\not\cong \ZZ[S_n]$, $S_n$ still gives a basis for
$\nilCox_n$. Let
$\Inv(\sigma)$ denote the set of inversions of $\sigma$.
An equivalent formulation is that we define 
\[
\sigma\cdot_{\nilCox}
\tau=
\begin{cases}
  \tau\circ\sigma & \text{if
  }\#\Inv(\tau\circ\sigma)=\#\Inv(\sigma)+\#\Inv(\tau)\\
0 & \text{else.}
\end{cases}
\]

If we work over $\Field$, as is our tendency, we can define a
differential on $\nilCox_n$ by declaring that $d(\sigma)$ is the sum of
all ways of smoothing a crossing in $\sigma$. More formally, let
$\tau_{i,j}$ denote the transposition exchanging $i$ and $j$. Then
define
\begin{equation}\label{eq:d-on-nilCox}
d(\sigma)=\sum_{\substack{(i,j)\in\Inv(\sigma)\\ \#\Inv(\tau_{i,j}\sigma)=\#\Inv(\sigma)-1}}\tau_{i,j}\circ\sigma.
\end{equation}
It is straightforward to verify that this makes $\nilCox_n$ into a
differential algebra.  (If we want to define this differential with
signs, we need an odd version of the nilCoxeter algebra;
see~\cite{Khovanov10:gl12}.)

\subsection{The algebra \texorpdfstring{$\Alg(n,k)$}{A(n,k)}}\label{sec:Ank}
Now, instead of permutations of $\{1,\dots,n\}$, consider
\emph{partial permutations}, i.e., triples $(S,T,\sigma)$ where
$S,T\subset\underline{n}$ and $\sigma\co S\to T$ is a bijection.  Call
a partial permutation $(S,T,\sigma)$ \emph{upward-veering} if
$\sigma(i)\geq i$ for all $i\in S$. Let $\Alg(n)$ denote the
$\Field$-vector space generated by all upward-veering partial
permutations. Define a product on $\Alg(n)$ by 
\begin{equation}\label{eq:mul-on-Ank}
(S,T,\phi)\cdot(U,V,\psi)=
\begin{cases}
  0 & \text{if }T\neq U\\
  0 & \text{if } \#\Inv(\psi\circ\phi)\neq \#\Inv(\psi)+\#\Inv(\phi)\\
  (S,V,\psi\circ\phi) & \text{otherwise.}
\end{cases}
\end{equation}

Define a differential on $\Alg(n)$ by setting
\[
d(S,T,\phi)=\sum_{\substack{(i,j)\in\Inv(\phi)\\
    \#\Inv(\tau_{i,j}\circ \phi)=\#\Inv(\phi)-1}}(S,T,\tau_{i,j}\circ\phi).
\]

Graphically, we can still represent generators of $\Alg(n)$ as strand
diagrams; for example, in $n=5$, we draw the partial permutation
$(\{1,2,3\}, \{3,4,5\}, (1\mapsto 5, 2\mapsto 4, 3\mapsto 3))$ as
\[
\begin{overpic}[tics=10,height=1in]{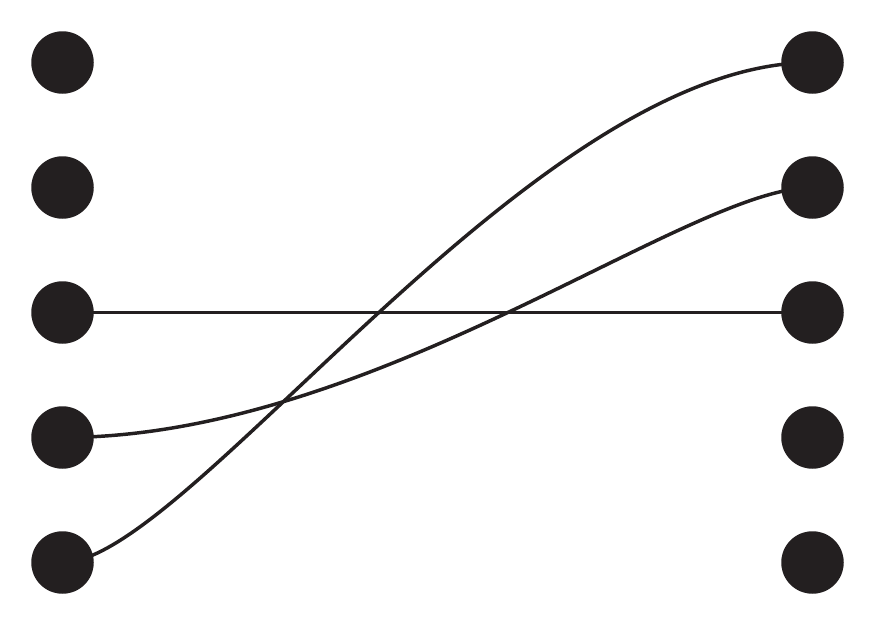}
\put(-5,4){1}
\put(-5,18){2}
\put(-5,32){3}
\put(-5,46){4}
\put(-5,60){5}
\put(100,4){1}
\put(100,18){2}
\put(100,32){3}
\put(100,46){4}
\put(100,60){5}
\end{overpic}
\]
Multiplication is $0$ if the endpoints do not match up (the first
condition in Equation~\eqref{eq:mul-on-Ank}) or if the concatenation
contains a double crossing (the second condition in
Equation~\eqref{eq:mul-on-Ank}); otherwise, the product is just the
concatenation. The differential is gotten by summing over all ways of
smoothing one crossing, and then throwing away any diagrams involving
double crossings.

\begin{proposition}\label{prop:Alg-n-is-dg-alg}\cite[Lemma~\ref*{LOT1:lem:Ank-is-dga}]{LOT1}
  These operations make $\Alg(n)$ into a differential algebra.
\end{proposition}
Proposition~\ref{prop:Alg-n-is-dg-alg} is not especially difficult,
though keeping track of the double-crossing condition adds some
complication. The reader is invited to prove it as an extra exercise.

Notice that $\Alg(n)$ decomposes as a direct sum
\begin{equation}\label{eq:Alg-n-decomp}
\Alg(n)=\bigoplus_{l=0}^{n}\Alg(n,i)
\end{equation}
where $\Alg(n,i)$ is generated
by partial permutations $(S,T,\phi)$ with $|S|=|T|=i$.

The algebra $\Alg(n)$ has an obvious grading by the number of
crossings. This grading does not, however, descend in a nice way to
the subalgebras associated to pointed matched circles.

\subsection{The algebra associated to a pointed matched circle}\label{sec:pmc}
Fix a pointed matched circle $\PMC=(Z,\CircPts,M,z)$ for a surface of
genus $k$, so $|\CircPts|=4k$. The basepoint $z$ and orientation of
$Z$ identify $\CircPts$ with $\underline{4k}=\{1,\dots,4k\}$. The
algebra $\Alg(\PMC)$ is a subalgebra of $\Alg(4k)$.

Call a generator $(S,T,\phi)$ of $\Alg(4k)$ \emph{$M$-admissible} if
$S\cap M(S)=T\cap M(T)=\emptyset$. (This terminology is not used
elsewhere in the literature.) Write $\Fix(\phi)=\{i\in S\mid
\phi(i)=i\}$. Suppose that $\phi$ is $M$-admissible. Then, given
$U\subset\Fix(\phi)$ we can define a new element $(S\setminus U\cup M(U),
T\setminus U\cup M(U), \phi_U) \in\Alg(n)$ by replacing the horizontal
strands at $U$ by horizontal strands at $M(U)$. That is, $\phi_U$ is
characterized by $\phi_U|_{S\setminus U}=\phi|_{S\setminus U}$ and
$\phi_U|_{M(U)}=\Id$. Given an $M$-admissible $(S,T,\phi)$ define
\[
a(S,T,\phi)=\sum_{U\subset \Fix(\phi)}(S\setminus U\cup M(U),
T\setminus U\cup M(U), \phi_U).
\]
For example,
\begin{center}
  \begin{overpic}[tics=10,height=2in]{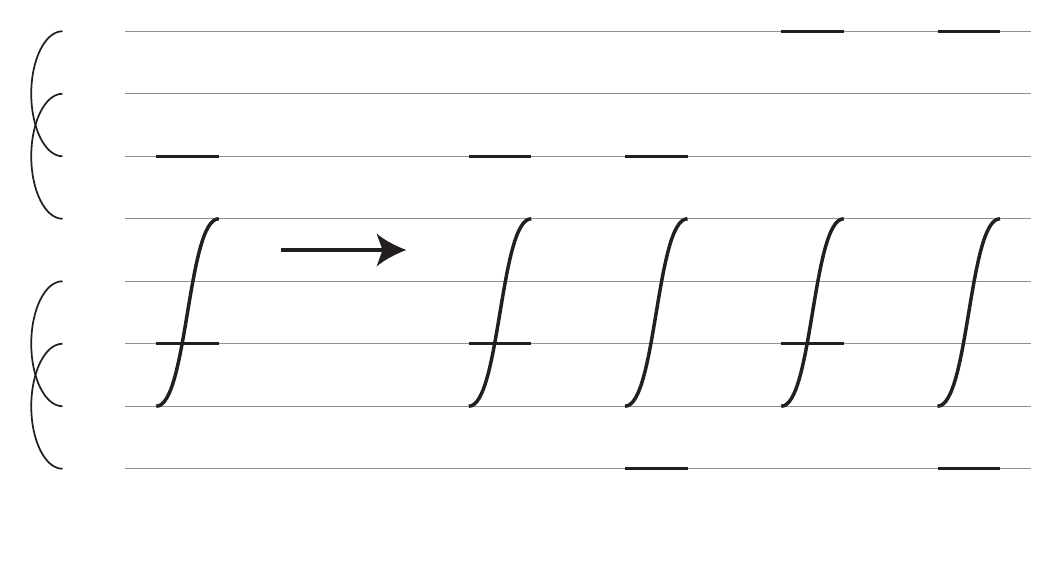}
\put(8,8){1}
\put(8,14){2}
\put(8,19){3}
\put(8,25){4}
\put(8,31){5}
\put(8,37){6}
\put(8,42){7}
\put(8,49){8}
\put(10,3){$(S,T,\phi)$}
\put(30,31){$a$}
\put(35 ,3){$U=$}
\put(45,3){$\emptyset$}
\put(52,29){$+$}
\put(58,3){$\{3\}$}
\put(68,29){$+$}
\put(73,3){$\{6\}$}
\put(82,29){$+$}
\put(86,3){$\{3,6\}$}
\end{overpic}
\end{center}

Now, $\Alg(\PMC)$ is defined to be the subalgebra of $\Alg(4k)$
generated by $a(S,T,\phi)$ for $M$-admissible generators $(S,T,\phi)$.

The decomposition of $\Alg(n)$ from Formula~\eqref{eq:Alg-n-decomp}
gives a decomposition of $\Alg(\PMC)$. It is convenient to change the
indexing slightly: let $\Alg(\PMC,i)=\Alg(\PMC)\cap \Alg(4k,k+i)$, so
$\Alg(\PMC)=\bigoplus_{i=-k}^k \Alg(\PMC,i)$.

\section{Exercises}

\begin{exercise}
  Let $Y$ be a closed $3$-manifold. How do you go from a pointed
  Heegaard diagram for $Y$ to a bordered Heegaard diagram for
  $Y\setminus \bD^3$? Vice-versa? (Hint: both directions are easy.)
\end{exercise}

\begin{exercise}\label{ex:glue-bord-to-cob}
  Let $Y_1$ be a bordered $3$-manifold with boundary $\PMC_1$ and
  $Y_2$ an arced cobordism from $\PMC_1$ to $\PMC_2$. There is a
  natural way to glue $Y_1$ and $Y_2$ to get a bordered $3$-manifold
  with boundary $\PMC_2$; how? 

  Similarly, if $Y_1$ is an arced cobordism from $\PMC_1$ to $\PMC_2$
  and $Y_2$ is an arced cobordism from $\PMC_2$ to $\PMC_3$ then there
  is a natural way to glue $Y_1$ to $Y_2$ to obtain an arced cobordism
  from $\PMC_1$ to $\PMC_3$; how?

  (Both parts are a little awkward with our
  definition of arced cobordism; the definitions in~\cite{LOT2}
  and~\cite{LOTHomPair} make them more obvious.)
\end{exercise}
\begin{exercise}
        Let $\HD$ be a bordered Heegaard diagram with no $\alpha$ circles.
        What is the underlying three-manifold $Y({\HD})$? 
\end{exercise}
\begin{exercise}
  Formulate precisely the notion of equivalence for arced cobordisms.
\end{exercise}
\begin{exercise}
  The bordered Heegaard diagram in Figure~\ref{fig:TrefoilComplement}
  represents the trefoil complement with some particular
  framing. Which one (as an element of $\ZZ$)?
\end{exercise}
\begin{exercise}
        Draw a bordered Heegaard diagram for the $0$-framed complement of the figure eight knot.
\end{exercise}
\begin{exercise}
  Verify that the differential given in Formula~\eqref{eq:d-on-nilCox}
  makes the nilCoxeter algebra into a differential algebra, i.e.,
  that it satisfies $d^2=0$ and the Leibniz rule.
\end{exercise}
\begin{exercise}
  Give an example of an element $(S,T,\phi)\in\Alg(n)$ and a pair
  $(i,j) \in \Inv(\phi)$ so that $(S,T,\tau_{i,j}\circ \phi)$ is not
  in $d(S,T,\phi)$.
\end{exercise}
\begin{exercise}\label{ex:path-for-T2}
  Verify the path algebra description in Equation~\ref{eq:torus-alg}
  for the algebra $\Alg(T^2,0)$.
\end{exercise}
\begin{exercise}\label{ex:idempotents}
  Prove: There is a one-to-one correspondence between indecomposable
  idempotents in $\Alg(\PMC)$ and subsets of the set of matched pairs
  of $\PMC$, i.e., subsets of $\CircPts/M$. (An idempotent $I$ is
  called \emph{indecomposable} if for any idempotent $J$, either $I\cdot
  J=I$ or $I\cdot (1-J)=I$.) (Hint: this should be easy.)
\end{exercise}

\begin{figure}
  \centering
  \includegraphics{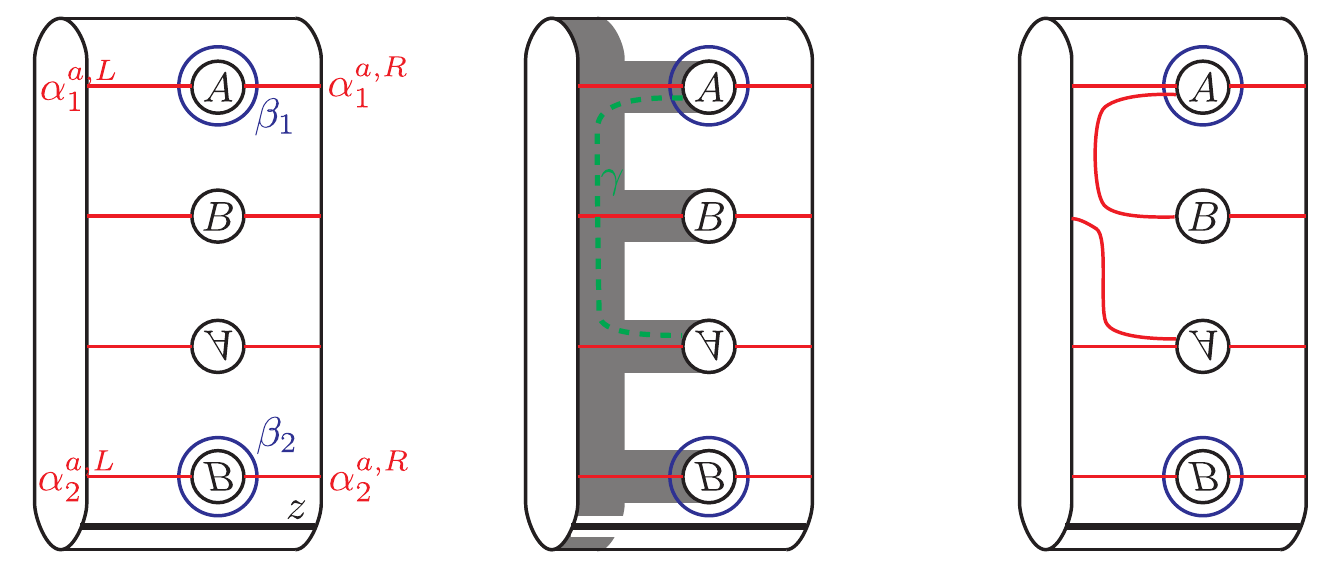}
  \caption{\textbf{Building Heegaard diagrams for mapping cylinders.}
    Left: a Heegaard diagram for the identity map of the
    torus. Center: the sub-surface $\PunctF(\PMC_L)$ and a dashed
    curve $\gamma$ on $\PunctF(\PMC_L)$. Right: a Heegaard diagram for
    a Dehn twist around $\gamma$. This figure is adapted
    from~\cite[Figure~\ref*{LOT2:fig:Build-HD}]{LOT2}.}
  \label{fig:HD-for-cyl}
\end{figure}

\begin{exercise}\label{ex:build-cyl-hd}
  In this exercise we explain how to produce arced Heegaard diagrams
  for mapping cylinders.  This algorithm is explained in somewhat more
  detail in~\cite[Section~\ref*{LOT2:sec:DiagramsForAutomorphisms}]{LOT2}.
  \begin{enumerate}
  \item Show that the arced Heegaard diagram on the left of
    Figure~\ref{fig:HD-for-cyl} represents the mapping cylinder of the
    identity map (of the pointed matched circle for a
    torus). Generalize this to give a diagram for the identity map of
    any pointed matched circle. (See Figure~\ref{fig:Genus2Identity}
    for the standard arced Heegaard diagram for the identity map of
    another pointed matched circle.)
  \item Let $\phi\co F(\PMC_L)\to F(\PMC_R)$ be a strongly based
    homeomorphism. Recall from Construction~\ref{const:arced-to-mfld}
    that a neighborhood $\PunctF_L$ of the graph $\Gamma_L$ is
    identified with $\PunctF(\PMC_L)$. Start with the identity
    Heegaard diagram for $F(\PMC_L)$, and apply the homeomorphism
    $\phi$ to the $\alpha_i^{a,L}\subset \PunctF_L$. (See
    Figure~\ref{fig:HD-for-cyl} for an example.) Prove: the result is
    an arced Heegaard diagram for $\phi$.
  \end{enumerate}
\end{exercise}

\begin{exercise}
  There is a unique pointed matched circle representing the
  once-punctured torus.
  \begin{enumerate}
  \item List several different pointed matched circles representing
    the once-punctured genus $2$ surface.
  \item Show that the set of matched circles representing the
    once-punctured genus $k$ surface is in bijection with the set of gluing
    patterns for the $4k$-gon giving the genus $k$ surface.
  \end{enumerate}
\end{exercise}

\begin{exercise}\label{ex:opposite}
  Prove that
  $\Alg(-\PMC)$ is the opposite algebra to $\Alg(\PMC)$.
\end{exercise}

\begin{exercise}
  Let $\PMC$ be the \emph{split pointed matched circle} for a surface
  of genus $k$, as illustrated in Figure~\ref{fig:split-matching}
  (page~\pageref{fig:split-matching}). Give a path algebra description
  of $\Alg(\PMC,-k+1)$, similar to Formula~\eqref{eq:torus-alg}.

  Similarly, let $\PMC$ be the \emph{antipodal pointed matched circle}
  for a surface of genus $k$, i.e., the pointed matched circle in
  which $a_i$ is matched to $a_{i+2k}$ ($i=1,\dots,2k$). Give a path
  algebra description of $\Alg(\PMC,-k+1)$, similar to
  Formula~\eqref{eq:torus-alg}. (For a solution to this part,
  see~\cite[Example~\ref*{faith:eg:antipodal}]{LOT13:faith}.)
\end{exercise}

\chapter{Modules associated to bordered \texorpdfstring{$3$}{3}-manifolds}\label{lec:CFD}
\section{Brief review of the cylindrical setting for Heegaard Floer
  homology}
\subsection{A quick review of the original formulation of Heegaard Floer homology}
We start by recalling the definition of Heegaard Floer homology in
the closed setting~\cite{OS04:HolomorphicDisks}, as well as a
``cylindrical'' reformulation of the
definition~\cite{Lipshitz06:CylindricalHF}; this reformulation will be
useful for defining the bordered Floer invariants.

Fix a pointed Heegaard diagram $\HD=(\Sigma,\alphas,\betas,z)$ (in the
sense of~\cite{OS04:HolomorphicDisks}) for a closed $3$-manifold
$Y$. Associated to $\HD$ are various Heegaard Floer homology groups;
as noted in the previous lecture, bordered Floer homology (so far) relates to the
technically simplest of these, $\HFa(Y)$. The group $\HFa(Y)$ is
defined as follows. Suppose $\Sigma$ has genus $g$. Choosing a complex
structure $j_\Sigma$ on $\Sigma$ makes the symmetric product 
\[
\Sym^g(\Sigma)=\overbrace{\Sigma\times\dots\times \Sigma}^{g\text{ copies}}/S_g
\]
into a smooth---in fact, K\"ahler---manifold.
(This is not obvious.)
Writing
$\alphas=\{\alpha_1,\dots,\alpha_g\}$ and
$\betas=\{\beta_1,\dots,\beta_g\}$, the tori
$\alpha_1\times\dots\times
\alpha_g,\beta_1\times\dots\times\beta_g\subset \Sigma^{\times g}$
project to embedded tori $T_\alpha$ and $T_\beta$ in
$\Sym^g(\Sigma)$. Each of $T_\alpha$ and $T_\beta$ is totally real; in
fact, it was shown in~\cite{Perutz07:handleslides} that for an
appropriate choice of K\"ahler form the tori $T_\alpha$ and $T_\beta$
are Lagrangian. Then, $\HFa(Y)$ is the Lagrangian Floer homology of
$(T_\alpha,T_\beta)$ inside $\Sym^g(\Sigma\setminus\{z\})$.

In a little more detail,  $\HFa(Y)$ is the homology of
a chain complex $(\CFa(Y),\bdy)$. $\CFa(Y)$ is the free
$\Field$-vector space generated by $T_\alpha\cap T_\beta$. The
differential $\bdy \co \CFa(Y)\to \CFa(Y)$ is defined by counting
holomorphic disks of the following kind. Given $\x,\y\in T_\alpha\cap T_\beta$ we consider
the space of maps $\bD^2\to \Sym^g(\Sigma\setminus\{z\})$ such that:
\begin{itemize}
\item $-i$ maps to $\x$.
\item $+i$ maps to $\y$.
\item $\{p\in \bdy\bD^2\mid \Re(p)>0\}$ maps to $T_\alpha$.
\item $\{p\in \bdy\bD^2\mid \Re(p)<0\}$ maps to $T_\beta$.
\end{itemize}
\begin{figure}
  \centering
  \[
\begin{overpic}[tics=10,height=2in]{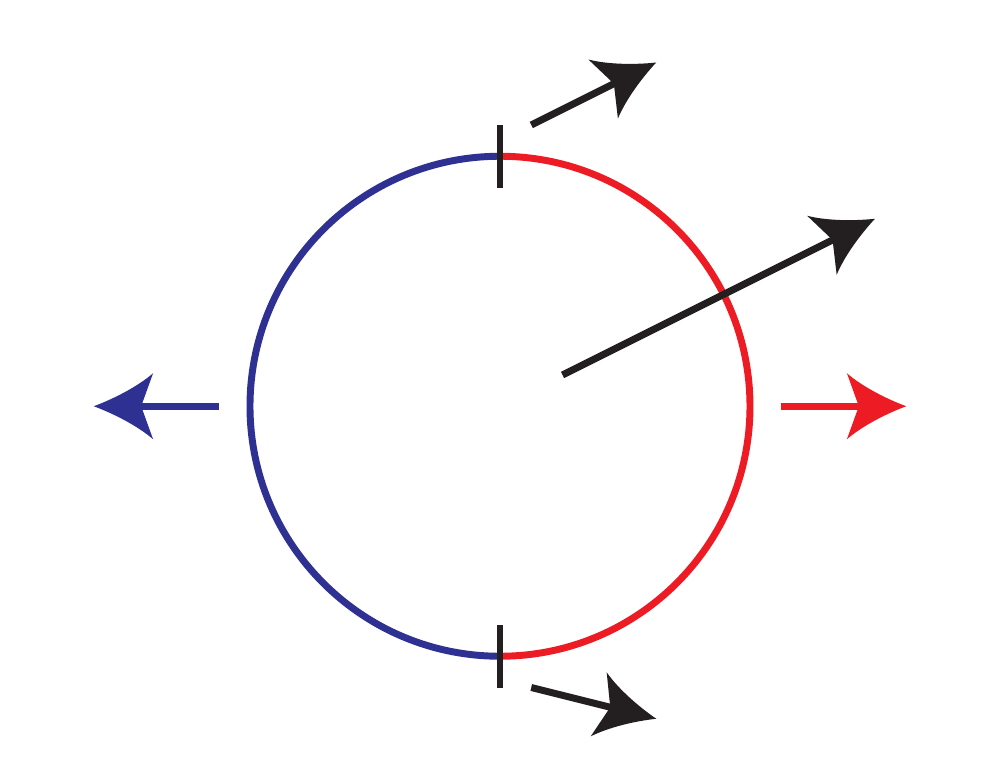}
\put(66,5){$\x$}
\put(66,70){$\y$}
\put(91,36){$\textcolor{red}{T_\alpha}$}
\put(1,36){$\textcolor{blue}{T_\beta}$}
\put(89,54){$\Sym^g(\Sigma)$}
\end{overpic}
\]
  \caption{\textbf{Boundary conditions for Whitney disks.}}
  \label{fig:Whitney}
\end{figure}
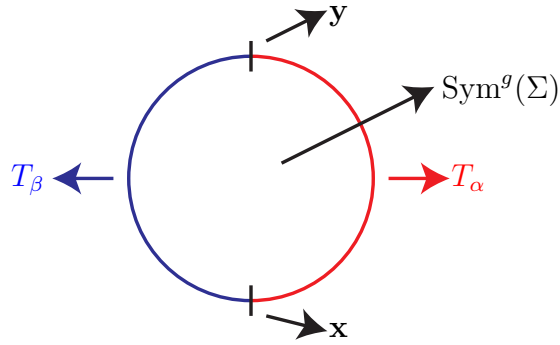
See Figure~\ref{fig:Whitney}. Such disks are called \emph{Whitney disks}.
Let $\cB(\x,\y)$ denote the space of Whitney disks from $\x$ to
$\y$. Further:
\begin{itemize}
\item Let $\pi_2(\x,\y)$ denote the set of homotopy
  classes of Whitney disks, i.e., the set of path components in $\cB(\x,\y)$.
\item Let $\tcM(\x,\y)\subset \cB(\x,\y)$ denote the space of
  holomorphic Whitney disks.
\end{itemize}
The space $\tcM(\x,\y)$ decomposes according to elements of
$\pi_2(\x,\y)$:
\[
\tcM(\x,\y)=\coprod_{B\in\pi_2(\x,\y)}\tcM^B(\x,\y)
\]
If $\tcM(\x,\y)$ is transversally cut-out, each space
$\tcM^B(\x,\y)$ is a smooth manifold whose dimension is given
by a number $\mu(B)$ called the \emph{Maslov index} of $B$. There is an
$\RR$-action on both $\cB(\x,\y)$ and $\tcM(\x,\y)$ by translation in
the source (thought of as an infinite strip). Let $\cM^B(\x,\y)=\tcM^B(\x,\y)/\RR$. Finally, the
differential on $\CFa(Y)$ is given by
\begin{equation}\label{eq:CF-d}
\bdy(\x)=
\sum_{\y\in T_\alpha\cap T_\beta}\sum_{\substack{B\in\pi_2(\x,\y)\\\mu(B)=1}}\bigl(\#\cM^B(\x,\y)\bigr) \y.
\end{equation}
(Here, $\#$ denotes the modulo-2 count of points.) Under certain
assumptions on $\HD$, called \emph{admissibility}, this count is
guaranteed to be finite, so $\bdy$ is well-defined. Moreover:
\begin{theorem}
  \cite{OS04:HolomorphicDisks} For any suitably generic choice of almost-complex structure, the map $\bdy$ satisfies
  $\bdy^2=0$. Moreover, the homology $\HFa(Y)=H_*(\CFa(Y),\bdy)$ is an
  invariant of $Y$.
\end{theorem}

\subsection{The cylindrical reformulation}\label{sec:cylindrical}
Before proceeding to bordered Floer homology, it will be helpful to
have a mild reformulation of the definition of $\HFa$. It is based on
the \emph{tautological correspondence} between maps from $\bD^2$ to
$\Sym^g(\Sigma)$ and multi-valued functions from $\bD^2$ to $\Sigma$:
\vspace{1em}
\begin{center}
\begin{tabular}{lcl}
Holomorphic maps $\bD^2\to \Sym^g(\Sigma)$ &
\quad$\longleftrightarrow$\quad \! &
Diagrams 
$\vcenter{\xymatrix{
  S\ar[r]^{u_\Sigma}\ar[d]_{u_\bD} & \Sigma\\
  \bD^2
}}$  \\
& &
with $u_\Sigma$, $u_\bD$ holomorphic, \\
& & $u_\bD$ a $g$-fold branched cover.
\end{tabular}
\end{center}
\vspace{1em}
One direction is easy: given a diagram as on the right, consider the
map $\bD^2\to \Sym^g(\Sigma)$ given by mapping $p$ to the $g$-tuple
$u_\Sigma(u_\bD^{-1}(p))$.  The other direction is not hard, either; see, for
instance,~\cite[Section 13]{Lipshitz06:CylindricalHF}.

In light of the tautological correspondence, we can reformulate $\HFa$
in terms of maps to $\Sigma\times\bD^2$. It will be convenient later
to view $\bD^2\setminus \{\pm i\}$ as a strip $[0,1]\times\RR$. Then:
\begin{itemize}
\item Generators of $\CFa(Y)$ correspond to $g$-tuples of points
  $\x=\{x_i\}_{i=1}^g$ with $x_i\in \alpha_i\cap \beta_{\sigma(i)}$
  for some $\sigma\in S_g$. These generators can be thought of as
  $g$-tuples of chords $\x\times [0,1]\subset \Sigma\times[0,1]$,
  connecting $\alphas\times \{1\}$ and $\betas\times\{0\}$.
\item The differential counts embedded holomorphic maps 
  \begin{equation}\label{eq:cyl-map}
  u\co (S,\bdy S)\to
  \bigl((\Sigma\setminus\{z\})\times[0,1]\times\RR,(\alphas\times\{1\}\times\RR)\cup(\betas\times\{0\}\times\RR) \bigr).
  \end{equation}
  modulo translation in $\RR$. Here, $S$ is a Riemann surface with
  boundary and punctures on its boundary. The punctures are divided
  into $+$ punctures and $-$ punctures. Near the $-$ punctures, $u$ is
  asymptotic to $\x\times[0,1]\times\{-\infty\}$ and near the $+$ punctures $u$ is
  asymptotic to $\y\times[0,1]\times\{+\infty\}$.
\end{itemize}
In the cylindrical setting, the set of homotopy classes $\pi_2(\x,\y)$
of Whitney disks becomes the set of homology classes (in a suitable
sense) of maps as in Formula~\ref{eq:cyl-map}.  (Philosophically, this
is related to the Dold-Thom theorem that $\pi_k(\Sym^\infty(X)) \cong
H_k(X)$.)

We have been suppressing almost-complex structures. In order to
achieve transversality, one typically perturbs the complex
structure $j_\Sigma\times j_\bD$ on $\Sigma\times[0,1]\times\RR$ to a
more generic almost-complex structure $J$. In this cylindrical
setting, it is important to ensure that translation in $\RR$ remains
$J$-holomorphic. Some other conditions which are necessary or
convenient are given in~\cite[Section 1]{Lipshitz06:CylindricalHF}.

\begin{remark}
  It would have been more consistent with conventions in contact
  homology to consider $\RR\times[0,1]\times\Sigma$ rather that
  $\Sigma\times[0,1]\times\RR$.
\end{remark}

\section{Holomorphic curves and Reeb chords}
Now consider a bordered Heegaard diagram
$\HD=(\Sigma,\alphas^a,\alphas^c,\betas,z)$. Rather than viewing
$\Sigma$ as a compact surface-with-boundary, attach a cylindrical end
$\RR\times S^1$ to $\bdy \Sigma$; and extend the $\alpha$-arcs
$\alphas^a$ in a translation-invariant way to $\RR\times
S^1$. (Topologically, this is the same as simply deleting
$\bdy\Sigma$; but if one is paying attention to the symplectic form
and almost-complex structure then there is a difference.) We 
abuse notation, using the same notation $\Sigma$ and $\alphas^a$ for 
the versions with cylindrical ends.
We will still consider holomorphic maps as in
Formula~\eqref{eq:cyl-map}; but now there is a third source of
non-compactness, $\bdy\Sigma$, and these maps can have asymptotics
there as well.

We start with the asymptotics at $\pm\infty$. A term for the
asymptotics at $\pm\infty$:
\begin{definition}
  By a \emph{generator} we mean a $g$-tuple $\x\subset
  \alphas\cap\betas$ which has one point on each $\alpha$-circle, one point
  on each $\beta$-circle, and at most one point on each $\alpha$-arc.
\end{definition}

We consider holomorphic curves disjoint from a neighborhood of $z$. It
follows from this and the fact that only the $\alpha$-arcs touch
$\bdy\Sigma$ that the asymptotics at $\bdy\Sigma$ are of the form
$\rho_i\times (1,t_i)$, where $\rho_i$ is a chord in
$\bdy\Sigma\setminus\{z\}$ with boundary on $\alphas^a$. We collect
these curves into moduli spaces. Let $\tcM(\x,\y;\rho_1,\dots,\rho_n)$
denote the moduli space of embedded holomorphic maps as in
Formula~\eqref{eq:cyl-map} where:
\begin{itemize}
\item $S$ is a surface with boundary and punctures on its
  boundary. Of these punctures, $g$ are labeled $-$, $g$ are labeled
  $+$, and the rest are labeled $e$.
\item $\x$ and $\y$ are generators.
\item at the punctures labeled $-$, $u$ is asymptotic to
  $\x\times[0,1]\times\{-\infty\}$. 
\item at the punctures labeled $+$, $u$ is asymptotic to
  $\y\times[0,1]\times\{+\infty\}$.
\item at the punctures labeled $e$, $u$ is asymptotic to the chords
  $\rho_i\times (1,t_i)\in \bdy\Sigma\times\{1\}\times\RR$. Moreover,
  we require that $t_1< t_2 < \cdots < t_n$.
\end{itemize}
There is an $\RR$-action on $\tcM(\x,\y;\rho_1,\dots,\rho_n)$ by translation in the target; let 
\[
\cM(\x,\y;\rho_1,\dots,\rho_n)=\tcM(\x,\y;\rho_1,\dots,\rho_n)/\RR.
\]

We call the chords $\rho$ \emph{Reeb chords}; they are
Reeb chords for the contact structure on $S^1=\bdy\Sigma$. This comes
from thinking of the setup as related to a Morse-Bott case of
(relative) symplectic field theory. The asymptotic boundary is then
$(\bdy\Sigma\times[0,1]\times\RR,\bdy\alphas^a\times \{1\}\times
\RR)$, and we are in the Levi-flat case of,
e.g.,~\cite{BEHWZ03:CompactnessInSFT}.

As in the closed case, the space of maps of the form just described
naturally decomposes into homology classes; see~\cite[Section~\ref*{LOT1:sec:homology-classes-generators}]{LOT1}. To keep notation consistent with the closed case, we let
$\pi_2(\x,\y)$ denote the set of homology classes of maps connecting
$\x$ to $\y$; note that we do not specify the Reeb chords here. Then
\[
\cM(\x,\y;\rho_1,\dots,\rho_n)=\coprod_{B\in
  \pi_2(\x,\y)}\cM^B(\x,\y;\rho_1,\dots,\rho_n).
\]
As in the closed case, we have been suppressing the almost-complex
structure $J$ from the discussion; the interested reader is referred
to~\cite[Section~\ref*{LOT1:sec:curves-in-sigma}]{LOT1}.  For a generic choice of $J$, each of the
spaces $\cM^B(\x,\y;\rho_1,\dots,\rho_n)$ is a manifold whose
dimension is given by a number $\ind(B;\rho_1,\dots,\rho_n)-1$. The
notation $\ind$ stands for index: as is usual for holomorphic curves,
the dimension is given by the index of the linearized
$\overline{\bdy}$-operator. One can give an explicit formula for
$\ind(B;\rho_1,\dots,\rho_n)$; see~\cite[Section~\ref*{LOT1:sec:expected-dimensions}]{LOT1}.

The next natural thing to talk about, from an analytic perspective, is
what the compactifications of $\cM^B(\x,\y;\rho_1,\dots,\rho_n)$ look
like. We defer this discussion to Lecture~\ref{lec:analysis}, and
instead turn to the definition of the bordered invariant $\CFDa(Y)$.

\section{The definition of \texorpdfstring{$\CFDa$}{CFD}}
\label{sec:defCFDa}
\subsection{Reeb chords and algebra elements}
Before defining $\CFDa(\HD)$ we need one more piece of notation. Let
$\PMC=(Z,\CircPts,M,z)$ be a pointed matched circle and $\rho$ a chord
in $Z\setminus\{z\}$ with boundary in $\CircPts$. Orienting $\rho$
according to the orientation of $Z$ and identifying
$\CircPts=\{1,\dots,4k\}$, the chord $\rho$ has an initial point $i$
and a terminal point $j$. Write
\begin{equation}\label{eq:a-of-rho}
a(\rho)=\sum_{\substack{S\subset \underline{4k}\\ i\in S}}(S, S\setminus
\{i\}\cup\{j\}, \phi_S)
\end{equation}
where $\phi_S(i)=j$ and $\phi_S|_{S\setminus i}=\Id$, and the sum is
only over $S$'s so that $S$ and $S\setminus\{i\}\cup\{j\}$ are
$M$-admissible. That is,
$a(\rho)$ is the union of a strand from $i$ to $j$ and any admissible
set of horizontal strands. A somewhat trivial example is given by
Exercise~\ref{ex:path-for-T2}.

\subsection{The definition of \texorpdfstring{$\CFDa$}{CFD}}\label{sec:def-CFD}
Fix a bordered Heegaard diagram
$\HD=(\Sigma,\alphas^a,\alphas^c,\betas,z)$ with boundary $\PMC$. We
will define a left \dg module $\CFDa(\HD)$ over $\Alg(-\PMC)$ (where,
as usual, $-$ denotes orientation reversal). The module $\CFDa(\HD)$ will
lie over $\Alg(-\PMC,0)$, in the sense that the other summands
$\Alg(-\PMC,i)$, $i\neq 0$, of $\Alg(-\PMC)$ act trivially on
$\CFDa(\HD)$.

Let $\Gen(\HD)$ denote the set of generators for $\HD$. Given a
generator $\x\in\Gen(\HD)$, let $I(S)$ denote the set of $\alpha$-arcs
which are disjoint from $\x$.%
\footnote{This $I(S)$ was denoted $I_D(S)$ in \cite{LOT1}, where
  $I(S)$ was used for $I_A(S)$ introduced in Section~\ref{sec:prove-pairing}.}
Then $I(S)$ corresponds to a set of
matched pairs in $-\PMC$, and hence, by Exercise~\ref{ex:idempotents},
to an indecomposable idempotent of $\Alg(-\PMC)$. As a (left) module, define
\[
\CFDa(\HD)=\bigoplus_{\x\in\Gen(\HD)}\Alg(-\PMC)\cdot I(S).
\]
It remains to define the differential on $\CFDa(\HD)$. For
$\x\in\Gen(\HD)$ define 
\begin{equation}\label{eq:CFD-d}
\bdy(\x)=\sum_{\y\in\Gen(\HD)}\sum_{n\geq
  0}\sum_{(\rho_1,\dots,\rho_n)}\sum_{B\mid
  \ind(B,\rho_1,\dots,\rho_n)=1}\bigl(\#\cM^B(\x,\y;\rho_1,\dots,\rho_n)\bigr)a(-\rho_1)\cdots a(-\rho_n)\y.
\end{equation}
Here, the minus signs are included because $\CFDa$ is a module over
$\Alg(-\PMC)$ rather than $\Alg(\PMC)$; $-\rho$ is the chord $\rho$
but viewed as running in the opposite direction (i.e., as a chord in
$-\PMC$).

Extend the differential to the rest of $\CFDa(Y)$ by the Leibniz rule.
This completes the definition of $\CFDa(Y)$.

\begin{figure}
  \centering
  \includegraphics[height=1.5in]{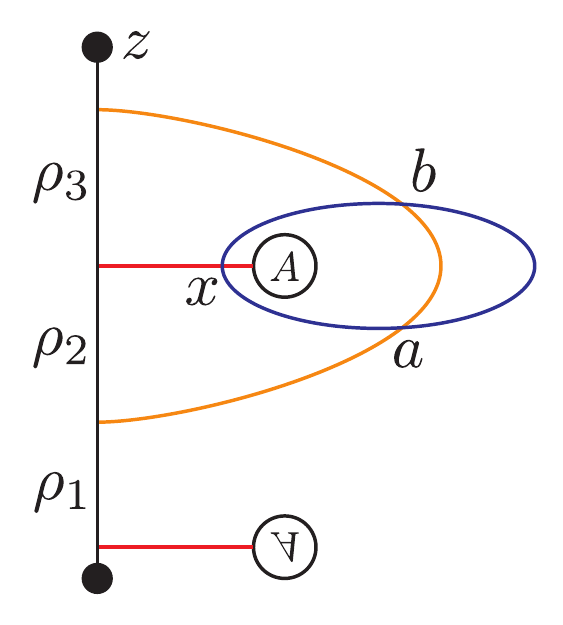}\includegraphics[height=1.5in]{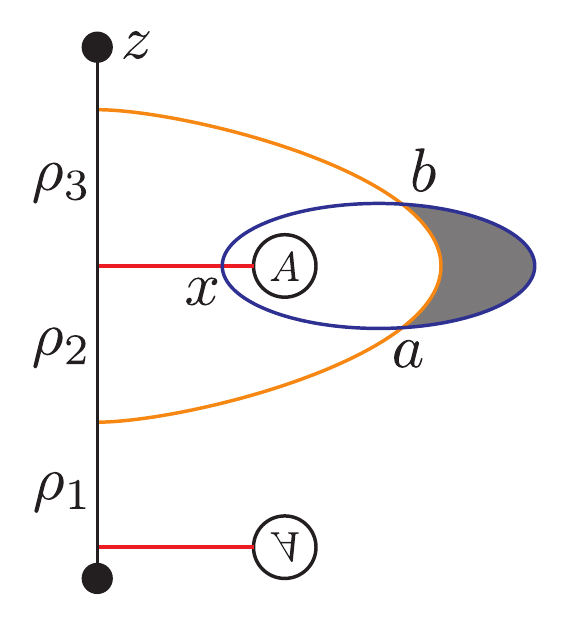}\includegraphics[height=1.5in]{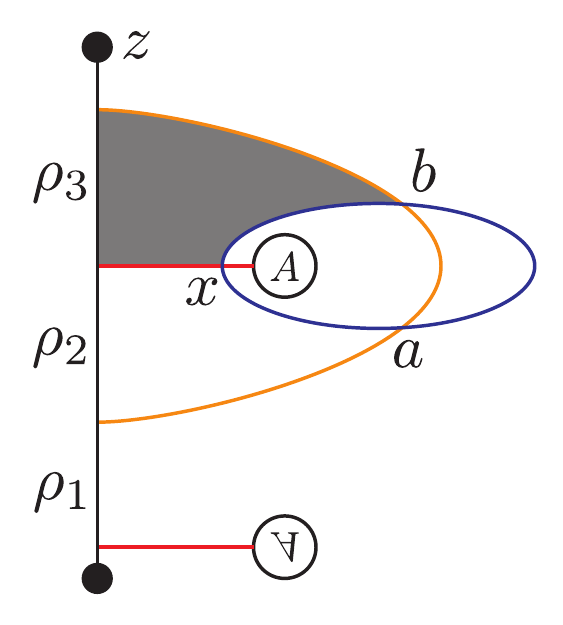}\includegraphics[height=1.5in]{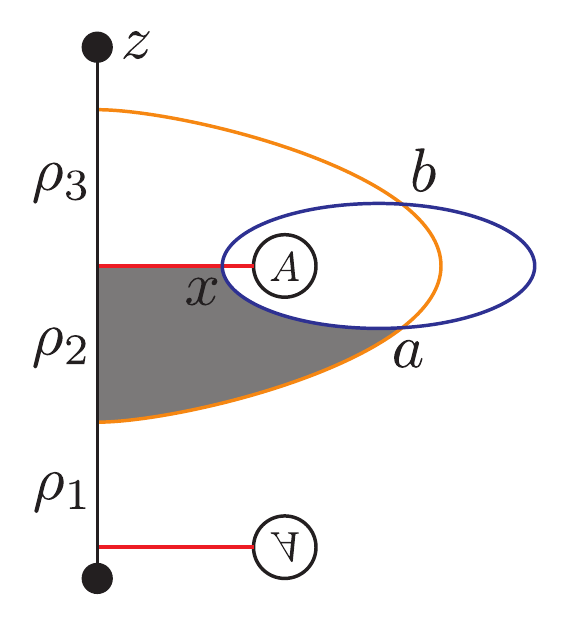}
  \caption{\textbf{A Heegaard diagram for a solid torus, and some holomorphic
      curves in it.} The circles labeled $A$ indicate a handle. The
    shaded regions in the second through fourth figures indicate the
    domains giving $a\in \bdy (b)$, $\rho_3x\in\bdy b$, and
    $\rho_2a\in\bdy x$, respectively.}
  \label{fig:HD-for-torus}
\end{figure}

\begin{example}
  Consider the bordered Heegaard diagram in
  Figure~\ref{fig:HD-for-torus}. We have labeled the three length-1
  Reeb chords; notice that we have ordered them in the opposite of the
  order induced by the orientation of $\bdy\HD$, because we are
  thinking of the algebra $\Alg(-\bdy\HD)$. The module $\CFDa(\HD)$
  has three generators, $x$, $a$ and $b$. With notation as in
  Formula~\ref{eq:torus-alg}, the idempotents are given by
  \[
  I(x)=\iota_1  \qquad\qquad I(a)=\iota_0 \qquad\qquad I(b)=\iota_0 
  \]
  The differentials are given by
  \begin{align*}
    \bdy(b)&=a+\rho_3x\\
    \bdy(x)&=\rho_2a\\
    \bdy(a)&=0.
  \end{align*}
  Each of these differentials comes from a disk mapped to
  $\Sigma\times[0,1]\times\RR$; the projections of these disks to
  $\Sigma$ (their \emph{domains}---see
  Definition~\ref{def:domain-admis}) are indicated in the figure.
  Since $\rho_3\rho_2=0$, $\bdy^2=0$.
\end{example}

\subsection{Finiteness conditions}
As in the closed case, the definition of $\CFDa$
(Formula~\eqref{eq:CFD-d}) only makes sense if the sums involved are
finite. To ensure finiteness, we add assumptions on the Heegaard
diagram $\HD$, analogous to admissibility in the closed case:
\begin{definition}\label{def:domain-admis}
  Given a homology class $B\in \pi_2(\x,\y)$, the projection of $B$ to
  $\Sigma$ defines a cellular $2$-chain with respect to the
  cellulation of $\Sigma$ given by $\alphas\cup\betas$. This $2$-chain
  is called the \emph{domain of $B$}, and determines $B$. A non-trivial class
  $B$ is called \emph{positive} if its local multiplicities are all non-negative. The domains of homology classes $B\in\pi_2(\x,\x)$
  are called \emph{periodic domains}. The set of periodic domains does
  not depend on $\x$.

  The Heegaard diagram $\HD$ is called \emph{provincially admissible}
  if it has no positive periodic domains which have multiplicity
  $0$ everywhere along $\bdy\Sigma$.

  The Heegaard diagram $\HD$ is called \emph{admissible} if it has no
  positive periodic domains.
\end{definition}

\begin{lemma}\label{lem:admis-good}\cite[Lemma~\ref*{LOT1:lem:finite-typeD}]{LOT1}
  If $\HD$ is provincially admissible then the sums in
  Formula~\eqref{eq:CFD-d} are finite. Moreover, if $\HD$ is
  admissible then the operator $\bdy$ is nilpotent in the following
  sense. Consider sequences of generators $(\x_1,\x_2,\dots, \x_n)$
  such that $\x_{i+1}$ occurs in $\bdy \x_i$ with nonzero
  coefficient. If $\HD$ is admissible then there is a universal bound
  on the length of such sequences.
\end{lemma}

The proof of Lemma~\ref{lem:admis-good} is not hard; it is an
adaptation of the proof of the corresponding fact from the closed
case~\cite[Lemma 4.14]{OS04:HolomorphicDisks}.  The nilpotency
condition in Lemma~\ref{lem:admis-good} guarantees that $\CFDa(\HD)$
is projective (or rather, $\mathcal{K}$-projective in the sense of,
e.g.,~\cite{BernsteinLunts94:EquivariantSheaves}). It is not
particularly relevant until we start taking tensor products, e.g.
in the statement of Theorem~\ref{thm:pairing1}.

\begin{theorem}\label{CFD-d-squared}\cite[Proposition~\ref*{LOT1:prop:typeD-d2}]{LOT1}
  Let $\HD$ be a provincially admissible Heegaard diagram. Then
  $\CFDa(\HD)$ is a differential module.
\end{theorem}
The only nontrivial thing to check is that $\bdy^2=0$. The proof
involves studying the boundaries of $1$-dimensional moduli spaces; we
will sketch it in the next lecture.

\section[The surgery exact triangle]{The surgery exact triangle\texorpdfstring{\footnote{The discussion in this section is taken
  from~\cite[Section~\ref*{LOT1:sec:surg-exact-triangle}]{LOT1}.}}{}}\label{sec:surgery}

Recall that Heegaard Floer homology admits a surgery exact
triangle~\cite{OS04:HolDiskProperties}.  Specifically, for a pair
$(M,K)$ of a 3-manifold~$M$ and a framed knot~$K$ in $M$, there
is an exact triangle
\begin{equation}
  \label{eq:surgery-exact-triangle}
\mathcenter{\begin{tikzpicture}[x=2.3cm,y=48pt]
  \node at (0,0) (m0) {$\HFa(M_{-1})$} ;
  \node at (2,0) (m1) {$\HFa(M_0)$} ;
  \node at (1,-1) (minf) {$\HFa(M_\infty)$} ;
  \draw[->] (m0) to (m1) ;
  \draw[->] (m1) to (minf) ;
  \draw[->] (minf) to node[auto] {} (m0) ;
\end{tikzpicture}}
\end{equation}
where $M_{-1}$, $M_0$, and $M_\infty$ are $-1$, $0$, and
$\infty$ surgery on $K$,  respectively.  As a simple application of
bordered Floer theory, we reprove this result.

Consider the three diagrams
\begin{equation}
  \label{eq:torus-diagrams}
  \HD_\infty:\mfigb{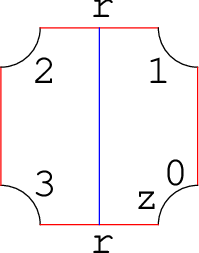}\qquad
  \HD_{-1}:\mfigb{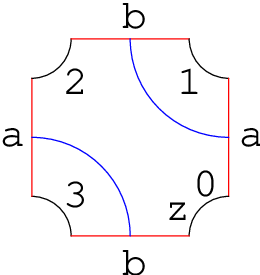} \qquad
  \HD_0:\mfigb{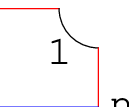}
\end{equation}
(Opposite edges are identified, to give $T^2\setminus\bD^2$. Each
diagram has two $\alpha$-arcs and one $\beta$-circle. The numbers
indicate which chord, in the notation of Formula~\eqref{eq:torus-alg},
corresponds to which arc in $\bdy\HD_\bullet$.  Note again that the chords
are numbered in the opposite of the order induced
by the orientation of $\bdy\HD_\bullet$.)  A generator for
$\CFDa(\HD_\bullet)$ consists of a single intersection point between
the $\beta$-circle in $\HD_\bullet$ and an $\alpha$-arc. These
intersections are labeled above.

The boundary operators on the $\CFDa(\HD_\bullet)$ (and the relevant domains) are given by
\newcommand{\tdiag}[2]{\underset{\includegraphics[scale=0.5]{torus-#2}}{\strut #1}}
\begin{equation*}
  \bdy r = \tdiag{\rho_{23}r}{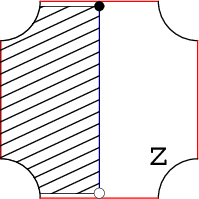}
  \qquad\qquad
  \bdy a = \tdiag{\rho_3b}{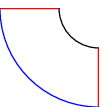} + \tdiag{\rho_1b}{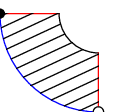}
  \qquad\qquad
  \bdy b = \tdiag{0}{}
  \qquad\qquad
  \bdy n = \tdiag{\rho_{12}n}{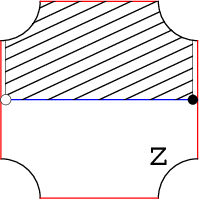}
\end{equation*}

There is a short exact sequence
\[
0 \longrightarrow \CFDa(\HD_\infty)\overset{\varphi}{\longrightarrow} \CFDa(\HD_{-1})
  \overset{\psi}{\longrightarrow} \CFDa(\HD_0) \longrightarrow 0
\]
where the maps $\phi$ and $\psi$ are given by
\begin{equation*}
  \varphi(r) = b + \rho_{2}a \qquad\qquad
    \psi(a) = n\qquad\qquad
    \psi(b)= \rho_{2}n.
\end{equation*}

Now, the surgery exact triangle follows immediately from the pairing theorem
and properties of the derived tensor product.

\section{The definition of
  \texorpdfstring{$\CFDDa$}{CFDD}}\label{sec:def-CFDD}
Suppose $\PMC_L$ and $\PMC_R$ are pointed matched circles. We can form
their connected sum $\PMC_L\#\PMC_R$. There are two natural choice of
where to put a basepoint in $\PMC_L\#\PMC_R$; let $z$ be a point in
one of these places and $w$ a point in the other. Thinking of $z$ as
the basepoint, there is an associated algebra
$\Alg(\PMC_L\#\PMC_R)$. Moreover, there is an algebra homomorphism
\[
p\co \Alg(\PMC_L\#\PMC_R)\to \Alg(\PMC_L)\otimes_\Field\Alg(\PMC_R)
\]
given by setting to zero any algebra element crossing the extra
basepoint $w$.

Now, suppose that $\HD$ is an arced Heegaard diagram. Performing
surgery on $\HD$ along the arc $\arcz$ gives a bordered Heegaard
diagram $\drHD$. (Again, there are two choices of where to put the
basepoint in $\drHD$; choose either.) If the boundary of $\HD$ was
$\PMC_L\amalg \PMC_R$ then the boundary of $\drHD$ is
$\PMC_L\#\PMC_R$.

Associated to $\drHD$ is a bordered module $\CFDa(\drHD)$ over
$\Alg(-(\PMC_L\#\PMC_R))$.
\begin{definition}
  With notation as above, let 
  \[
  \CFDDa(\HD)=\bigl((-\Alg(\PMC_L))\otimes_\Field (-\Alg(\PMC_R))\bigr)\otimes_{\Alg(-(\PMC_L\#\PMC_R))}\CFDa(\drHD),
  \]
  be the image of the bordered bimodule $\CFDa(\drHD)$ under the
  induction functor associated to the homomorphism $p$. Via the
  correspondence between left-left bimodules over
  $\Alg(-\PMC_L)$ and $\Alg(-\PMC_R)$ and left modules over
  $\bigl((-\Alg(\PMC_L))\otimes_\Field (-\Alg(\PMC_R))\bigr)$,
  we view $\CFDDa(\HD)$ as a left-left bimodule over
  $\Alg(-\PMC_L)$ and $\Alg(-\PMC_R)$.
\end{definition}

Of course, this definition can be unpacked to define $\CFDDa(\HD)$
directly in terms of intersection points and holomorphic curves; doing
so is Exercise~\ref{ex:define-CFDD}.

\section{Exercises}

\begin{exercise}
  There is a unique almost-complex structure $\Sym^g(j_\Sigma)$ on
  $\Sym^g(\Sigma)$ so that the projection map $(\Sigma^{\times g},
  j_\Sigma^{\times g})\to (\Sym^g(\Sigma), \Sym^g(j_\Sigma))$ is
  holomorphic.  In the tautological correspondence of
  Section~\ref{sec:cylindrical}, show that if $u_\Sigma$ and $u_\bD$
  are holomorphic then the map $\bD^2\to \Sym^g(\Sigma)$, $p\mapsto
  u_\Sigma(u_\bD^{-1}(p))$ is holomorphic with respect to
  $\Sym^g(j_\Sigma)$.
\end{exercise}

\begin{exercise}\label{ex:solid-tori}
  Consider the Heegaard diagrams of Section~\ref{sec:surgery}. Replacing the blue ($\beta$) curve in the diagrams $\HD_\bullet$ by
  a circle of slope $p/q$ gives a bordered Heegaard diagram
  $\HD_{p/q}$ for a $p/q$-framed solid torus. It is fairly easy to
  compute the invariants $\CFDa(\HD_{p/q})$ for these diagrams;
  compute some. 

  For any triple of rational numbers $(p_1/q_1,p_2/q_2,p_3/q_3)$ (with
  $p_i,q_i$ relatively prime) such that $p_1+p_2+p_3=q_1+q_2+q_3=0$
  there is a corresponding surgery triangle; check this for some other
  examples.
\end{exercise}

\begin{figure}
  \input{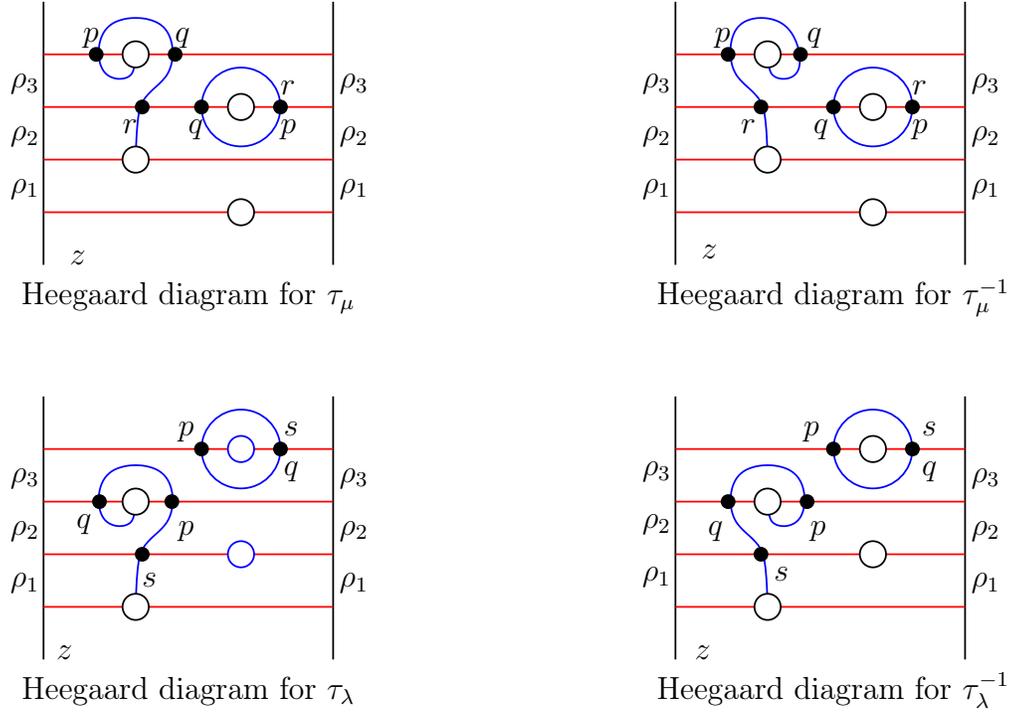}
  \caption[Heegaard diagrams for Dehn twists of the
  torus.]{\label{fig:DehnTwistsGenusOne} {\bf{Heegaard diagrams for
        mapping class group elements.}}  Genus $2$ diagrams for
    $\tau_\mu$, $\tau_\mu^{-1}$, $\tau_\lambda$ and
    $\tau_\lambda^{-1}$ are shown. In each of the four diagrams, there
    are three generators in the $i=0$ summand. (This figure is drawn
    from~\cite[Figure A.2]{LOT0}.)}
\end{figure}

\begin{exercise}
  Compute $\Mor(\CFDa(\HD_{p/q}),\CFDa(\HD_{r/s}))$ for a few choices of
  $p,q,r,s$. For example, $\Mor(\CFDa(\HD_\infty),\CFDa(\HD_{-1}))$
  has generators $(r\mapsto b)$, $(r\mapsto \rho_{23} b)$ and
  $(r\mapsto \rho_2a)$. The differentials are given by
  \begin{align*}
    \bdy (r\mapsto b)&= (r\mapsto \rho_{23} b)\\
    \bdy(r\mapsto \rho_2a)&=(r\mapsto \rho_{23}b).
  \end{align*}
  In particular, the homology of this $\Mor$ complex is
  $1$-dimensional.

  Recall that $\HFa(L(p,q))\cong(\mathbb{F}_2)^p$ , and $\HFa(S^2\times
  S^1)\cong (\mathbb{F}_2)^2$; check that your answers are consistent with this.
\end{exercise}

\begin{exercise}\label{ex:DD-id-torus}
  We explain the type \DD\ bimodule $\CFDDa(\Id,0)$ associated to the
  mapping cylinder for the identity map of $F(\PMC)$. The notation is
  somewhat cumbersome, as $\CFDDa(\Id,0)$ has two commuting left
  actions by $\Alg(T^2,0)$. We write one of these copies of
  $\Alg(T^2,0)$ in the notation of Formula~\eqref{eq:torus-alg}, and
  the other in the same way but with $\sigma$'s in place of $\rho$'s
  and $\eta$'s in place of $\iota$'s.  Then, the bimodule $\CFDDa(\Id,0)$
  has two generators, $x$ and $y$, with
  \[
  \iota_{0} x = \eta_{0} x = x \qquad\qquad \iota_{1} y = \eta_{1} y = y
  \]
  and differential given by 
  \begin{equation}\label{eq:torus-DD-id}
    \begin{split}
      \bdy x &= (\rho_1\sigma_3+ \rho_3\sigma_1 + \rho_{123}\sigma_{123}) \otimes y \\
      \bdy y &= (\rho_2\sigma_2)\otimes x.
    \end{split}
  \end{equation}
  (Compare~\cite[Section~\ref*{LOT1:sec:dd-of-id}]{LOT1}.)

    Verify that for the modules $\CFDa(\HD_\bullet)$ of
    Section~\ref{sec:surgery}, $\Mor(\CFDDa(\Id,0),\cdot)$ acts as the
    identity. That is, check that
    \[
    \Mor_{\Alg(T^2,0)}(\CFDDa(\Id,0),\CFDa(\HD_0))\simeq \CFDa(\HD_0),
    \]
    and similarly for $\HD_{-1}$, $\HD_\infty$. (You will have to use
    the equivalence of categories between left $\Alg(T^2,0)$-modules
    and right $\Alg(T^2,0)$-modules coming from the fact that
    $\Alg(T^2,0)\cong \Alg(T^2,0)^\op$. Note that this isomorphism
    exchanges $\rho_1$ and $\rho_3$.)
\end{exercise}

\begin{remark}
  There are two non-equivalent notions of the $\Mor$ complex above,
  depending on how one treats the other algebra action on
  $\CFDDa(\Id,0)$. The exercise will be true with either
  notion. See~\cite[Theorems 5 and 6]{LOTHomPair} for an example where
  this distinction matters.
\end{remark}

\begin{remark}\label{rem:DD-notation}
  It is sometimes convenient to encode the operations in
  Formula~\eqref{eq:torus-DD-id} by:
  \[
  \begin{tikzpicture}[y=54pt,x=1in]
    \node at (0,0) (x) {$x$} ; \node at (2,0) (y) {$y$} ; \draw[->,
    bend left=20] (x) to node[above]
    {${\rho_1\sigma_3+\rho_3\sigma_1+\rho_{123}\sigma_{123}}$} (y) ;
    \draw[->, bend left=20] (y) to node[below] {${\rho_2\sigma_2}$}
    (x) ;
  \end{tikzpicture}.
  \]
  This way of encoding operations on \DD\ bimodules will be used in
  Exercise~\ref{ex:torus-twists}.
\end{remark}

\begin{exercise}
  Note that the identity for $\Mor$ is the $\Alg$-bimodule $\Alg$. In
  spite of the computations in Exercise~\ref{ex:DD-id-torus},
  $\CFDDa(\Id)\not\simeq \Alg(T^2,0)$. Check this
  two ways:
  \begin{itemize}
  \item Directly. (Think about the rank of the homologies.)
  \item By finding a module $M$ over $\Alg(T^2,0)$ so that
    $\CFDDa(\Id)\otimes_{\Alg(T^2,0)} M\not\simeq M$. (Or, you can use
    $\Mor(\CFDDa(\Id),M)$ if you prefer.)
  \end{itemize}
\end{exercise}

\begin{exercise}\label{ex:torus-twists}
  Let $\tau_\mu$ and $\tau_\lambda$ denote the Dehn twists of the
  torus along a meridian and a longitude, respectively. Heegaard
  diagrams for the mapping cylinders of $\tau_\mu$ and $\tau_\lambda$
  are shown in Figure~\ref{fig:DehnTwistsGenusOne}. With notation as
  in Remark~\ref{rem:DD-notation}, the type \DD\ bimodules associated
  to these Dehn twists and their inverses are given by
\begin{center}
  \begin{tikzpicture}[y=54pt,x=1in]
    \node at (0,2) (p) {${\mathbf p}$} ;
    \node at (2,2) (q) {${\mathbf q}$} ;
    \node at (1,0) (r) {${\mathbf r}$} ;
    \node at (1,1.25) (label) {$\boldsymbol{\tau_\mu}$};
    \draw[->] (p) to node[above,sloped] {\lab{\rho_1\sigma_3+\rho_{123}\sigma_{123}}}  (q)  ;
    \draw[->] (p) [bend left=15] to node[above,sloped] {\lab{
\rho_3\sigma_{12}}} (r) ;
    \draw[->] (q) [bend left=15] to node[below,sloped] {\lab{\rho_{23}\sigma_{2}}} (r) ;
    \draw[->] (r) [bend left=15] to node[below,sloped] {\lab{\rho_2}} (p) ;
    \draw[->] (r) [bend left=15] to node[below,sloped] {\lab{\sigma_1}} (q) ;
  \end{tikzpicture}
  \qquad
  \begin{tikzpicture}[y=54pt,x=1in]
    \node at (0,2) (p) {${\mathbf p}$} ;
    \node at (2,2) (q) {${\mathbf q}$} ;
    \node at (1,0) (r) {${\mathbf r}$} ;
    \node at (1,1.25) (label) {$\boldsymbol{\tau_\mu^{-1}}$};
    \draw[->] (p) to node[above,sloped] {\lab{\rho_1 \sigma_3+\rho_{123}\sigma_{123}}}  (q)  ;
    \draw[->] (p) [bend left=15] to node[above,sloped]  {\lab{\rho_3}} (r) ;
    \draw[->] (q) [bend left=15] to node[below,sloped] {\lab{\sigma_2}} 
                  (r) ;
    \draw[->] (r) [bend left=15] to node[below,sloped] {\lab{\rho_2 \sigma_{12}}} (p) ;
    \draw[->] (r) [bend left=15] to node[below,sloped] {\lab{\rho_{23} \sigma_1}} (q) ;
  \end{tikzpicture} \\
  \begin{tikzpicture}[y=54pt,x=1in]
    \node at (2,2) (p) {${\mathbf q}$} ;
    \node at (0,2) (q) {${\mathbf p}$} ;
    \node at (1,0) (r) {${\mathbf s}$} ;
    \node at (1,1.25) (label) {$\boldsymbol{\tau_\lambda}$};
    \draw[->] (p) [bend left=15] to node[below,sloped] {\lab{\rho_2 \sigma_{23}}} (r) ;
    \draw[->] (q) [bend left=15] to node[above,sloped] {\lab{\rho_{12}\sigma_3}} (r) ;
    \draw[->] (q) to node[above,sloped] {\lab{\rho_3\sigma_1+\rho_{123}\sigma_{123}}} (p) ;
    \draw[->] (r) [bend left=15] to node[above,sloped] {\lab{\rho_1}} (p) ;
    \draw[->] (r) [bend left=15] to node[below,sloped] {\lab{\sigma_2}} (q) ;
  \end{tikzpicture}
  \begin{tikzpicture}[y=54pt,x=1in]
    \node at (2,2) (p) {${\mathbf q}$} ;
    \node at (0,2) (q) {${\mathbf p}$} ;
    \node at (1,0) (r) {${\mathbf s}$} ;
    \node at (1,1.25) (label) {$\boldsymbol{\tau_\lambda^{-1}}$};
    \draw[->] (p) [bend left=15] to node[below,sloped] {\lab{\rho_2}}  (r)  ;
    \draw[->] (q) [bend left=15] to node[above,sloped]  {\lab{\sigma_3}} (r) ;
    \draw[->] (q)  to node[above,sloped] {\lab{\rho_3\sigma_1+\rho_{123}\sigma_{123}}} (p) ;
    \draw[->] (r) [bend left=15] to node[above,sloped] {\lab{\rho_1 \sigma_{23}}} (p) ;
    \draw[->] (r) [bend left=15] to node[below,sloped] {\lab{\rho_{12}\sigma_2}} (q) ;
  \end{tikzpicture}
\end{center}
Convince yourself that these bimodules satisfy $\bdy^2=0$. 
  Compute $\Mor(\CFDDa(\tau_\mu),\HD_0)$ and
  $\Mor(\CFDDa(\tau_\lambda),\HD_0)$. Compare the results with the
  answers you computed in Exercise~\ref{ex:solid-tori}.
\end{exercise}

\begin{exercise}
  Up to Heegaard moves, there are some symmetries relating the
  diagrams in Figure~\ref{fig:DehnTwistsGenusOne}. How are these
  symmetries reflected in the bimodules in Exercise~\ref{ex:torus-twists}?
\end{exercise}

\begin{exercise}\label{ex:define-CFDD}
  Unpack the definition of $\CFDDa$ from Section~\ref{sec:def-CFDD} to
  give a direct definition, avoiding the induction functor.
\end{exercise}

\chapter[Analysis underlying the invariants]{Analysis underlying the invariants and the pairing theorem}\label{lec:analysis}
\section{Broken flows in the cylindrical setting}\label{sec:broken-cyl}
As a warm-up, we begin this lecture by discussing the proof that $\bdy^2=0$ for the
cylindrical picture for Heegaard Floer homology. We start with an
example.  Consider the Heegaard diagram for $S^3$ shown in
Figure~\ref{fig:S3-moduli}. There are five generators, labeled $a$,
$b$, $c$, $d$ and $e$. The differentials are given by
\[
\bdy(a) = b+c \qquad\qquad \bdy(b)=\bdy(c)=d \qquad\qquad \bdy(d)=0\qquad\qquad\bdy(e)=b+c.
\]
(Remember that we are working with $\Field$-coefficients.)

Consider the moduli space $\cM(a,d)$ of curves connecting $a$ to
$d$. This moduli space consists of holomorphic maps
\[
u\co (\bD^2\setminus \{\pm i\})\to \Sigma\times[0,1]\times\RR.
\]
Suppose we are working with the almost-complex structure
$j_\Sigma\times j_\bD$. Then there are projection maps $\pi_\Sigma\co
\Sigma\times[0,1]\times\RR\to \Sigma$ and $\pi_\bD\co
\Sigma\times[0,1]\times\RR\to [0,1]\times\RR$, and $u$ being
holomorphic is equivalent to $\pi_\Sigma\circ u$ and $\pi_\bD\circ u$
being holomorphic. 

The map $\pi_\bD\circ u$ is a $1$-fold branched cover, i.e., an
isomorphism; up to translation, there is a unique such isomorphism.

A short argument using the Riemann mapping theorem shows that the map
$\pi_\Sigma\circ u$ is determined by the image of
$\bdy\bD^2$. Figure~\ref{fig:S3-moduli} shows two possibilities for
$\pi_\Sigma(u(\bdy\bD^2))$. Note the branch point on $\alpha_1$ or
$\beta_1$. The whole moduli space is determined by where the branch
point lies; so, $\cM(a,d)$ is an (open) interval. The ends of
$\cM(a,d)$ occur when the branch point approaches $b$ or $c$.

\begin{figure}
  \centering
  \begin{overpic}[tics=10,width=.8\textwidth]{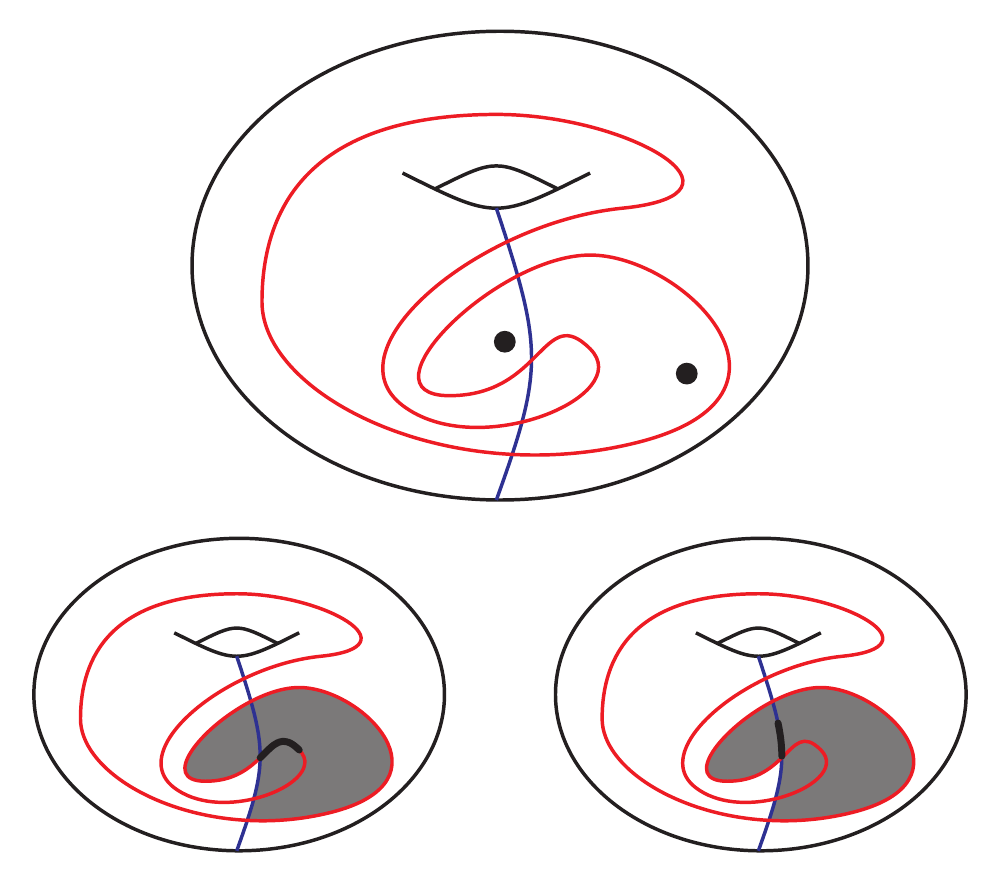}
    \put(51,40){$a$}
    \put(53,43){$b$}
    \put(54,50){$d$}
    \put(53,58){$c$}
    \put(51,64.5){$e$}
    \put(48,34){\textcolor{blue}{$\beta$}}
    \put(23,57){\textcolor{red}{$\alpha$}}
    \put(50,80){$z$}
    \put(46,53){$p_2$}
    \put(64,50){$p_1$}
  \end{overpic}
  \caption{\textbf{An unnecessarily complicated diagram for $S^3$.}
    In the two pictures on the bottom we have indicated the image
    $\pi_\Sigma(u(\bdy\bD^2))$ for two typical elements of
    $\cM(a,d)$. The thick black segments indicate cuts.}
  \label{fig:S3-moduli}
\end{figure}
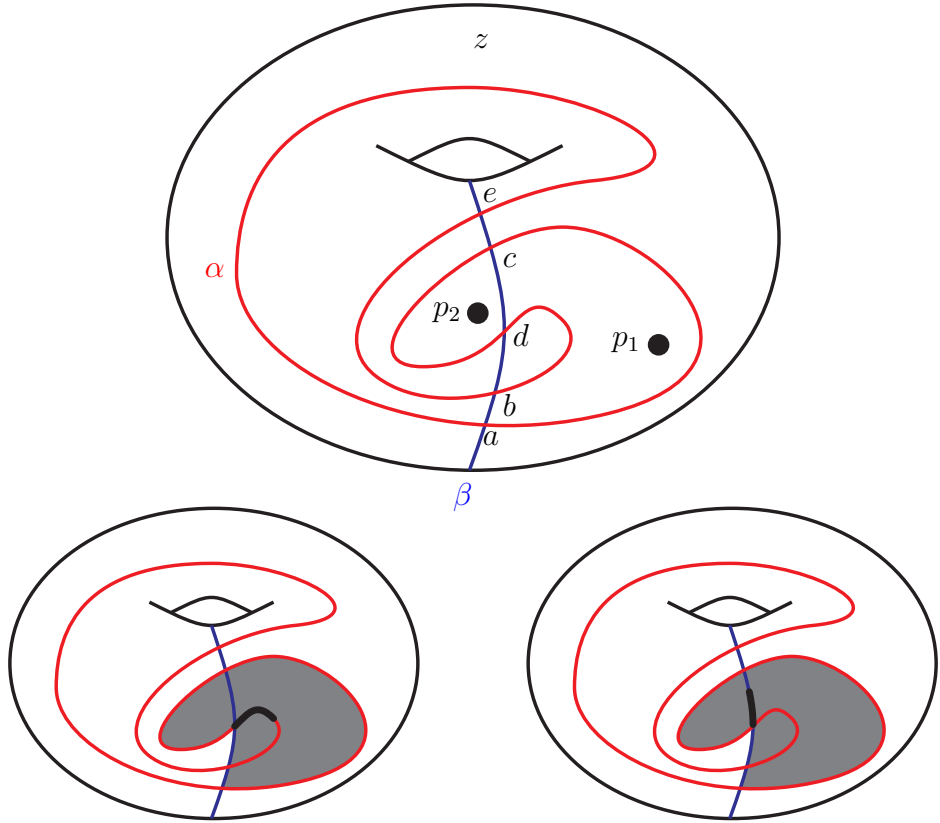

We want to describe the limiting objects. In the ordinary setting for
Morse theory, these would be broken flows. In this setting, they are
multi-story holomorphic buildings. We see this as follows.

Consider a sequence of curves $u_i$ approaching the end of $\cM(a,d)$
where the
branch point approaches $c$. Notice the points $p_1,p_2\in\Sigma$
shown in Figure~\ref{fig:S3-moduli}. Consider the points
$q_1=(\pi_\Sigma\circ u_i)^{-1}(p_1)$ and $q_2=(\pi_\Sigma\circ
u_i)^{-1}(p_2)$ in $\bD^2$. The points $(\pi_\bD\circ u_i)(q_1)$ and
$(\pi_\bD\circ u_i)(q_2)$ in $[0,1]\times\RR$ are getting farther and
farther apart. Indeed, from the point of view of $q_1$, half of the
holomorphic curve is heading towards
$\Sigma\times[0,1]\times\{+\infty\}$, while from the point of view of
$q_2$, half of the holomorphic curve is heading towards
$\Sigma\times[0,1]\times\{-\infty\}$. So, the limiting object has two
``stories'': the part of the limit containing $q_1$ and the part of
the limit containing $q_2$. More formally:
\begin{definition}
  An \emph{$\ell$-story holomorphic building} connecting $\x$ to $\y$
  consists of a sequence of holomorphic curves $u_i\in
  \cM(\x_i,\x_{i+1})$, $i=1,\dots,\ell$, with $\x_1=\x$ and
  $\x_{\ell+1}=\y$.

  Each holomorphic building carries a homology class in
  $\pi_2(\x,\y)$, by adding up (concatenating) the homology classes of
  its stories.
\end{definition}

We should now give a topology on the space of holomorphic buildings,
to say precisely what it means for a sequence of one-story buildings,
i.e., elements of $\cM(a,d)$, to converge to a multi-story
building. Instead, however, we refer the reader
to~\cite{BEHWZ03:CompactnessInSFT}.

The main structural result is:
\begin{theorem}\label{thm:closed-cyl-cpct-glue}
  Suppose that $B\in \pi_2(\x,\y)$ has $\mu(B)=2$. Let $\ocM^B(\x,\y)$
  denote the space of $1$- or $2$-story holomorphic buildings
  connecting $\x$ to $\y$ in the homology class $B$. Then for a
  generic choice of almost-complex structure, $\ocM^B(\x,\y)$ is a
  compact $1$-dimensional manifold-with-boundary. The boundary of
  $\ocM^B(\x,\y)$ consists exactly of the $2$-story holomorphic
  buildings connecting $\x$ to $\y$ in the homology class $B$.
\end{theorem}
In the cylindrical formulation, this is~\cite[Corollary
7.2]{Lipshitz06:CylindricalHF}; the analogous result for Heegaard
Floer homology in the non-cylindrical setting was proved
in~\cite{OS04:HolomorphicDisks}. (Both proofs are relatively modest
adaptations of standard holomorphic curve techniques.)

To conclude the warm-up, we recall that $\bdy^2=0$ follows from
Theorem~\ref{thm:closed-cyl-cpct-glue} by a standard argument:
\begin{corollary}\label{cor:closed-cyl-d-squared}
  Let $\HD$ be an admissible Heegaard diagram for a closed
  $3$-manifold. Then the differential $\bdy$ on $\CFa(\HD)$ satisfies
  $\bdy^2=0$.
\end{corollary}
\begin{proof}
  The proof involves the usual looking at ends of one-dimensional
  moduli spaces, as is familiar in Floer homology:
  \begin{align*}
    \bdy(\x)&=\sum_{\y\in\Gen(\HD)}\sum_{\substack{B_1\in\pi_2(\x,\y)\\
        \mu(B_1)=1}}\bigl(\#\cM^{B_1}(\x,\y)\bigr)\y\\
    \bdy^2(x)&=\sum_{\y\in\Gen(\HD)}\sum_{\substack{B_1\in\pi_2(\x,\y)\\
        \mu(B_1)=1}}\bigl(\#\cM^{B_1}(\x,\y)\bigr)\bdy(\y)\\
    &=\sum_{\y,\z\in\Gen(\HD)}\sum_{\substack{B_1\in\pi_2(\x,\y)\\
        \mu(B_1)=1}}\sum_{\substack{B_2\in\pi_2(\y,\z)\\
        \mu(B_2)=1}}\bigl(\#\cM^{B_1}(\x,\y)\bigr)\bigl(\#\cM^{B_2}(\y,\z)\bigr)\z\\
    &=\sum_{\z\in\Gen(\HD)}\sum_{\substack{B\in\pi_2(\x,\z)\\
        \mu(B)=2}}\bigl(\#\bdy
    \cM^{B}(\x,\z)\bigr)\z\\
    &=0.
  \end{align*}
  Most of this is just manipulation of symbols; the key point is the
  fourth equality, which uses
  Theorem~\ref{thm:closed-cyl-cpct-glue}. The last equality follows
  from the fact that a $1$-dimensional manifold-with-boundary has an
  even number of ends. (The assumption about admissibility is used to
  ensure that the sums involved at each stage are finite.)
\end{proof}

\section{The codimension-one boundary: statement}\label{sec:codim-1-bdy}
To prove that $\bdy^2=0$ for $\CFDa$ we need to investigate the
boundary of the $1$-dimensional moduli spaces, analogously to
Theorem~\ref{thm:closed-cyl-cpct-glue}. So, fix a bordered Heegaard
diagram $\HD=(\Sigma,\alphas^c,\alphas^a,\betas)$. As above, we can
have breaking at $\pm\infty$, giving multi-story holomorphic
buildings; but now there are two other sources of non-compactness:
\begin{enumerate}
\item\label{item:e-infty-degen} The manifold $\Sigma$ has a
  cylindrical end, giving another direction in which curves in
  $\Sigma\times[0,1]\times\RR$ can break.
\item\label{item:collapse-degen} In the moduli space
  $\cM^B(\x,\y;\rho_1,\dots,\rho_n)$ we had Reeb chords $\rho_i\times
  (1,t_i)$ where $t_1<t_2<\dots<t_n$. This can degenerate when
  $t_{i+1}-t_i\to 0$.
\end{enumerate}
(There is overlap between the two cases.)

Degenerations of type~(\ref{item:e-infty-degen}) lead to the analogue of
$2$-story holomorphic buildings, but in the ``horizontal'', i.e.,
$\Sigma$, direction. In principle, one can have degenerations in both
the vertical ($\RR$) and horizontal ($\Sigma$) directions at once. We
called the resulting objects \emph{holomorphic combs}~\cite[Definition~\ref*{LOT1:def:comb}]{LOT1}. In codimension
$1$, the kinds of combs that can appear are quite limited, so rather
than giving the general story we will simply explain these cases.

By \emph{east $\infty$} we mean
$\RR\times(\bdy\Sigma)\times[0,1]\times\RR$; this is the symplectic
manifold that one sees at the (``horizontal'') end of
$\Sigma$. Note that there are projection maps 
\begin{align*}
  \pi_\Sigma&\co\RR\times(\bdy\Sigma)\times[0,1]\times\RR\to\RR\times(\bdy\Sigma)\\
  \pi_\bD&\co\RR\times(\bdy\Sigma)\times[0,1]\times\RR\to
  [0,1]\times\RR\\
  t&\co \RR\times(\bdy\Sigma)\times[0,1]\times\RR\to\RR,
\end{align*}
where $t$ is projection onto the second (last) $\RR$-factor.
Degenerations of type~(\ref{item:e-infty-degen}) lead to pairs
$(u,v)$ where $u$ is a curve in $\Sigma\times[0,1]\times\RR$ of the
kind we have been considering and $v$ is a curve at east $\infty$,
i.e., a holomorphic map
\[
v\co (S,\bdy S)\to (\RR\times (\bdy\Sigma)\times[0,1]\times\RR, \RR\times(\alphas\cap\bdy\Sigma)\times\{1\}\times\RR).
\] 
Here, $S$ is a surface with boundary and punctures on the
boundary. Each puncture is labeled either $e$ or $w$. Near each $e$
puncture, $v$ is asymptotic to some
$\{\infty\}\times\rho_i\times(1,t_i)$ where $\rho_i$ is a chord in
$\bdy\Sigma$ and $t_i\in\RR$. Similarly, near each $w$ puncture, $v$
is asymptotic to some $\{-\infty\}\times\rho_i\times(1,t_i)$.

\begin{figure}
\includegraphics{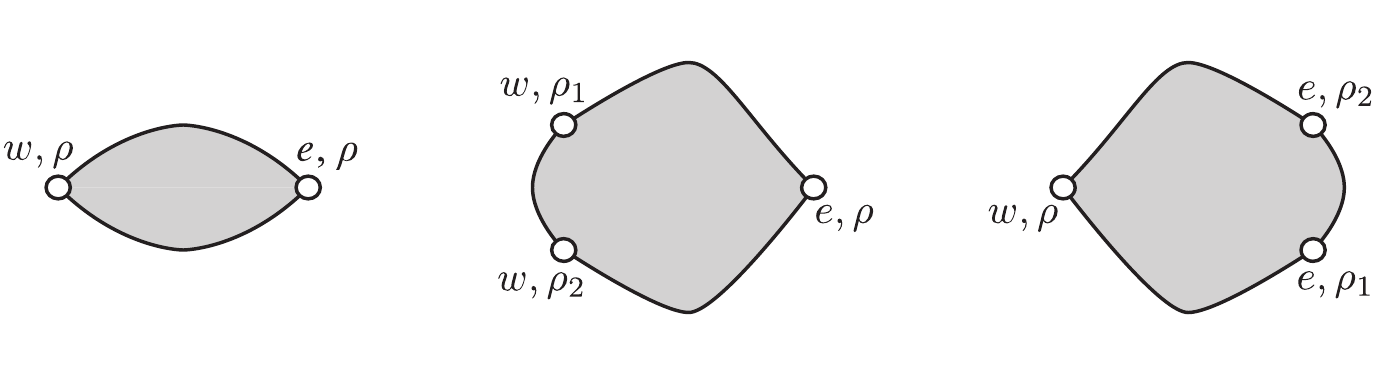}
\caption[Sources of curves at east $\infty$.]{\textbf{Sources of curves at east $\infty$.} Left: a trivial component. Center: a join component. Right: a
  split component. This is~\cite[Figure~\ref*{LOT1:fig:trivial_split_join_domain}]{LOT1}.}\label{fig:trivial_split_join_domain}
\end{figure}

It follows from the boundary conditions and asymptotics that for each
component of $v$, the map $\pi_\bD\circ v$ is, in fact, constant. This
makes describing holomorphic curves at east $\infty$ relatively
straightforward. Three kinds of curves will play special
roles in studying $\CFDa$:
\begin{itemize}
\item A \emph{trivial component} is a disk in
  $\RR\times(\bdy\Sigma)\times[0,1]\times\RR$ which is invariant under
  translation in the first $\RR$-factor. It follows that a trivial
  component has one $w$ punctures and one $e$ puncture, and is
  asymptotic to the same chord $\rho$ at both punctures.
\item A \emph{join component} is a disk in
  $\RR\times(\bdy\Sigma)\times[0,1]\times\RR$ with two $w$ punctures and
  one $e$ puncture. At the two $w$ punctures the curve is asymptotic
  to chords $\rho_1$ and $\rho_2$ and at the $e$ puncture the curve is
  asymptotic to a chord $\rho$. With respect to the cyclic ordering of
  the punctures $(\rho,\rho_1,\rho_2)$ around the boundary of the disk
  (see
  Figure~\ref{fig:trivial_split_join_domain}), the terminal endpoint
  of $\rho_2$ is the initial endpoint of $\rho_1$; and
  $\rho=\rho_2\cup\rho_1$.

  A \emph{join curve} is the disjoint union of one join component and
  finitely many trivial components.
\item Roughly, a \emph{split component} is the mirror of a join
  component. In more detail, a split component is a disk in
  $\RR\times(\bdy\Sigma)\times[0,1]\times\RR$ with one $w$ punctures and
  two $e$ puncture. At the two $e$ punctures the curve is asymptotic
  to chords $\rho_1$ and $\rho_2$ and at the $w$ puncture the curve is
  asymptotic to a chord $\rho$. With respect to the cyclic ordering of
  the punctures $(\rho,\rho_1,\rho_2)$ around the boundary of the disk
   (see
  Figure~\ref{fig:trivial_split_join_domain}), the terminal endpoint
  of $\rho_1$ is the initial endpoint of $\rho_2$; and
  $\rho=\rho_1\cup\rho_2$.

  For our purposes, a \emph{split curve} is the disjoint union of one
  split component and finitely many trivial components. (If we were also
  interested in $\CFAa$, we would have to allow more than one split
  component in a split curve.)
\end{itemize}

\begin{figure}
\includegraphics[scale=.857143]{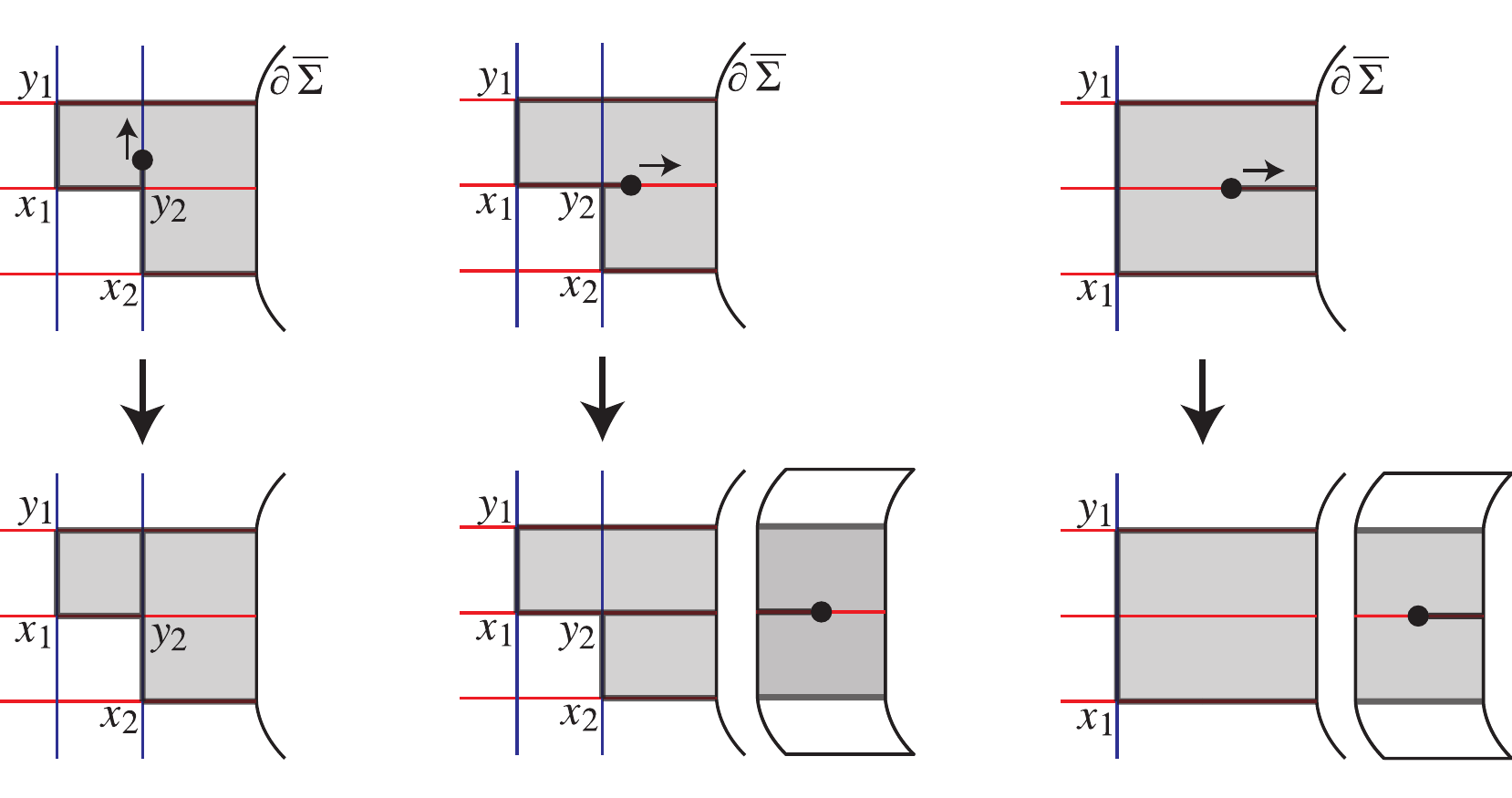}
\caption[Examples of three kinds of codimension $1$
  degenerations.]{\textbf{Examples of three kinds of codimension $1$
  degenerations.} The large black dot represents a boundary branch
  point of $\pi_\Sigma\circ u$. Left: degenerating into a two-story
  building. Center: degenerating a join curve. Right: degenerating a
  split curve. The diagrams show the projection of the curve in
  $\Sigma \times [0,1] \times \RR$ to $\Sigma$. This figure is adapted
  from~\cite[Figure~\ref*{LOT1:fig:degen_examples}]{LOT1}.}\label{fig:degen_examples}
\end{figure}

Figure~\ref{fig:degen_examples} gives examples of degenerating a join
curve and a split curve at east $\infty$, as well as breaking into a
two-story holomorphic building.

\begin{remark}
  In studying $\CFAa$, a third kind of curve at east $\infty$, called
  a \emph{shuffle curve}, is also important. See~\cite[Section~\ref*{LOT1:sec:curves_at_east_infinity}]{LOT1} for a discussion of shuffle curves.
\end{remark}

\begin{theorem}\label{thm:codim-1-degens}
  Suppose that $\ind(B;\rho_1,\dots,\rho_n)=2$. Then the ends of the
  moduli space $\cM^B(\w,\y;\rho_1,\dots,\rho_n)$ consist exactly of
  the following configurations:
  \begin{enumerate}
  \item\label{item:degen:2-story} Two-story holomorphic buildings, i.e., 
    \[
    \bigcup_{i=0}^n\,\bigcup_{\x\in\Gen(\HD)}\,\bigcup_{\substack{B_1\in\pi_2(\w,\x)\\B_2\in\pi_2(\x,\y)\\B_1*B_2=B}}\cM^{B_1}(\w,\x;\rho_1,\dots,\rho_i)\times\cM^{B_2}(\x,\y;\rho_{i+1},\dots,\rho_n).
    \]
  \item\label{item:degen:collapse} Collapses of levels, i.e., curves $u$ as in the definition of
    $\cM^B(\w,\y;\rho_1,\dots,\rho_n)$ except that the $t$-coordinates
    of $\rho_i$ and $\rho_{i+1}$ are equal. Moreover, either:
    \begin{enumerate}[label=(\ref*{item:degen:collapse}\alph*)]
    \item\label{item:collapse-1} the set of (one or two) $\alpha$-arcs
      containing $\bdy \rho_i$ must be disjoint from the set of (one
      or two) $\alpha$-arcs containing $\bdy \rho_{i+1}$, or
    \item\label{item:collapse-2} the initial endpoint of $\rho_i$ is
      the same as the final endpoint of $\rho_{i+1}$.
    \end{enumerate}
  \item\label{item:degen:join} Join curve degenerations, i.e., pairs
    $(u,v)$ where $u$ is a curve like those in
    \[
    \cM^B(\w,\y;\rho_1,\dots,\rho'_i,\rho''_i,\rho_{i+1},\dots,\rho_n)
    \]
    except that the $t$-coordinates of $\rho'_i$ and $\rho''_i$ are
    equal; and $v$ is a join curve with $w$ asymptotics $\rho_1,\dots,
    \rho'_i, \rho''_i,\dots , \rho_n$ and $e$ asymptotics
    $\rho_1,\dots, \rho_i ,\dots , \rho_n$. In particular,
    $\rho_i=\rho''_i\cup\rho'_i$. Moreover:
    \begin{itemize}
    \item The $\alpha$-arc containing the terminal end of $\rho''_i$
      is distinct from the $\alpha$-arcs containing the initial and
      terminal ends of $\rho_i$.
    \item The $t$-coordinates of the $w$ asymptotics of $v$ agree with
      the $t$-coordinates of the $e$ asymptotics of $u$.
    \end{itemize}
  \item\label{item:degen:split} Split curve degenerations, i.e., pairs $(u,v)$ where
    \[
    u\in\cM^B(\w,\y;\rho_1,\dots,\rho_i\cup\rho_{i+1},\dots,\rho_n)
    \]
    and $v$ is a split curve with $w$ asymptotics
    $\rho_1,\dots, (\rho_i\cup\rho_{i+1}) ,\dots ,
    \rho_n$ and $e$ asymptotics $\rho_1,\dots, \rho_i
    ,\rho_{i+1}, \dots , \rho_n$. Moreover, the
    $t$-coordinates of the $w$ asymptotics of $v$ agree with the
    $t$-coordinates of the $e$ asymptotics of $u$.

    In particular, the space of such pairs $(u,v)$ can be canonically
    identified with
    $\cM^B(\w,\y;\rho_1,\dots,\rho_i\cup\rho_{i+1},\dots,\rho_n)$.
  \end{enumerate}
\end{theorem}
This is a combination of \cite[Theorem~\ref*{LOT1:thm:master_equation}]{LOT1} and \cite[Lemma~\ref*{LOT1:lemma:collision-is-composable}]{LOT1}.

As in most of holomorphic curve theory, the key ingredients in the
proof of Theorem~\ref{thm:codim-1-degens} are:
\begin{itemize}
\item A transversality statement: for generic almost-complex
  structures, the relevant moduli spaces are transversally cut
  out. For curves in $\Sigma\times[0,1]\times\RR$ this
  is~\cite[Proposition~\ref*{LOT1:prop:transversality}]{LOT1}; for curves at east $\infty$, it
  is~\cite[Proposition~\ref*{LOT1:prop:east_transversality}]{LOT1}. Because we are not able to perturb the
  complex structure at east $\infty$, less transversality holds for
  curves at east $\infty$ than one might like. (Specifically, we can not always
  ensure that the evaluation maps at the punctures are transverse to
  the diagonal.)
\item A compactness statement: sequences of holomorphic curves in
  $\Sigma\times[0,1]\times\RR$ converge to holomorphic combs. This
  is~\cite[Proposition~\ref*{LOT1:prop:compactness}]{LOT1}.
\item Various gluing statements. Because of the Morse-Bott nature of
  the asymptotics at east $\infty$ and transversality issues for curves
  at east $\infty$, these statements become somewhat
  intricate. See~\cite[Section~\ref*{LOT1:sec:combs-gluing}]{LOT1}.
\item An analysis of which of the possible degenerations can occur in
  codimension-$1$. See~\cite[Sections~\ref*{LOT1:sec:degenerations-holomorphic-curves}
  and~\ref*{LOT1:sec:embedded-degen}]{LOT1}.
\end{itemize}
There is one more ingredient, because we are working with embedded
curves:
\begin{itemize}
\item A computation of the index of the $\overline{\bdy}$ operator
  shows that sequences of embedded curves converge to embedded
  curves. Philosophically, this is related to the adjunction
  formula. See~\cite[Section~\ref*{LOT1:sec:expected-dimensions}]{LOT1} for further discussion.
\end{itemize}

\begin{remark}
  The fact that $\pi_\bD$ is constant on each component of a curve at
  east $\infty$ suggests that we have lost some information in our
  formulation of the limiting objects. One could recover this
  information by rescaling while taking the limit. Specifically,
  suppose a sequence of holomorphic curves $u_i$ converges to a pair
  $(u,v)$, where $v\co T\to \RR\times(\bdy\Sigma)\times[0,1]\times\RR$
  is a curve at east $\infty$. Fix a marked point $p_i$ on each $u_i$
  converging to a marked point $p$ on $u$. In taking the limit,
  rescale the map $\pi_\bD\circ u_i$ on a neighborhood of $p_i$ so
  that $d_{p_i}(\pi_\bD\circ u_i)$ has norm $1$. With some work, one
  thus obtains a rescaled version of $\pi_\bD\circ v$ in the form of a map
  $T\to \{x+iy\in \CC\mid x\leq 1\}$. 

  The moduli spaces at east $\infty$ are sufficiently simple that this
  refined limiting procedure turns out not to be necessary to
  construct the bordered invariants; but it seems more relevant to
  constructing a bordered version of $\HF^\pm$.
\end{remark}

\section{\texorpdfstring{$\bdy^2=0$}{d-squared equals zero} on \texorpdfstring{$\CFDa$}{CFD}}
With the codimension-$1$ boundary in hand, we are now ready to prove
that $\CFDa$ is a \dg module. 

\begin{theorem}\cite[Proposition~\ref*{LOT1:prop:typeD-d2}]{LOT1}\label{thm:CFD-sq-0}
  Fix a provincially admissible bordered Heegaard diagram $\HD$. Then
  for a generic choice of almost-complex structure,
  the differential $\bdy$ on $\CFDa(\HD)$ satisfies $\bdy^2=0$.
\end{theorem}
\begin{proof}[Sketch of proof.]
  It suffices to show that for each generator $\w\in\Gen(\HD)$,
  $\bdy^2(\w)=0$. We have
  \begin{align*}
    \bdy^2(\w)&=\bdy\Bigl(
    \sum_{\substack{\y\in\Gen(\HD)\\(\rho_1,\dots,\rho_n)\\B\in\pi_2(\w,\y)}}\bigl(\#\cM^B(\w,\y;\rho_1,\dots,\rho_n)\bigr)a(-\rho_1)\cdots
    a(-\rho_n)\y \Bigr)\nonumber\\
    &=\sum_{\substack{\x\in\Gen(\HD)\\(\rho_1,\dots,\rho_i)\\B_1\in\pi_2(\w,\x)}}\sum_{\substack{\y\in\Gen(\HD)\\(\rho_{i+1},\dots,\rho_n)\\B_2\in\pi_2(\w,\x)}}
    \bigl(\#\cM^{B_1}(\w,\x;\rho_1,\dots,\rho_i)\bigr)\bigl(\#\cM^{B_2}(\w,\x;\rho_{i+1},\dots,\rho_n)\bigr)\\[-.4in]
    &\hspace{3in}
    \cdot a(-\rho_1)\cdots
    a(-\rho_i) a(-\rho_{i+1})\cdots a(-\rho_n)\y\nonumber\\[\baselineskip]
    &\qquad+
    \sum_{\substack{\x\in\Gen(\HD)\\(\rho_1,\dots,\rho_n)\\B\in\pi_2(\w,\x)}}\bigl(\#\cM^B(\w,\x;\rho_1,\dots,\rho_n)\bigr)a(-\rho_1)\cdots
    d(a_i)\cdots a(-\rho_n)\x.
  \end{align*}
  (There is some possibly confusing re-indexing: in the second line we have replaced $n\to i$, $\y\to
  \x$, and $B\to B_1$. In the last line we use the same notation as
  in the first line, however.)

  The sum in the second line corresponds exactly to the $2$-story
  holomorphic buildings, degeneration~(\ref{item:degen:2-story}) in
  Theorem~\ref{thm:codim-1-degens}.  The sum in the last line
  corresponds to the split curve degenerations,
  degeneration~(\ref{item:degen:split}) in
  Theorem~\ref{thm:codim-1-degens}. 

  It remains to see that the other ends of the $1$-dimensional
  moduli spaces cancel in pairs. Indeed, it is easy to see that
  Case~\ref{item:collapse-1} ends of
  $\cM^B(\w,\y;\rho_1,\dots,\rho_n)$ correspond to 
  Case~\ref{item:collapse-1} ends of
  $\cM^B(\w,\y;\rho_1,\dots,\rho_{i+1},\rho_i,\dots,\rho_n)$; and
  Case~\ref{item:collapse-2} ends of
  $\cM^B(\w,\y;\rho_1,\dots,\rho_n)$ correspond to join curve ends of
  $\cM^B(\w,\y;\rho_1,\dots,\rho_i\cup\rho_{i+1},\dots,\rho_n)$. 
  This completes the proof.
\end{proof}

\section{Deforming the diagonal, \texorpdfstring{$\CFAa$}{CFA} and the pairing theorem}\label{sec:prove-pairing}
Our goals for the rest of the lecture are two-fold:
\begin{enumerate}
\item Define the invariant $\CFAa(Y)$ associated to a bordered $3$-manifold.
\item Prove the pairing theorem, 
Theorem~\ref{thm:pairing1}.
\end{enumerate}
We will do this in the opposite order: we will start proving
Theorem~\ref{thm:pairing1}, and $\CFAa$ will appear naturally. The
material in this section is drawn from~\cite[Chapter~\ref*{LOT1:chap:tensor-prod}]{LOT1}, to which
we refer the reader for further details.

\begin{figure}
  \begin{overpic}[tics=10]{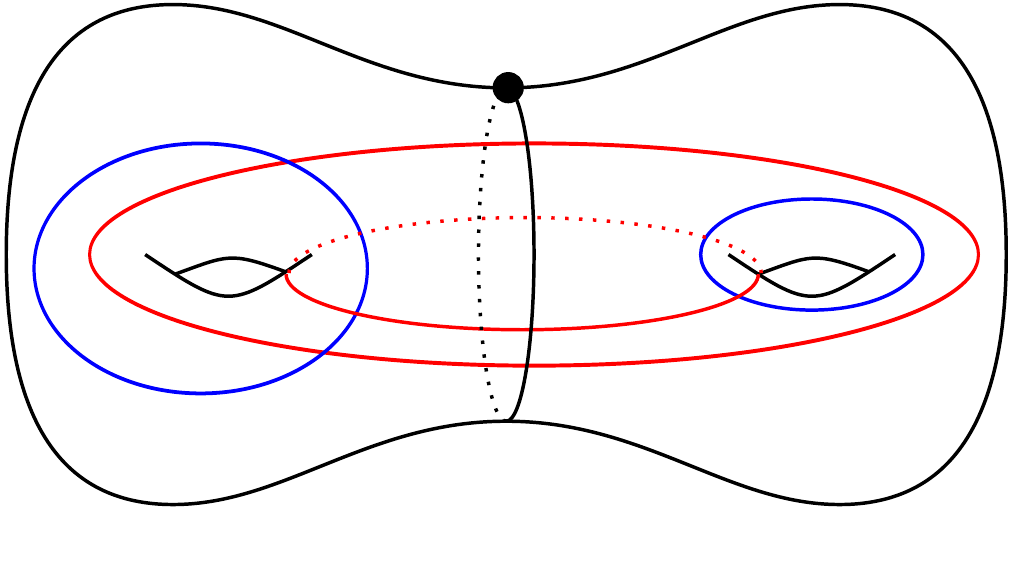}
    \put(48,10){$Z$}
    \put(20,2){$\HD_1$}
    \put(78,2){$\HD_2$}
    \put(48,0){$\HD$}
    \put(31,18){$a$}
    \put(29,41){$b$}
    \put(72,23){$s$}
    \put(49,49){$z$}
  \end{overpic}
\caption{\textbf{Splitting a closed Heegaard diagram.} The bordered Heegaard diagrams $\HD_1$ and $\HD_2$ are glued along the circle $Z\subset \HD$.\label{fig:split-HD}}
\end{figure}
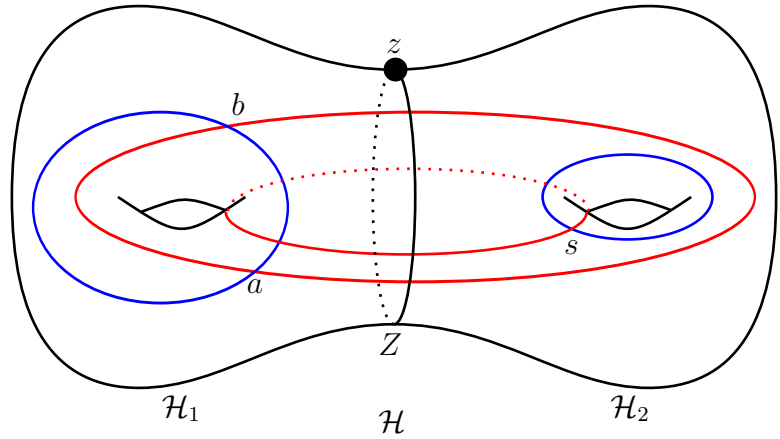

So, fix bordered Heegaard diagrams $\HD_1$, $\HD_2$ with $\bdy\HD_1=\PMC=-\bdy\HD_2$ and let $\HD=\HD_1\cup_\bdy\HD_2$. (See Figure~\ref{fig:split-HD}.) We want to understand $\CFa(\HD)$ in terms of invariants of $\HD_1$ and $\HD_2$. 

On the level of generators, this is trivial: a generator
$\x\in\CFa(\HD)$ corresponds to a pair of generators $(\x_1,\x_2)$ for
$\HD_1$ and $\HD_2$ so that the $\alpha$-arcs occupied by $\x_1$ are
complementary to the $\alpha$-arcs occupied by $\x_2$. So, if we
define $I_A(\x_1)$ to be the idempotent in $\Alg(\PMC)$ corresponding
to the $\alpha$-arcs occupied by $\x_1$---this is the opposite of
$I(\x)$ as defined in Section~\ref{sec:def-CFD}---and let
\[
\CFAa(\HD_1)=\Field\langle \Gen(\HD_1)\rangle,
\]
with $I_A(\x_1)\x_1=\x_1$, so other indecomposable idempotents
kill $\x_1$, then we have
\begin{equation}\label{eq:CFA-tensor-CFD}
\CFAa(\HD_1)\otimes_{\Alg(\PMC)}\CFDa(\HD_2)
\end{equation}
as $\Field$-vector spaces.
Note that we have \emph{not} defined an $\Alg(\PMC_1)$-module structure on $\CFAa(\HD_1)$ yet: Equation~\ref{eq:CFA-tensor-CFD} uses only the action of the idempotents and the fact that $\CFDa(\HD_2)$ is a sum of elementary projective modules.

Holomorphic curves are more complicated.

Let $Z\subset \HD$ denote the circle $\bdy\HD_1$. Recall that to define $\CFDa(\HD_2)$ we attached a cylindrical end to $-Z=\bdy\Sigma_2$. Correspondingly, to prove the pairing theorem, we consider inserting a long neck into $\Sigma$ along $Z$. That is, fix a complex structure $j_\Sigma$ on $\Sigma$ and choose a neighborhood $U$ of $Z$ which is biholomorphic to $[-\epsilon,\epsilon]\times S^1$ for some $\epsilon>0$. Let $j_\Sigma^R$ denote the result of replacing $U$ by $[-R,R]\times S^1$.

Let $R_i\in\RR$ be a sequence with $R_i\to\infty$, and suppose
$u_i\in\cM^B(\x,\y)$ is a sequence of holomorphic curves with respect
to $j_\Sigma^{R_i}\times j_\bD$. We are interested in the limit of the
sequence $\{u_i\}$. Modulo some technicalities, this is the kind of
limit studied in symplectic field theory; the limiting objects have
the following form:
\begin{definition}
A \emph{matched holomorphic curve} is a pair of curves
\[
(u_1,u_2)\in
\cM^{B_1}(\x_1,\y_1;\rho_1,\dots,\rho_n)\times\cM^{B_2}(\x_2,\y_2;\rho_1,\dots,\rho_n)
\]
so that for each $i=1,\dots,n$, the $t$-coordinate at which $u_1$ is
asymptotic to $\rho_i$ is equal to the $t$-coordinate at which $u_2$
is asymptotic to $\rho_i$.

Equivalently, there is an evaluation map
\[
\ev\co \cM^{B_i}(\x_i,\y_i;\rho_1,\dots,\rho_n)\to \RR^{n-1}
\]
which takes a curve asymptotic to
$\rho_1\times(1,t_1),\dots,\rho_n\times(1,t_n)$ to
$(t_2-t_1,t_3-t_2,\dots,t_n-t_{n-1})$. Then a matched holomorphic
curve is a pair $(u_1,u_2)$ such that $\ev(u_1)=\ev(u_2)$.

Let $\cM^B(\x,\y;\infty)$ denote the moduli space of matched
holomorphic curves in the homology class $B$. That is,
\begin{equation}\label{eq:matched-curve}
\cM^B(\x,\y;\infty)=
  \bigcup_{(\rho_1,\dots,\rho_n)}\cM^{B_1}(\x_1,\y_1;\rho_1,\dots,\rho_n)\sos{\ev}{\times}{\ev}\cM^{B_2}(\x_2,\y_2;\rho_1,\dots,\rho_n).
\end{equation}
Here, $\x$ (respectively $\y$) corresponds to the pair of generators
$(\x_1,\x_2)$ (respectively $(\y_1,\y_2)$) and $B_i$ is the
intersection of $B$ with $\HD_i$.
\end{definition}

\begin{proposition}
 Let $\cM^B(\x,\y;R)$ denote the moduli space of holomorphic curves
 (in $\Sigma\times[0,1]\times\RR$, in the homology class $B$)
 with respect to an appropriate perturbation\footnote{As usual, we
   will suppress the fact that one needs to perturb the almost-complex
   structure in order to achieve transversality from the discussion.}
 of the almost-complex
 structure $j_\Sigma^R\times
 j_\bD$. Suppose that $\mu(B)=1$. Then $\bigcup_{R>0}\cM^B(\x,\y;R)$
 is a $1$-manifold whose ends as $R\to\infty$ are identified with
 $\cM^B(\x,\y;\infty)$. More precisely, let
\[
\cM^B(\x,\y;\geq\! R_0)=
\cM^B(\x,\y;\infty)\cup\bigcup_{R\geq R_0}\cM^B(\x,\y;R)
\]
Then there is a there is a topology on $\cM^B(\x,\y;\geq R_0)$ and an $R_0$ so that $\cM^B(\x,\y;\geq R_0)$
is a compact $1$-manifold with boundary exactly
\[
\cM^B(\x,\y;\infty)\amalg\cM^B(\x,\y;R_0).
\]
\end{proposition}
This follows from compactness and gluing arguments, in a fairly
standard way.

\begin{corollary}\label{cor:splitting-1}
  Define $\bdy_1\co \CFa(\HD)\to \CFa(\HD)$ by 
  \begin{equation}\label{eq:bdy-prime}
\bdy_1(\x)=\sum_{\y\in T_\alpha\cap T_\beta}\sum_{\substack{B\in\pi_2(\x,\y)\\\mu(B)=1}}\#\cM^B(\x,\y;\infty) \y
  \end{equation}
  (cf.~Formula~\eqref{eq:CF-d}).
  Then $H_*(\CFa(\HD),\bdy_1)\cong \HFa(Y)$.
\end{corollary}

\begin{example}\label{eg:split-HD-1} 
  Consider the splitting in
  Figure~\ref{fig:split-HD}. The complex $\CFa(\HD)$ has two
  generators, $\x=\{a,s\}$ and $\y=\{b,s\}$; in the notation above,
  $\x_1=\{a\}$, $\x_2=\{s\}$, $\y_1=\{b\}$ and $\y_2=\{s\}$. The
  generator $\y$ occurs twice in $\bdy(\x)$: once from the small bigon
  region near the left of the diagram and once from the annular region
  crossing through the circle $Z$. We focus on the second of
  these contributions, the domain of which is shown in
  Figure~\ref{fig:split-domain}. (It takes a little work to show that
  this domain has a holomorphic representative; see
  Exercise~\ref{ex:annulus}.)
  
  Now, consider the result of stretching the neck along $Z$. There are
  two cases, depending on whether the cut goes through $Z$ or not
  (which in turn depends on the complex structure on~$\HD$). If the
  cut does not go through $z$, the resulting matched curve $(u_1,u_2)$
  has $u_1$ a disk with one Reeb chord and $u_2$ an annulus with one
  Reeb chord. (In fact, this case does not occur in the limit; see
  Exercise~\ref{ex:cut-bad-case}.)

  The more interesting case---and the one which actually occurs---is
  when the cut does pass through $Z$. Then both $u_1$ and $u_2$ are
  disks with two Reeb chords on each of their boundaries. The disk
  $u_2$ is rigid, but the disk $u_1$ comes in a $1$-parameter family,
  depending on the length of the cut. There is algebraically one
  length of cut for which the height difference of the two Reeb chords
  in $u_1$ agrees with the height difference of the Reeb chords in
  $u_2$ (Exercise~\ref{ex:algebraically-one}).
\end{example}

\begin{figure}
  \centering
  \begin{overpic}[tics=10, scale=1.5]{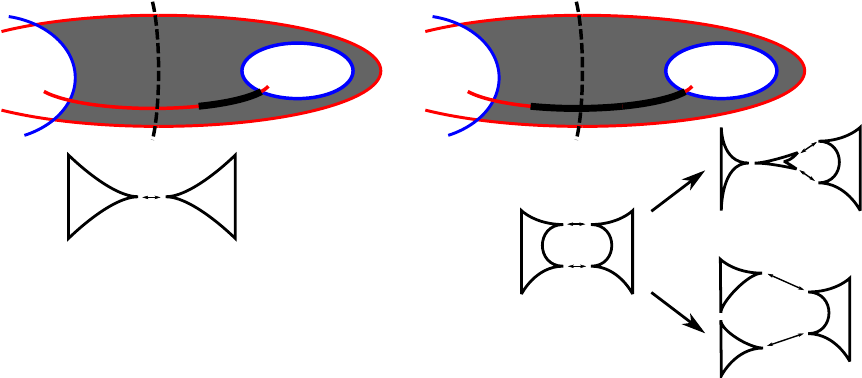}
    \put(67,22){$T\to\infty$}
    \put(68,5){$T\to0$}
  \end{overpic}
  \caption{\textbf{Splitting an interesting domain.} Depending on the
    complex structure, there are two possible phenomena after
    splitting: either the cut stays entirely on the right side of the
    diagram, as in the left picture, or the cut runs through the
    collapsed circle $Z$, as in the right picture. We have drawn
    schematic illustrations of the matched holomorphic curves below
    the two pictures.}
  \label{fig:split-domain}
\end{figure}

Corollary~\ref{cor:splitting-1} is a step in the direction
of a pairing theorem: it
gives a definition of $\HFa$ in terms of holomorphic curves in
$\Sigma_1\times[0,1]\times\RR$ and $\Sigma_2\times[0,1]\times\RR$. But
as we saw in Example~\ref{eg:split-HD-1}, the corollary still has two
(related) drawbacks:
\begin{enumerate}
\item The moduli spaces we are considering in for $\HD_1$ and $\HD_2$
  are typically high\hyph{}dimensional. Indeed, in
  Formula~(\ref{eq:bdy-prime}), we have
  \[
  \dim\cM^{B_1}(\x_1,\y_1;\rho_1,\dots,\rho_n)+
  \dim\cM^{B_2}(\x_2,\y_2;\rho_1,\dots,\rho_n)=n-1.
  \]
\item\label{item:not-independent} Since we are taking a fiber product
  of moduli spaces, which curves we want to consider in $\HD_1$
  depends on $\HD_2$. So, it is not yet obvious how to define
  independent invariants of $\HD_1$ and $\HD_2$ containing the
  information needed to compute $\bdy_1$.
\end{enumerate}
To address complaint~(\ref{item:not-independent}) we could try to
formulate an algebra which remembers the chain
$\ev_*[\cM^{B_1}(\x_1,\y_1;\rho_1,\dots,\rho_n)]\in
C_*(\RR^{n-1})$. 
This is a natural way to try to define
a bordered Heegaard Floer invariant, and with enough effort it could
probably be made to work. This approach would be far from
combinatorial, and is also unnecessarily complicated, as we will now
show.

The next step is to deform the fiber product in
Formula~\eqref{eq:matched-curve}:
\begin{definition}
  A \emph{$T$-matched holomorphic curve} is a pair 
  \[
  (u_1,u_2)\in
  \cM^{B_1}(\x_1,\y_1;\rho_1,\dots,\rho_n){\times}\cM^{B_2}(\x_2,\y_2;\rho_1,\dots,\rho_n)
  \]
  such that $T\cdot \ev(u_1)=\ev(u_2)$. Let $\cM^B_T(\x,\y;\infty)$
  denote the moduli space of $T$-matched holomorphic curves, i.e., 
  \[
  \cM^B_T(\x,\y;\infty)=
  \bigcup_{(\rho_1,\dots,\rho_n)}
\cM^{B_1}(\x_1,\y_1;\rho_1,\dots,\rho_n)\sos{T\cdot\ev}{\times}{\ev}\cM^{B_2}(\x_2,\y_2;\rho_1,\dots,\rho_n).
  \]
\end{definition}
So, in particular, a $1$-matched holomorphic curve is just a matched
holomorphic curve.

A standard continuation-map argument shows:
\begin{proposition}
  Let $\bdy_T$ denote the map defined analogously to
  Formula~\ref{eq:CF-d} (or Formula~\ref{eq:bdy-prime}) but using the
  moduli spaces $\cM^B_T(\x,\y;\infty)$. Then
  $H_*(\CFa(\HD),\bdy_T)\cong \HFa(Y)$.
\end{proposition}

Now, of course, we send $T\to \infty$. Consider a sequence of
$T_i$-matched curves $(u_1^i,u_2^i)$ with $T_i\to\infty$. Suppose that
$u_1^i\in\cM^{B_1}(\x_1,\y_1;(\rho_1,\dots,\rho_n))$.  Let $s_j^i$ be
the $\RR$--coordinate at which $u_1^i$ is asymptotic to $\rho_j$ and
let $t_j^i$ be the $\RR$--coordinate at which $u_2^i$ is asymptotic to
$\rho_j$. Then, after passing to a subsequence, for each $\rho_j$,
either:
\begin{itemize}
\item $(s_{j+1}^i-s_j^i)\in(0,\infty)$ stays bounded away from $0$ and
  $(t_{j+1}^i-t_j^i)\to\infty$ as $i\to\infty$; or
\item $(s_{j+1}^i-s_j^i)\to 0$ and $(t_{j+1}^i-t_j^i)$ stays bounded
  as $i\to\infty$.
\end{itemize}
So, in the limit:
\begin{itemize}
\item On the right we have an $\ell$-story holomorphic building (for
  some $\ell$) $U_2^\infty=(v_1,\dots,v_\ell)$, where
  $v_j\in\cM(\x_{1,j},\x_{1,j+1};\rho_{n_j},\dots,\rho_{n_{j+1}})$,
  $\x_{1,j}=\x_1$, $\x_{1,\ell+1}=\y_1$, $1=n_1\leq n_2\leq\dots\leq n_{\ell+1}=n$.
\item On the left we have a curve $u_1^\infty$ asymptotic to
  some sets of Reeb chords $\rhos_1,\dots,\rhos_\ell$ at
  $t$-coordinates $t_1<\dots<t_\ell\in\RR$. Let 
  \[
  \cM^{B_2}(\x_2,\y_2;\rhos_1,\dots,\rhos_\ell)
  \]
  denote the moduli space of such curves.
\end{itemize}
Importantly, there is no longer a matching condition between the
curves $u_1^\infty$ and $U_2^\infty$.

\begin{example}
  Continuing with Example~\ref{eg:split-HD-1} in the case that the cut
  goes through the neck, as on the right of
  Figure~\ref{fig:split-domain}, as $T\to\infty$ the $\RR$-coordinates
  of the two Reeb chords in $u_1$ come together. (This results in
  degenerating a split curve at $\bdy\Sigma$; we elided this point in
  the rest of this section.) This is indicated schematically in Figure~\ref{fig:split-domain}.
  
  Now, suppose we turned the diagram $180^\circ$. To avoid re-drawing
  the figure, we can think of this as sending $T\to 0$ instead of
  $T\to\infty$. In this case, the two chords in
  Figure~\ref{fig:split-domain} are pushed farther and farther apart;
  in the limit, the cut goes all the way through to the $\beta$-curve,
  giving a $2$-story holomorphic building. Again, this is indicated schematically in Figure~\ref{fig:split-domain}.
  
  Observe that in both cases, the relevant curves are completely
  determined, i.e., belong to rigid moduli spaces: there is no ``cut''
  left.
\end{example}

Now, associated to a \emph{set} of Reeb chords $\rhos$ is an algebra
element $a(\rhos)$, defined analogously to
Equation~\eqref{eq:a-of-rho}; see Exercise~\ref{ex:a-of-rhos}
or~\cite[Definition~\ref*{LOT1:def:arhos}]{LOT1}. Define maps
\begin{align*}
  m_{i+1}\co \CFAa(\HD_1)\otimes\Alg(\PMC)^{\otimes i}&\to
  \CFAa(\HD_1)\\
  m_{i+1}(\x;a(\rhos_1),\dots,a(\rhos_i)) &= \sum_{\y\in\Gen(\HD_1)}\sum_{\substack{B\in\pi_2(\x,\y)\\\ind(B,\rhos_1,\dots,\rhos_i)=1}}\bigl(\#\cM^B(\x,\y;\rhos_1,\dots,\rhos_i)\bigr)\y.
\end{align*}

An argument similar to but in some ways easier than the proof of
Theorem~\ref{thm:CFD-sq-0} proves:
\begin{theorem}
  For $\HD_1$ a provincially-admissible Heegaard diagram and $J$ a
  generic almost-complex structure, the operations $m_{i+1}$ make
  $\CFAa(\HD_1)$ into an $\Ainf$-module.
\end{theorem}

The argument above is a sketch of the pairing theorem,
Theorem~\ref{thm:pairing1}. Specifically, it follows from the sketch
above that
\[
\CFa(\HD)\simeq \CFAa(\HD_1)\DT\CFDa(\HD_2),
\]
where $\DT$ is the model for the tensor product of an $\Ainf$-module
with a type $D$ structure described in~\cite[Section~\ref*{LOT1:sec:DT}]{LOT1}.

\section{Exercises}

\begin{exercise}\label{ex:Riemann}
  In the setting of Section~\ref{sec:broken-cyl}, use the Riemann
  mapping theorem to show that the map $\pi_\Sigma\circ u$ is
  determined by the position of the branch point (as claimed), and that
  there are no other elements of $\cM(a,d)$.
\end{exercise}

\begin{exercise}
  Suppose that $v\co S\to \RR\times(\bdy\Sigma)\times[0,1]\times\RR$ is
  a holomorphic curve at east $\infty$, as discussed in
  Section~\ref{sec:codim-1-bdy}. Show that the restriction of
  $\pi_\bD\circ v$ to each component of $S$ is constant.
\end{exercise}

\begin{exercise}
  Prove: If $\x$ is a generator for $\CFDa(Y)$, where $\bdy Y=F(\PMC)$
  then $I(\x)\in \Alg(\PMC,0)\subset\Alg(\PMC)$. (Hint: this is easy.)
  What is the corresponding statement for the bimodules $\CFDDa$
  associated to arced cobordisms?
\end{exercise}

\begin{exercise}
  The differential on the algebra $\Alg(T^2,0)$ associated to the
  torus is trivial. This means that one of the cases in the proof of
  Theorem~\ref{thm:CFD-sq-0} does not arise if the boundary is a
  torus. Which one? Why?
\end{exercise}

\begin{exercise}\label{ex:annulus}
  Show that the annular region in Figure~\ref{fig:split-domain}
  is the domain of a holomorphic map $S\to\Sigma\times[0,1]\times\RR$,
  in two ways:
  \begin{enumerate}
  \item By adapting the argument from~\cite[Lemma
    9.3]{OS04:HolomorphicDisks}.
  \item By using handleslide invariance of Heegaard Floer homology.
    (After performing the right handleslide on
    Figure~\ref{fig:split-HD}, it is easy to compute $\HFa$.)
  \end{enumerate}
\end{exercise}

\begin{exercise}\label{ex:cut-bad-case}
  Show that when one stretches the neck in Figure~\ref{fig:split-HD},
  as in Example~\ref{eg:split-HD-1}, the domain in
  Figure~\ref{fig:split-domain} must have a cut passing through the
  neck.
\end{exercise}

\begin{exercise}\label{ex:algebraically-one}
  In Example~\ref{eg:split-HD-1} we claimed there is algebraically one
  length of cut so that the height difference of the two Reeb chords
  in $u_1$ agrees with the height difference of the two Reeb chords in
  $u_2$. (Since we are working with $\Field$-coefficients, probably we
  really meant that there are an odd number of such cut lengths.)
  Prove this. (Hint: what is the height difference in $u_1$ when the
  cut has length $0$? When the cut goes all the way to the
  $\beta$-circle?)
\end{exercise}

\begin{figure}[t]
\centerline{\input{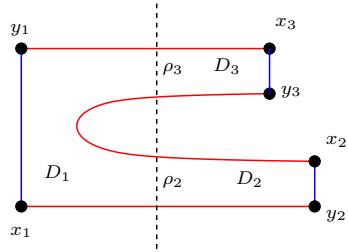}}
\caption[Degenerating the hexagon.]{\label{fig:Hexagon}
  {\bf{Degenerating the hexagon.}} 
  A hexagon in the Heegaard diagram 
  (giving a flow from $\x=\{x_1,x_2,x_3\}$ to $\y=\{y_1,y_2,y_3\}$)
  is divided into three pieces $D_1$, $D_2$, and $D_3$,
  grouped as $D_1$ and $D_2\cup D_3$. This is~\cite[Figure~\ref*{LOT1:fig:Hexagon}]{LOT1}.
  }
\end{figure}

\begin{exercise}
  Figure~\ref{fig:Hexagon} shows a hexagonal domain connecting
  $\x=\{x_1,x_2,x_3\}$ to $\y=\{y_1,y_2,y_3\}$. Note that this domain
  always contributes a term of $\y$ in $\bdy(\x)$. Consider the result
  of degenerating this domain along the dashed line, and then
  deforming the diagonal as in Section~\ref{sec:prove-pairing}.  (In
  the notation of Section~\ref{sec:prove-pairing}, consider both the
  case of sending $T\to \infty$ and the case of sending $T\to 0$.)
  What happens to the holomorphic representative for this domain in
  the process? How is this encapsulated algebraically?
  (See~\cite[Section~\ref*{LOT1:sec:tensor-prod-eg}]{LOT1} for a detailed discussion of this example.)
\end{exercise}

\begin{exercise}\label{ex:a-of-rhos}
  Define $a(\rhos)\in\Alg(\PMC)$ when $\rhos$ is a set of chords in
  $\PMC$, no two of which start (respectively end) at points from the
  same matched pair. (This is a generalization of
  Formula~\eqref{eq:a-of-rho}, and should be
  straightforward. See~\cite[Definition~\ref*{LOT1:def:arhos}]{LOT1} for a solution.)
\end{exercise}

\chapter[Knot complements]{Computing with bordered Floer homology I: knot complements}\label{lec:torus}
In this section we will discuss how the torus boundary case of
bordered Floer homology can be used to do certain kinds of
computations. The main goal is a technique for studying satellite
knots, from~\cite[Chapter~\ref*{LOT1:chap:TorusBoundary}]{LOT1}. This technique and extensions of
it have been used in~\cite{Levine12:slicingBing,Levine12:doubling,Petkova:cableofthin,Hom:bord-tau}.

We start with a review of knot Floer
homology~\cite{OS04:Knots,Rasmussen03:Knots}, mainly to fix notation
(Section~\ref{sec:knot-Floer}). We then discuss how the knot Floer
homology of a knot $K$ in $S^3$ determines the bordered Floer homology of
$S^3\setminus K$ (Section~\ref{sec:CFK-to-CFD}). Finally, we turn
this around to use our understanding of bordered Floer homology to
study the knot Floer homology of satellites (Section~\ref{sec:satellites}).

\section{Review of knot Floer homology}\label{sec:knot-Floer}

Let $K$ be a knot in $S^3$, and let $\HD=(\Sigma,\alphas,\betas,z,w)$
be a \emph{doubly pointed Heegaard diagram} for $K$, in the sense
of~\cite{OS04:Knots}. (For example, a doubly pointed Heegaard diagram
for the trefoil is shown in Figure~\ref{fig:trefoil}.) Associated to
$\HD$ are various knot Floer homology groups. The most general of
these is $\CFK^-(K)$, which is a filtered chain complex over
$\Field[U]$. The complex $\CFK^-(K)$ is freely generated (over
$\Field[U]$) by $T_\alpha\cap T_\beta$, the same generators as
$\CFa(\Sigma,\alphas,\betas)$. The differential is given by
\[
\bdy^-(\x)=\sum_{\y}\,
  \sum_{\substack{B\in\piBig(\x,\y)\\ \mu(B)=1}}
\#\bigl(\Mod^B(\x,\y)\bigr) U^{n_{w}(B)}\cdot \y.
\]
Here, unlike the discussion above, we allow disks to cross the
basepoint $z$; we have used the notation $\piBig(\x,\y)$ rather than
$\pi_2(\x,\y)$ to indicate this. 

The complex $\CFK^-(K)$ has an integral grading, called the {\em Maslov
  grading}, which is decreased by one by the differential. We will
make no particular reference to this additional structure in the
present notes; but it will be convenient (for the purposes of taking
Euler characteristic, cf. Equations~\eqref{eq:EulerCFa}
and~\eqref{eq:EulerCFm} below) to have its parity, as
encoded in $(-1)^{M(\x)}$. This parity is given as the local
intersection number of $T_\alpha$ and $T_\beta$ at $\x$. (As defined,
we have specified a function $\Gen(\HD)\to\{\pm 1\}$ which is
well-defined up to overall sign.) Now, the fact that $\bdy^-$
respects this parity is equivalent to the the statement that if
$B\in\piBig(\x,\y)$ has $\Mas(B)=1$, then the local intersection
numbers of $T_\alpha$ and $T_{\beta}$ at $\x$ and $\y$ are
opposite.

The complex $\CFK^-(K)$ has an \emph{Alexander filtration} which is
uniquely determined
up to translation by
\[
\begin{aligned}
  A(\y)-A(\x)&=n_{w}(B)-n_{z}(B)\\
  A(U\cdot\y) &= A(\y) - 1
\end{aligned}
\]
where $B\in \pi_2(\x,\y)$. In other words, a term of the form
$U^{n_w(B)}\y$ in $\bdy^-(x)$ has $A(U^{n_w(B)}\y)=A(\x)-n_z(B)$. 

Let $\gCFKm(K)$ denote the associated graded complex to
$(\CFK^-(K),A)$. Explicitly, the differential on $\gCFKm(K)$ is
defined in the same way as the differential on $\CFK^-(K)$ except that
we no longer allow holomorphic curves to cross the $z$ basepoint.
Thus, the chain complex $\gCFKm$ splits as a direct sum of complexes,
determined by the Alexander grading:
\[ \gCFKm(K)=\bigoplus_{s\in\ZZ}\gCFKm(K,s).\]

Finally, there is the complex $\CFKa(K)$ obtained from $\gCFKm(K)$ by
setting $U=0$. In other words, $\CFKa(K)$ is generated over $\Field$
by $T_\alpha\cap T_\beta$, and the differential counts holomorphic
curves which do not cross $z$ or $w$.  Like $\gCFKm$, $\CFKa$ has a
direct sum splitting induced by the Alexander grading.

A key property of knot Floer homology is that its graded Euler
characteristic is the Alexander polynomial:
\begin{equation}
  \label{eq:EulerCFa}
 \Delta_K(T)=\sum_{s\in\ZZ}\chi(\CFKa(K,s)) T^s; 
\end{equation}
and similarly,
\begin{equation}
  \label{eq:EulerCFm}
 \Delta_K(T)/(1-T)=\sum_{s\in\ZZ}\chi(\CFKm(K,s)) T^s. 
\end{equation}
(Note that the parity of the Maslov grading is used to compute the Euler
characteristic. Also, both sides of Formula~(\ref{eq:EulerCFm}) are
formal power series.)

The translation indeterminacy in the Alexander grading can then be removed
by requiring the graded Euler characteristic of $\CFKa$ to be the
Conway normalized Alexander polynomial (or equivalently
$\chi(\CFKa(K,s))=\chi(\CFKa(K,-s))$ for all $s\in\ZZ$); this
normalization can also be used to
remove the overall indeterminacy in the parity of the Maslov grading.

There is a numerical invariant for knots derived from knot Floer
homology, $\tau(K)$, which will appear in Theorem~\ref{thm:CFK-to-CFD}
below.  This is defined with the help of the following observation.
There are $U$-non-torsion elements in $H_*(\gCFKm(K,s))$, i.e.,  elements
$h\in H_*(\gCFKm(K,s))$ with the property that for all positive
integers $m$, $U^m h$ is homologically non-trivial.  We can consider the maximal $s$
for which $H_*(\gCFKm(K,s))$ contains $U$-non-torsion
elements. Multiplying this $s$ by $-1$ gives the invariant $\tau(K)$.

\begin{figure}
  \centering
  \begin{overpic}[tics=10, width=.8\textwidth]{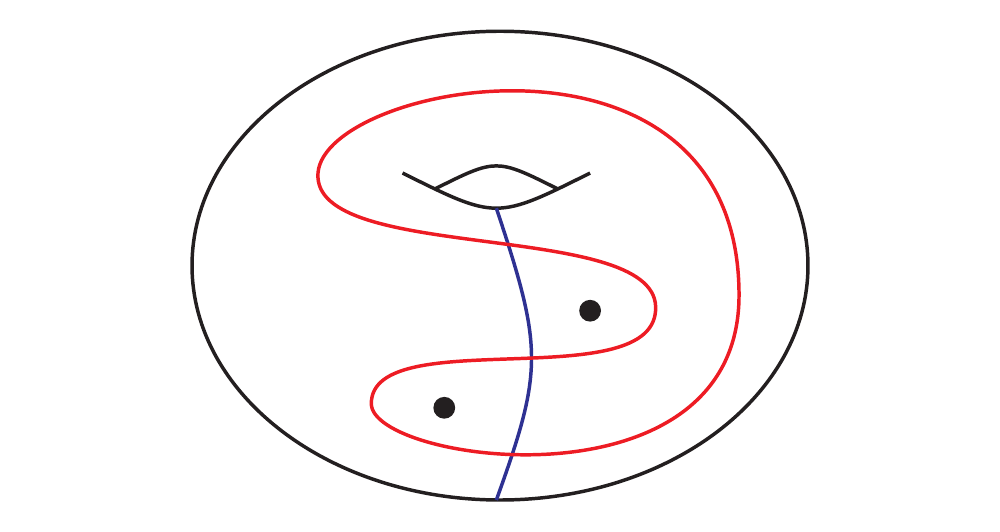}
    \put(52,5){$a$}
    \put(54,14.5){$b$}
    \put(51.5,29.5){$c$}
    \put(46,12){$z$}
    \put(55.5,22){$w$}    
    \put(48,0){\textcolor{blue}{$\beta$}}
    \put(29,35){\textcolor{red}{$\alpha$}}
  \end{overpic}
  \caption{\textbf{Doubly pointed Heegaard diagram for the trefoil.}}
  \label{fig:trefoil}
\end{figure}
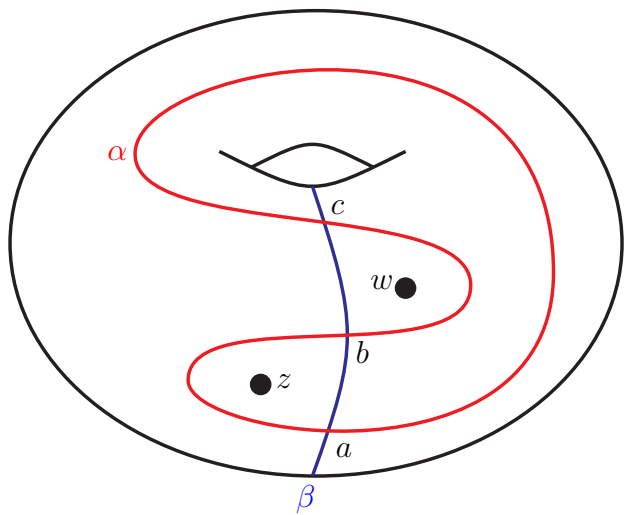

\begin{example}\label{eg:trefoil-1}
  Figure~\ref{fig:trefoil} shows a doubly-pointed Heegaard diagram for
  the trefoil knot. The chain complex $\CFK^-(\HD)$ is given by
  $\Field[U]\langle a,b,c\rangle$. The differential on $\CFK^-(\HD)$
  is given by
  \begin{align*}
    \bdy^-(a)&=b &
    \bdy^-(b)&=0 &
    \bdy^-(c)&=U b.
  \end{align*}
  The Alexander filtration is given by $A(a)=1$, $A(b)=0$, $A(c)=-1$.

  The differential on $\gCFKm(\HD)$ is given by
  \begin{align*}
    \bdy^-_{g}(a)&=0 &
    \bdy^-_g(b)&=0 &
    \bdy^-_g(c)&=U b.
  \end{align*}

  The complex $\CFKa(\HD)$ is $\Field\langle a,b,c\rangle$, with
  trivial differential.
\end{example}

\begin{figure}
  \centering
  \begin{overpic}[tics=10, width=.8\textwidth]{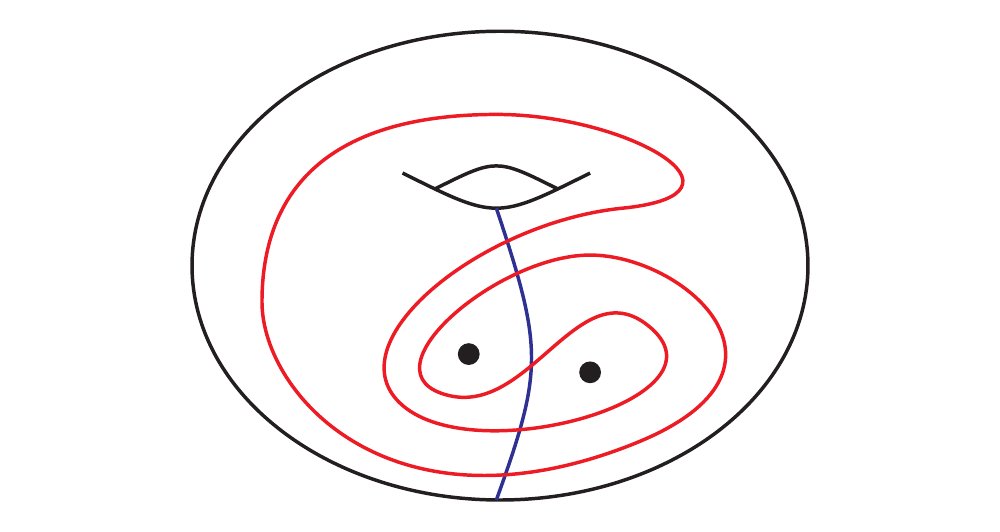}
    \put(49,6.5){$a$}
    \put(53,11){$b$}
    \put(51,17){$d$}
    \put(52.5,24.5){$c$}
    \put(48.5,29.5){$e$}
    \put(46,19.5){$w$}
    \put(60,17){$z$} 
    \put(48,0){\textcolor{blue}{$\beta$}}
    \put(23,25){\textcolor{red}{$\alpha$}}
  \end{overpic}
  \caption{\textbf{Doubly pointed Heegaard diagram for the figure
      eight knot}}
  \label{fig:fig8}
\end{figure}
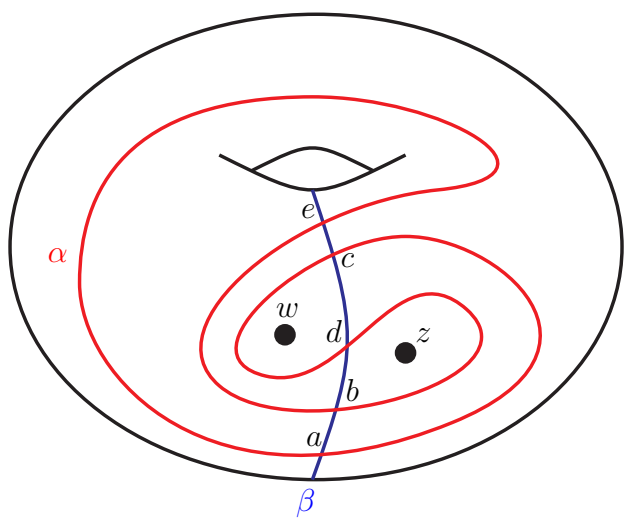

Another concrete example is furnished by the Figure 8 knot.

\begin{example}\label{eg:fig8}
  Figure~\ref{fig:fig8} shows a doubly-pointed Heegaard diagram for
  the figure eight knot. The chain complex $\CFK^-(\HD)$ is given by
  $\Field[U]\langle a,b,c,d,e\rangle$. The differential on $\CFK^-(\HD)$
  is given by
  \begin{align*}
    \bdy^-(a)&=U b + c&
    \bdy^-(b)&=d &
    \bdy^-(c) &= U d &
    \bdy^-(d) &= 0 &
    \bdy^-(e)&=U b+c.
  \end{align*}
  The Alexander filtration is given by $A(a)=A(d)=A(e)=0$, $A(b)=1$, $A(c)=-1$.

  The differential on $\gCFKm(\HD)$ is given by
  \begin{align*}
    \bdy_g^-(a)&=U b&
    \bdy_g^-(b)&=0 &
    \bdy_g^-(c) &= U d &
    \bdy_g^-(d) &= 0 &
    \bdy_g^-(e)&=U b.
  \end{align*}

  The complex $\CFKa(\HD)$ is $\Field\langle a,b,c,d,e\rangle$, with
  trivial differential.
\end{example}

We represent the chain complex $\CFK^-(\HD)$ graphically by choosing a
basis $\{\xi_i\}$ for $\CFK^-(\HD)$ over $\Field[U]$---for instance,
the standard basis whose elements are points in $T_\alpha\cap T_\beta$---and
placing a generator of the form $U^{-x} \cdot \xi_i$ with Alexander
depth~$y$ on the plane at the position $(x,y)$. Then the differential
of a generator at $(x,y)$ can be represented graphically by arrows
connecting the point at $(x,y)$ with the coordinates of other
generators.  These arrows necessarily point (non-strictly) to the left
and down.

Up to filtered homotopy equivalence, we can always ensure that the
differentials in the chain complex $\CFK^-(\HD)$ change the Alexander grading or the $U$
power, or both; we call a chain complex \emph{reduced} if it has this
property. Equivalently, $\CFK^-(\HD)$ is reduced if every arrow
changes the $x$-coordinate or the $y$-coordinate or both. A reduced
complex has two distinct kinds of lowest-order terms: horizontal
arrows and vertical arrows. We call the basis $\{\xi_i\}$
\emph{horizontally simplified} (respectively \emph{vertically
  simplified}) if every element $U^j\xi_i$ is the tail of at most one
horizontal (respectively vertical) arrow and the head of at most one
horizontal (respectively vertical) arrow. It is reasonably
straightforward to verify that a horizontally simplified basis
(respectively a vertically simplified basis) always exists;
see~\cite[Proposition~\ref*{LOT1:prop:SimplifyComplex}]{LOT1}.

Abusing notation, we will say there is a \emph{length $\ell$ horizontal
  arrow from $\xi_i$ to $\xi_j$} if there is a horizontal arrow from
$\xi_i$ to $U^\ell\xi_j$.

We can invert $U$, giving a complex
$U^{-1}\CFK^-(K)=\Field[U,U^{-1}]\otimes_{\Field[U]}\CFK^-(K)$. (This
complex is also denoted $\CFK^\infty(K)$ in the literature.) It still
makes sense to talk about horizontal and vertical arrows on
$U^{-1}\CFK^-(K)$.  The homology of $U^{-1}\CFK^-(K)$ with respect to
the horizontal (respectively vertical) differentials on
$U^{-1}\CFK^-(K)$ is $\Field[U,U^{-1}]$. If the basis $\{\xi_i\}$ is
horizontally (respectively vertically) simplified then this means
there is a single generator $\eta_0$ (respectively $\xi_0$) over $\Field[U,U^{-1}]$ with no
horizontal (vertical) arrows into or out of it (in $U^{-1}\CFK^-(K)$).

\begin{example}
  \label{ex:gCFKmTrefoil}
  Continuing with Example~\ref{eg:trefoil-1}, we draw the complex
  $\gCFKm(\HD)$ as
  \[
  \begin{tikzpicture}[x=1.5cm,y=32pt]
    \node at (0,0) (b) {$b$};
    \node at (0,1) (a) {$a$};
    \node at (0,-1) (c) {$c$};
    \node at (-1,-1) (Ub) {$Ub$};
    \node at (-1,0) (Ua) {$Ua$};
    \node at (-1,-2) (Uc) {$Uc$};
    \node at (-2,-2) (U2b) {$U^2b$};
    \node at (-2,-1) (U2a) {$U^2a$};
    \node at (-2,-3) (cdots) {$\revddots$};
    \node at (-3,-3) (bdots) {$\revddots$};
    \node at (-3,-2) (adots) {$\revddots$};
    \draw[->] (a) to (b);
    \draw[->] (c) to (Ub);
    \draw[->] (Ua) to (Ub);
    \draw[->] (Uc) to (U2b);
    \draw[->] (U2a) to (U2b);
  \end{tikzpicture}
  \]
  In particular:
  \begin{itemize}
  \item This basis is reduced and both horizontally and vertically simplified.
  \item There is a length $1$ horizontal arrow from $c$ to $b$ and a
    length $1$ vertical arrow from $a$ to $b$.
  \item The element $\eta_0$ is $a$. The element $\xi_0$ is $c$.
  \end{itemize}
\end{example}

Knot Floer homology has been computed extensively. It is determined by
the Alexander polynomial for torus knots~\cite{OS05:surgeries}; it is
determined by the Alexander polynomial and the signature for
alternating knots~\cite{AltKnots}; and it has an efficient combinatorial
description for knots whose doubly-pointed Heegaard diagram can be
drawn on the torus (so that the relevant holomorphic disks are in the
torus, rather than some higher symmetric product)~\cite{OneOne}.
Finally, it admits a purely combinatorial description using grid
diagrams~\cite{MOS06:CombinatorialDescrip,MOST07:CombinatorialLink}, which is amenable to
computations by computer~\cite{BaldwinGillam}, and via a cube of resolutions~\cite{OSS09:singular,OS09:cube}.

\section{From \texorpdfstring{$\CFKa$}{CFK} to \texorpdfstring{$\CFDa$}{CFD}: statement and
  example}\label{sec:CFK-to-CFD} For convenience, we recall our
notation for the torus algebra, from Formula~\ref{eq:torus-alg}. It is
given by:
\[
\Alg(T^2,0)= \xymatrix{
   \iota_0\bullet\ar@/^1pc/[r]^{\rho_1}\ar@/_1pc/[r]_{\rho_3} & \bullet\iota_1\ar[l]_{\rho_2}
  }/(\rho_2\rho_1=\rho_3\rho_2=0).
\]
We have named $\rho_{12}=\rho_1\rho_2$,
$\rho_{23}=\rho_2\rho_3$ and $\rho_{123}=\rho_1\rho_2\rho_3$, so
$\{\iota_0,\iota_1,\rho_1,\rho_2,\rho_3,\rho_{12},\rho_{23},\rho_{123}\}$
is an $\Field$-basis for $\Alg(T^2,0)$.

\begin{theorem}\label{thm:CFK-to-CFD}\cite[Theorem~\ref*{LOT1:thm:HFKtoHFDframed}]{LOT1}
  Let $K\subset S^3$ be a knot and let $\CFK^-(K)$ be a reduced model
  for the knot Floer complex of $K$. Suppose $\CFK^-(K)$ has a
  basis $\{\xi_i\}$ which is both horizontally and vertically
  simplified. 
  
  Fix an integer $n$, and let $Y=S^3\setminus\nbd(K)$ with framing
  $n$. We will describe $\CFDa(Y)$.

  The submodule $\iota_0\CFDa(Y)$ has one generator for each basis
  element $\xi_i$. The submodule $\iota_1\CFDa(Y)$ has basis elements
  coming from the horizontal and vertical arrows in $\CFK^-(K)$.
  Specifically, for each length $\ell$ vertical arrow from $\xi_i$
  to $\xi_j$ here are $\ell$ basis elements
  $\kappa^{ij}_1,\dots,\kappa^{ij}_\ell$ for $\iota_1\CFDa(Y)$; and
  for each length $\ell$ horizontal arrow from $\xi_i$ to $\xi_j$ there
  are $\ell$ basis elements $\lambda^{ij}_1,\dots,\lambda^{ij}_\ell$
  for $\iota_1\CFDa(Y)$. Finally, there are $m=|2\tau(K)-n|$ more basis
  elements $\mu_1,\dots,\mu_{m}$ for $\iota_1\CFDa(Y)$.

  The differential on $\CFDa(Y)$ is given as follows. From the
  vertical arrows we get differentials
  \[
  \xi_i 
  \stackrel{\rho_{1}}\longrightarrow
  \kappa^{ij}_1
  \stackrel{\rho_{23}}\longleftarrow
  \cdots
  \stackrel{\rho_{23}}\longleftarrow
  \kappa^{ij}_{k}
  \stackrel{\rho_{23}}\longleftarrow
  \kappa^{ij}_{k+1}
  \stackrel{\rho_{23}}\longleftarrow
  \cdots
  \stackrel{\rho_{23}}\longleftarrow
  \kappa^{ij}_{\ell}
  \stackrel{\rho_{123}}\longleftarrow
  \xi_{j}.
  \]
  From the horizontal arrows we get differentials
  \[
  \xi_i 
  \stackrel{\rho_{3}}\longrightarrow
  {\lambda}^{ij}_1
  \stackrel{\rho_{23}}\longrightarrow
  \cdots
  \stackrel{\rho_{23}}\longrightarrow
  {\lambda}^{ij}_{k}
  \stackrel{\rho_{23}}\longrightarrow
  {\lambda}^{ij}_{k+1}
  \stackrel{\rho_{23}}\longrightarrow
  \cdots
  \stackrel{\rho_{23}}\longrightarrow
  {\lambda}^{ij}_\ell
  \stackrel{\rho_{2}}\longrightarrow
  \xi_{j}.
  \]
  Finally, we have the unstable chain:
  \begin{itemize}
  \item If $n<2\tau$ the unstable chain has the form
    \[
    \xi_0\stackrel{\rho_1}{\longrightarrow}\mu_1\stackrel{\rho_{23}}\longleftarrow
    \mu_2\stackrel{\rho_{23}}\longleftarrow
    \cdots
    \stackrel{\rho_{23}}\longleftarrow
    \mu_m
    \stackrel{\rho_3}\longleftarrow
    \eta_0.
    \]
  \item If $n>2\tau$ the unstable chain has the form
    \[
    \xi_0 \stackrel{\rho_{123}}{\longrightarrow} \mu_1 
  \stackrel{\rho_{23}}{\longrightarrow} \mu_2
  \dots
  \stackrel{\rho_{23}}{\longrightarrow} \mu_{m}
  \stackrel{\rho_{2}}{\longrightarrow} \eta_0,
    \]
  \item If $n=2\tau$ the unstable chain has the form
    \[
    \xi_0\stackrel{\rho_{12}}{\longrightarrow}
    \eta_0.
    \]
  \end{itemize}
\end{theorem}

It is fairly straightforward to remove the condition that there be a
basis which is both horizontally and vertically simplified: one simply
works with two bases, one horizontally simplified and one vertically
simplified, and keeps track of the transition
matrix. See~\cite[Theorem~\ref*{LOT1:thm:HFKtoHFDframed}]{LOT1}. There is also a basis-free
version of Theorem~\ref{thm:CFK-to-CFD}; see~\cite[Theorem~\ref*{LOT1:thm:HFKtoHFD2}]{LOT1}.

The proof of Theorem~\ref{thm:CFK-to-CFD} has two parts. The first
part is showing that the theorem holds for large negative surgery
coefficients. The argument is somewhat similar to techniques
in~\cite{OS04:Knots,Rasmussen03:Knots,Hedden,IntSurg}, but is still
quite involved. The second part is deducing the result for general
surgery coefficients. This is done by changing the framing one step at
a time, using the bimodules from Exercise~\ref{ex:torus-twists} (or
their type \DA\ analogues).

\begin{example}
  Continuing with the trefoil example, recall that the trefoil $K$ has
  $\tau(K)=-1$. (Compare Exercise~\ref{exr:compute-tau}.) The basis $\{a,b,c\}$ is horizontally and vertically
  simplified. So, $\CFDa$ of $S^3\setminus K$ with framing $1$, say,
  is given by
  \[
  \begin{tikzpicture}
    \node at (0,0) (a) {$a$};
    \node at (0,-2) (l) {$\kappa_{ab}$};
    \node at (0,-4) (b) {$b$};
    \node at (2, -4) (k) {$\lambda_{cb}$};
    \node at (4, -4) (c) {$c$};
    \node at (3.25, -3) (mu1) {$\mu_1$};
    \node at (2.5, -2) (mu2) {$\mu_2$};
    \node at (1.25,-.75) (mu3) {$\mu_3$};
    \draw[->] (a) to node[left]{$\rho_1$} (l);
    \draw[->] (b) to node[left]{$\rho_{123}$} (l);
    \draw[->] (c) to node[below]{$\rho_3$} (k);
    \draw[->] (k) to node[below]{$\rho_{2}$} (b);
    \draw[->] (c) to node[right]{$\rho_{123}$} (mu1);
    \draw[->] (mu1) to node[above,right]{$\rho_{23}$} (mu2);
    \draw[->] (mu2) to node[above,right]{$\rho_{23}$} (mu3);
    \draw[->] (mu3) to node[above]{$\rho_{2}$} (a);
  \end{tikzpicture}
  \]
\end{example}

\section{Studying satellites}\label{sec:satellites}
Suppose that $\HD_1$ is a bordered Heegaard diagram for $S^3\setminus
\nbd(K)$ with the $0$-framing of the boundary. Let $\HD_2$ be a
bordered Heegaard diagram for $\bD^2\times S^1$ with the
$\infty$-framing.
Place an extra basepoint $w$ in $\HD_2$, and let $\HD'_2$ denote the
result. Then $\HD_1\cup_\bdy \HD_2'$ is a doubly-pointed Heegaard
diagram representing a knot $L$ in $S^3$.

\begin{construction}\label{const:diagram-satellite} 
  Fix a doubly-pointed bordered Heegaard diagram
  \[
  \HD=(\Sigma,\alphas^a,\alphas^c,\betas,z,w)
  \]
  for $\bD^2\times S^1$.
  Consider the knot $P$ in $\bD^2\times S^1$ determined as
  follows. Connect the basepoints $z$ and $w$ in $\HD$ by an arc
  $\gamma$ in $\Sigma\setminus (\alphas^a\cup\alphas^c)$ and an arc
  $\eta$ in $\Sigma\setminus \betas$. Viewing $\Sigma$ as
  $\Sigma\times\{1/2\}$ inside $\Sigma\times[0,1]\subset Y(\HD)$
  (Construction~\ref{const:HD-to-mfld}), let $\gamma'$ be the result
  of pushing the interior of $\gamma$ slightly into
  $\Sigma\times[0,1/2)$ and let $\eta'$ be the result of pushing the
  interior of $\eta$ slightly into $\Sigma\times(1/2,1]$. Then let
  $P=\gamma'\cup\eta'$. We will say that $\HD$ \emph{induces
    $(\bD^2\times S^1,P)$}.
\end{construction}

\begin{lemma}\label{lem:diagram-satellite} 
  With notation as above, suppose that $\HD'_2$ induces $(\bD^2\times
  S^1, P)$. Then $L$ is the satellite knot with companion $K\subset
  S^3$ and pattern $P\subset \bD^2\times S^1$.
\end{lemma}
The proof is left as Exercise~\ref{ex:prove-diagram-satellite}.

\begin{figure}
$\mfigb{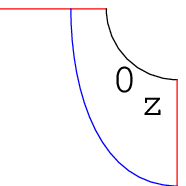}$
\caption[Diagram for the $(2,1)$ cable.]{\label{fig:Cable}
  {\bf{Heegaard diagram for the $(2,1)$-cabling operation.}}
  This is a doubly-pointed Heegaard diagram for the $(2,1)$ cable (of the unknot),
  thought of as a knot in the solid torus. 
  The basepoint $z$ lies in the region marked with a $0$. This picture
  is adapted from~\cite[Figure 11.14]{LOT1}.}
\end{figure}

\begin{example}
  Figure~\ref{fig:Cable} shows a doubly-pointed bordered Heegaard diagram
  inducing the $(2,1)$-cabling operation.
\end{example}

Given a doubly-pointed bordered Heegaard diagram $\HD$, let $\CFD^-(\HD,z,w)$
denote $\Field[U]\otimes_\Field\CFDa(\HD)$ with differential given
by 
\[
\bdy(\x)=\sum_{\y\in\Gen(\HD)}\sum_{n\geq
  0}\sum_{(\rho_1,\dots,\rho_n)}\sum_{B\mid
  \ind(B,\rho_1,\dots,\rho_n)=1}\bigl(\#\cM^B(\x,\y;\rho_1,\dots,\rho_n)\bigr)a(-\rho_1)\cdots a(-\rho_n)U^{n_w(B)}\y.
\]
That is, we count curves as before except that we weight the curves
which cross $w$ $n$ times by $U^n$.
\begin{corollary}\label{cor:satellite-Mor}
  With notation as above,
  \[
  \gCFKm(L)\cong \Mor(\CFDa(-\HD_1),\CFD^-(\HD_2,z,w)).
  \]
\end{corollary}
By Theorem~\ref{thm:CFK-to-CFD}, $\CFDa(\HD_1)$ is determined by
$\CFK^-(K)$.  Thus, if we can compute $\CFD^-(\HD_2,z,w)$ we obtain a
formula for the knot Floer complex $\gCFKm(L)$ in terms of
$\CFK^-(K)$ (for arbitrary $K$).

\begin{example}\label{eg:cable-of-tref}
  In~\cite[Section 11.9]{LOT1} we use these techniques to compute the
  $(2,-3)$ cable of the left-handed trefoil. However, the computation
  there uses the type $A$ invariant of the pattern. In the spirit of
  continuing to avoid $\CFAa$, we give a similar computation using the
  $\Mor$ version of the pairing theorem.

  Let $\HD_2$ denote the doubly-pointed bordered Heegaard diagram shown in
  Figure~\ref{fig:Cable}. The module $\CFD^-(\HD_2,z,w)$ has
  generators $x$, $y_1$ and $y_2$ with 
  \begin{align*}
    \iota_{1}x &= x & \iota_{0}y_1&=y_1 & \iota_{0}y_2&=y_2.
  \end{align*}
  The differentials are given by
  \begin{align*}
    \bdy(x)&=U^2\rho_{23}x\\
    \bdy(y_1)&=Uy_2+\rho_1x\\
    \bdy(y_2)&=U\rho_{123}x.
  \end{align*}

  By Theorem~\ref{thm:CFK-to-CFD}, the invariant $\CFDa(Y)$ of the
  $2$-framed left-handed trefoil complement $Y$ is given by
  \[
  \CFDa(Y)=
  \mathcenter{
  \begin{tikzpicture}
    \node at (0,0) (a) {$a$};
    \node at (0,-1) (l) {$\kappa$};
    \node at (0,-2) (b) {$b$};
    \node at (1.5, -2) (k) {$\lambda$};
    \node at (3, -2) (c) {$c.$};
    \draw[->] (a) to node[left]{$\rho_1$} (l);
    \draw[->] (b) to node[left]{$\rho_{123}$} (l);
    \draw[->] (c) to node[below]{$\rho_3$} (k);
    \draw[->] (k) to node[below]{$\rho_{2}$} (b);
    \draw[->] (c) to node[above]{$\rho_{12}$} (a);
  \end{tikzpicture}}
  \]

  As in Corollary~\ref{cor:satellite-Mor},
  $\Mor(\CFDa(Y),\CFD^-(\HD_2,z,w))$ is $\gCFKm$ of some cable of the
  left-handed trefoil. Computing this morphism space, a basis over
  $\Field[U]$ is given by:
  \begin{align*}
    a&\mapsto y_1 & a&\mapsto\rho_{12}y_1 & a&\mapsto y_2 & a&\mapsto
    \rho_{12}y_2\\
    a&\mapsto \rho_1x & a&\mapsto \rho_3x & a&\mapsto\rho_{123}x \\
    b&\mapsto y_1 & b&\mapsto\rho_{12}y_1 & b&\mapsto y_2 & b&\mapsto
    \rho_{12}y_2\\
    b&\mapsto \rho_1x & b&\mapsto \rho_3x & b&\mapsto\rho_{123}x \\
    c&\mapsto y_1 & c&\mapsto\rho_{12}y_1 & c&\mapsto y_2 & c&\mapsto
    \rho_{12}y_2\\
    c&\mapsto \rho_1x & c&\mapsto \rho_3x & c&\mapsto\rho_{123}x \\
    \lambda &\mapsto x & \lambda&\mapsto \rho_{23}x & \lambda&\mapsto
    \rho_2y_1 & \lambda&\mapsto \rho_2y_2\\
    \kappa &\mapsto x & \kappa&\mapsto \rho_{23}x & \kappa&\mapsto
    \rho_2y_1 & \kappa&\mapsto \rho_2y_2.
  \end{align*}
  (Nobody said this was quick. The complex is smaller if one uses
  $\CFAa(\HD_2,z,w)$.) The differentials are shown in
  Figure~\ref{fig:sat-complex}. Cancelling as many differentials not
  involving $U$ as possible gives
  Figure~\ref{fig:sat-complex-canceled}. In particular, the homology
  $\gHFKm(K)$ is given by 
  $\Field[U]\oplus \bigl(\Field[U]/U^2\bigr)\oplus\Field$; and
  $\HFKa(K)$ is given by $\mathbb{F}_2^5$.

\definecolor{coral}{rgb}{1,0.5,.31}
\definecolor{DeepPink}{rgb}{1,0.08,.58}
\definecolor{DarkGreen}{rgb}{0,0.4,0}
\definecolor{LimeGreen}{rgb}{0.2,0.8,.2}
\definecolor{Navy}{rgb}{0,0,.5}
\definecolor{SkyBlue}{rgb}{.53,.81,.98}
  \begin{figure}
    \centering
    \begin{tikzpicture}[x=3.5cm,y=48pt]
      \node at (0,0) (ay1) {$a\mapsto y_1$};
      \node at (1,0) (ar12y1) {$a\mapsto \rho_{12}y_1$};
      \node at (2,0) (ay2) {$a\mapsto y_2$};
      \node at (3,0) (ar12y2) {$a\mapsto \rho_{12}y_2$};
      \node at (.5,-1) (ar1x) {$a\mapsto \rho_1x$};
      \node at (1.5,-1) (ar3x) {$a\mapsto \rho_3x$};
      \node at (2.5,-1) (ar123x) {$a\mapsto \rho_{123}x$};      
      \node at (0,-2) (by1) {$b\mapsto y_1$};
      \node at (1,-2) (br12y1) {$b\mapsto \rho_{12}y_1$};
      \node at (2,-2) (by2) {$b\mapsto y_2$};
      \node at (3,-2) (br12y2) {$b\mapsto \rho_{12}y_2$};
      \node at (.5,-3) (br1x) {$b\mapsto \rho_1x$};
      \node at (1.5,-3) (br3x) {$b\mapsto \rho_3x$};
      \node at (2.5,-3) (br123x) {$b\mapsto \rho_{123}x$};      
      \node at (0,-4) (cy1) {$c\mapsto y_1$};
      \node at (1,-4) (cr12y1) {$c\mapsto \rho_{12}y_1$};
      \node at (2,-4) (cy2) {$c\mapsto y_2$};
      \node at (3,-4) (cr12y2) {$c\mapsto \rho_{12}y_2$};
      \node at (.5,-5) (cr1x) {$c\mapsto \rho_1x$};
      \node at (1.5,-5) (cr3x) {$c\mapsto \rho_3x$};
      \node at (2.5,-5) (cr123x) {$c\mapsto \rho_{123}x$};      
      \node at (0,-6) (kx) {$\lambda\mapsto x$};
      \node at (1,-6) (kr23x) {$\lambda\mapsto \rho_{23}x$};
      \node at (2,-6) (kr2y1) {$\lambda\mapsto \rho_{2}y_1$};
      \node at (3,-6) (kr2y2) {$\lambda\mapsto \rho_{2}y_2$};
      \node at (0,-7) (lx) {$\kappa\mapsto x$};
      \node at (1,-7) (lr23x) {$\kappa\mapsto \rho_{23}x$};
      \node at (2,-7) (lr2y1) {$\kappa\mapsto \rho_{2}y_1$};
      \node at (3,-7) (lr2y2) {$\kappa\mapsto \rho_{2}y_2$};
      \draw[->,color=red, bend left=10] (ar1x) to node[above]{$U^2$} (ar123x);
      \draw[->,color=red, bend left=10] (br1x) to node[above]{$U^2$} (br123x);
      \draw[->,color=red, bend left=10] (cr1x) to node[above]{$U^2$} (cr123x);
      \draw[->,color=red] (kx) to node[above]{$U^2$} (kr23x);
      \draw[->,color=red] (lx) to node[above]{$U^2$} (lr23x);
      \draw[->,color=orange, bend left=10] (ay1) to node[above]{$U$} (ay2);
      \draw[->,color=orange, bend left=10] (ar12y1) to node[above]{$U$} (ar12y2);
      \draw[->,color=orange, bend left=10] (by1) to node[above]{$U$} (by2);
      \draw[->,color=orange, bend left=10] (br12y1) to node[above]{$U$} (br12y2);
      \draw[->,color=orange, bend left=10] (cy1) to node[above]{$U$} (cy2);
      \draw[->,color=orange, bend left=10] (cr12y1) to node[above]{$U$} (cr12y2);
      \draw[->,color=orange] (kr2y1) to node[above]{$U$} (kr2y2);
      \draw[->,color=orange] (lr2y1) to node[above]{$U$} (lr2y2);
      \draw[->,color=pink] (ay1) to (ar1x);
      \draw[->,color=pink] (by1) to (br1x);
      \draw[->,color=pink] (cy1) to (cr1x);
      \draw[->,color=DeepPink] (ay2) to node[above]{$U$} (ar123x);
      \draw[->,color=DeepPink] (by2) to node[above]{$U$} (br123x);
      \draw[->,color=DeepPink] (cy2) to node[above]{$U$} (cr123x);
      \draw[->, color=green] (lx) to (ar1x);
      \draw[->, color=green] (lr23x) to (ar123x);
      \draw[->, color=green] (lr2y1) to (ar12y1);
      \draw[->, color=green] (lr2y2) to (ar12y2);
      \draw[->, color=DarkGreen] (lx) to (br123x);
      \draw[->, color=LimeGreen] (by1) to (kr2y1);
      \draw[->, color=LimeGreen] (by2) to (kr2y2);
      \draw[->, color=LimeGreen] (br3x) to (kr23x);
      \draw[->, color=blue] (kx) to (cr3x);
      \draw[->, color=SkyBlue] (ay1) to (cr12y1);
      \draw[->, color=SkyBlue] (ay2) to (cr12y2);           
      \draw[->, color=SkyBlue] (ar3x) to (cr123x);            
    \end{tikzpicture}
    \caption{\textbf{The complex from Example~\ref{eg:cable-of-tref}}}
      \label{fig:sat-complex}
  \end{figure}
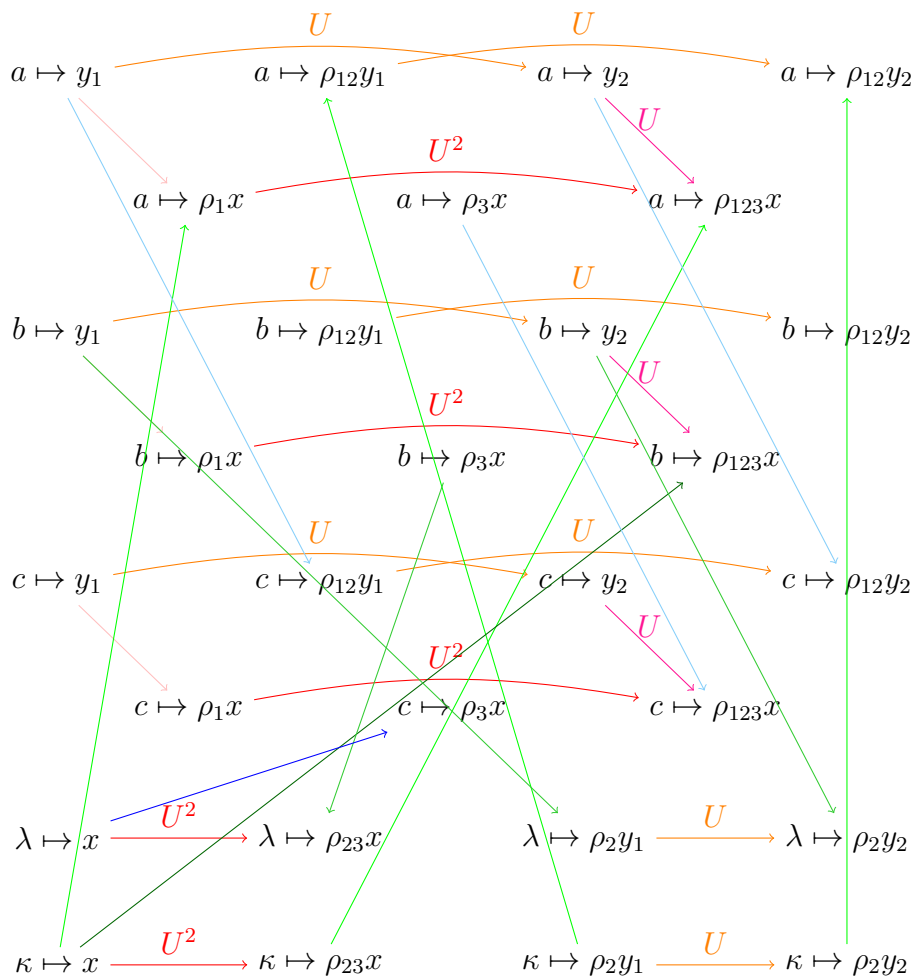

  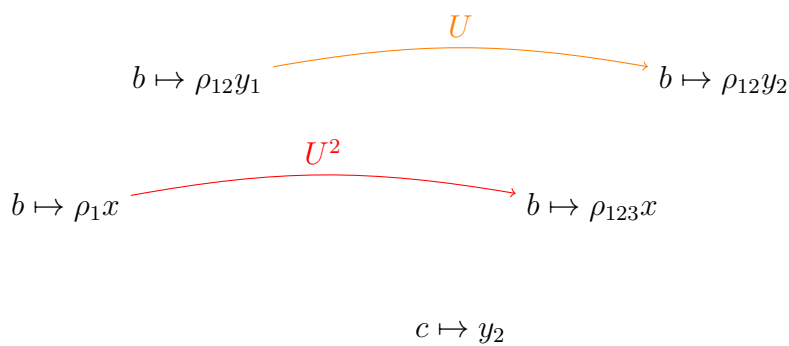
\begin{figure}
    \centering
    \begin{tikzpicture}[x=3.5cm,y=48pt]
      \node at (1,-2) (br12y1) {$b\mapsto \rho_{12}y_1$};
      \node at (3,-2) (br12y2) {$b\mapsto \rho_{12}y_2$};
      \node at (.5,-3) (br1x) {$b\mapsto \rho_1x$};
      \node at (2.5,-3) (br123x) {$b\mapsto \rho_{123}x$};      
      \node at (2,-4) (cy2) {$c\mapsto y_2$};
      \draw[->,color=red, bend left=10] (br1x) to node[above]{$U^2$} (br123x);
      \draw[->,color=orange, bend left=10] (br12y1) to node[above]{$U$} (br12y2);
    \end{tikzpicture}
    \caption{\textbf{Result of cancelling differentials in Figure~\ref{fig:sat-complex}.}}
      \label{fig:sat-complex-canceled}
  \end{figure}
\end{example}

In some sense, this strategy works in general:
\begin{lemma}\label{lem:diagram-satellite-exist}
  Given any pattern $P$ in $\bD^2\times S^1$ there is a doubly-pointed
  Heegaard diagram inducing $P$.
\end{lemma}
The proof is left as Exercise~\ref{ex:prove-diagram-satellite-exists}.
 
\begin{corollary}
  Let $P$ be a knot in $\bD^2\times S^1$. Given a knot $K$ in $S^3$
  let $K^P$ denote the satellite of $K$ with pattern $P$.  Then
  $\CFK^-(K)$ determines $\gCFKm(K^P)$ in the following sense: if
  $K_1$ and $K_2$ are knots with $\CFK^-(K_1)\cong\CFK^-(K_2)$ then
  $\gCFKm(K_1^P)\cong \gCFKm(K_2^P)$.
\end{corollary}

\begin{remark}
  The diagram $\HD'_2$ specifies more than just a knot in $\bD^2\times
  S^1$; see Exercise~\ref{ex:not-knot}. Probably the best way to think
  of $\HD'_2$ is as representing a bordered-sutured manifold (in the
  sense of~\cite{Zarev09:BorSut}).
\end{remark}

\section{Exercises}

\begin{exercise}\label{exr:compute-tau}
  For $K$ the trefoil and the figure eight, compute the $\Field[U]$
  module structure on $H_*(\gCFKm(K))$, using the descriptions of the
  complexes given in 
  Examples~\ref{eg:trefoil-1} and~\ref{eg:fig8} respectively.
  Use this to compute $\tau(K)$ for these knots.
\end{exercise}

\begin{exercise}
  Find a basis for $\CFKm(K)$ when $K$ is the figure eight knot
  which is both horizontally and vertically simplified. 
\end{exercise}

\begin{exercise}
  Let $Y$ be the complement of the unknot in $S^3$. Compute $\CFDa(Y)$
  in two ways:
  \begin{enumerate}
  \item Using Theorem~\ref{thm:CFK-to-CFD}.
  \item Directly from a bordered Heegaard diagram.
  \end{enumerate}
  (This exercise is courtesy of J.~Hom.)
\end{exercise}

\begin{exercise}
  Using Theorem~\ref{thm:CFK-to-CFD}, write down $\CFDa$ of the
  trefoil complement with framings $1$ and $-2$.
\end{exercise}

\begin{exercise}
  Figure~\ref{fig:TrefoilComplement} gives a bordered Heegaard diagram for
  the trefoil complement. Compute $\CFDa$ of that diagram directly,
  and compare the answer with that given by
  Theorem~\ref{thm:CFK-to-CFD}. (This is a fairly challenging
  computation, after which you are guaranteed to appreciated
  Theorem~\ref{thm:CFK-to-CFD}.)
\end{exercise}

\begin{exercise}
  Verify that the modules $\CFDa(Y)$ given by Theorem~\ref{thm:CFK-to-CFD}
  satisfy $\bdy^2=0$.
\end{exercise}

\begin{exercise}\label{ex:prove-diagram-satellite}
  Prove Lemma~\ref{lem:diagram-satellite}.
\end{exercise}

\begin{exercise}\label{ex:prove-diagram-satellite-exists}
  Prove Lemma~\ref{lem:diagram-satellite-exist}.
\end{exercise}

\begin{exercise}
  Use the bimodules of Exercise~\ref{ex:torus-twists} to show that if
  Theorem~\ref{thm:CFK-to-CFD} holds for surgery coefficient $n$ then
  it holds for surgery coefficient $n\pm 1$. (This is somewhat messy.)
\end{exercise}

\begin{exercise}\label{ex:not-knot} Find doubly-pointed bordered Heegaard diagrams
  $\HD$, $\HD'$ for $\bD^2\times S^1$ so that:
  \begin{itemize}
  \item The singly-pointed Heegaard diagrams obtained from $\HD$,
    $\HD'$ by forgetting the $w$ basepoint both specify the same
    framing for $\bD^2\times S^1$.
  \item The diagrams $\HD$ and $\HD'$ represent the same satellite
    operation in the sense of Construction~\ref{const:diagram-satellite}.
  \item The invariants $\CFD^-(\HD,z,w)$ and $\CFD^-(\HD',z',w')$ are not
    homotopy equivalent.
  \end{itemize}
  In particular, it is not true that any two diagrams representing the
  same pattern $P$ are related by a sequence of Heegaard moves in the
  complement of the basepoints.
\end{exercise}

\begin{exercise}
  We computed $\gCFKm$ of some cable of the trefoil in
  Example~\ref{eg:cable-of-tref}. Which one?
\end{exercise}

\chapter[Factoring mapping classes]{Computing with bordered Floer homology II: factoring mapping classes}\label{lec:compute-HFa}
The goal of this lecture is to discuss an algorithm, coming from
bordered Floer homology, for computing the invariant $\HFa(Y)$ for any
closed $3$-manifold $Y$. This is not the first algorithm for computing
$\HFa(Y)$, which is due
to Sarkar-Wang~\cite{SarkarWang07:ComputingHFhat}; but it is
independent of the Sarkar-Wang algorithm and conceptually fairly
satisfying. The algorithm has been implemented using Sage; see
\url{http://math.columbia.edu/~lipshitz/research.html\#Programming}.

\section{Overview of the algorithm}
Fix a closed $3$-manifold $Y$ and a Heegaard splitting 
\[
Y=\HB_1\cup_\psi \HB_2
\]
for $Y$. That is, $\HB_1$ and $\HB_2$ are handlebodies of some genus
$k$ and $\psi\co \bdy\HB_1\to \bdy\HB_2$ is an orientation-reversing
homeomorphism.

\begin{figure}
  \centering
  \includegraphics{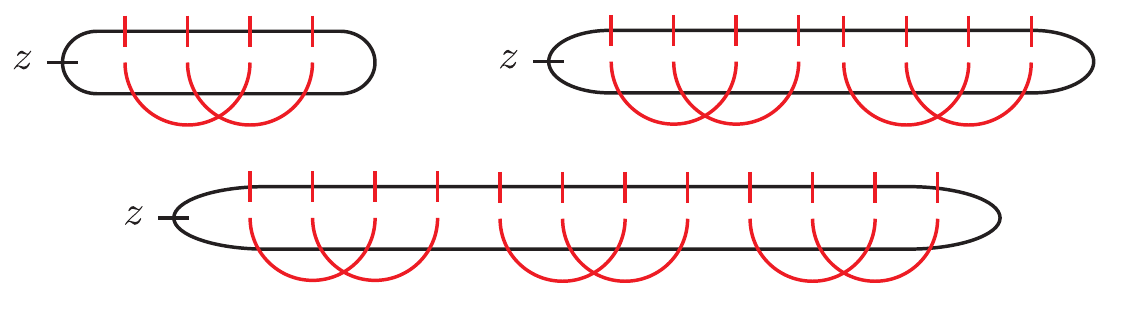}
  \caption{\textbf{The pointed matched circles $\PMC_k^0$.} The cases
    $k=1$, $k=2$ and $k=3$ are shown.}
  \label{fig:split-matching}
\end{figure}

Without loss of generality, we can assume that each $\HB_i$ is a
particular standard bordered handlebody $(\HB_k,\phi_0\co F(\PMC_k^0)\to
\bdy \HB_k)$. Here, $\PMC_k^0$ is a particular pointed matched
circle---we will take it to be the $k$-fold connect sum of the genus
$1$ pointed matched circle (i.e., the \emph{split matching}; see Figure~\ref{fig:split-matching}). Then the
map $\psi$ is specified by a map $\tpsi=\phi_0\circ \psi\circ \phi_0^{-1}\co
F(\PMC_k^0)\to F(\PMC_k^0)$. To specify $Y$ up to
homeomorphism we need only specify $\psi$ up to isotopy; so, it is
natural to view $\tpsi$ as an element of
the mapping class group of $F(\PMC_k^0)$. Up to isotopy, we can assume
that $\tpsi$ fixes the preferred disk in $F(\PMC_k^0)$, and regard
it as an element of the mapping class group of
$\PunctF(\PMC_k^0)$. (Of course, the lift to the strongly based
mapping class group depends on a choice.)

Let $M_{\tpsi}$ denote the mapping cylinder of $\tpsi$, as in
Example~\ref{eg:mapping-cyl}. Then by the relevant pairing theorems,
Corollary~\ref{cor:mod-mod-hom} and Theorem~\ref{thm:bimod-mod-hom},
we have
\[
\CFa(Y)\simeq \Mor\Bigl(\CFDa(\HB_k,\phi_0),\Mor\bigl(\CFDDa(-M_\tpsi),\CFDa(\HB_k,\phi_0)\bigr)\Bigr).
\]
So, we have ``reduced'' the problem to computing the invariants of
$(\HB_k,\phi_0)$ and $M_\tpsi$. 

This is not yet useful: there are about as many mapping classes as
$3$-manifolds. On the other hand, the mapping classes form a
group. Suppose that $\psi_1,\dots,\psi_N$ are generators for the
mapping class group of $\PunctF(\PMC_k^0)$ as a monoid---that is,
we include inverses in our list of generators. Then we can write
$\tpsi=\psi_{i_n}\circ\dots\circ \psi_{i_1}$ for some sequence of
generators
$\psi_{i_1},\dots,\psi_{i_n}\in\{\psi_1,\dots,\psi_N\}$. 
Repeatedly using Theorem~\ref{thm:bimod-mod-hom}, we have
\begin{multline*}
  \CFa(Y)\simeq
  \Mor\biggl(\CFDa(-\HB_k,\phi_0),\Mor\Bigl(\CFDDa(-M_{\psi_{i_n}}),\Mor\bigl(\cdots\\
\dots,\Mor(\CFDDa(-M_{\psi_{i_1}}),\CFDa(\HB_k,\phi_0))\dots\bigr)\Bigr)\biggr).
\end{multline*}
So, we now really have reduced the problem: we only need to compute the
invariants $\CFDa(\HB_k,\phi_0)$ and $\CFDDa(M_{\psi_i})$ for our
preferred set of generators $\psi_1,\dots,\psi_N$.

\subsection{Arc-slides as mapping class groupoid generators}
Generalizing the mapping class group to a groupoid leads to a particularly convenient set of generators.

\begin{definition}
  The \emph{genus $k$ mapping class groupoid} is the category whose
  objects are the pointed matched circles representing genus $g$
  surfaces, and with $\Hom(\PMC_1,\PMC_2)$ the set of isotopy classes
  of strongly-based homeomorphisms $F(\PMC_1)\to F(\PMC_2)$.

  In particular, $\Aut(\PMC)=\Hom(\PMC,\PMC)$ is the strongly-based
  mapping class group.
\end{definition}

\begin{figure}
\begin{center}
\input{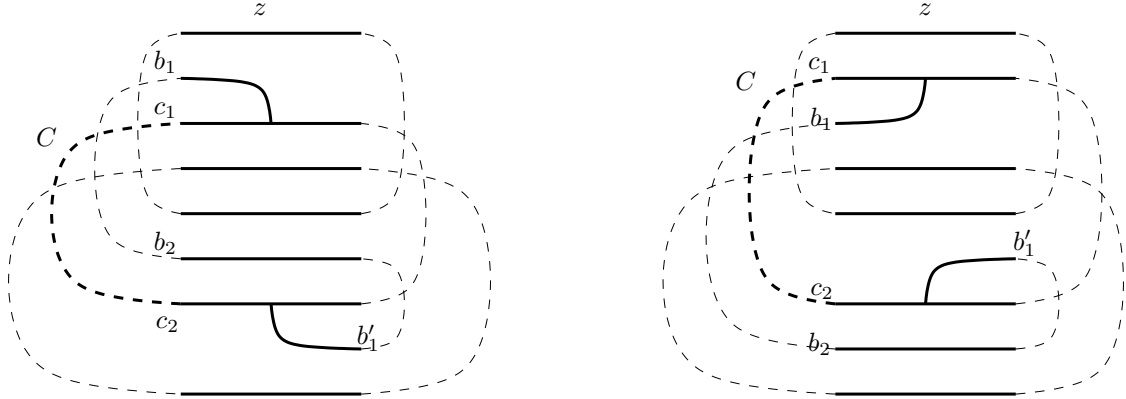}
\end{center}
\caption {{\bf Arc-slides.}
\label{fig:ArcslideMatching}
Two examples of arc-slides connecting pointed matched circles
for genus $2$ surfaces. In both cases, the foot $b_1$ is sliding
over the matched pair $C=\{c_1,c_2\}$ (indicated by the darker dotted
matching) at $c_1$. This figure is~\cite[Figure~\ref*{HFa:fig:ArcslideMatching}]{LOT4}.}
\end{figure}

\begin{definition}\label{def:arcslide}
  Let $\PMC$ be a pointed matched circle, and fix two matched pairs
$C=\{c_1,c_2\}$ and $B=\{b_1,b_2\}$ in $\PMC$. Suppose moreover that $b_1$ and
$c_1$ are adjacent, in the sense that there is an arc $\sigma$ connecting $b_1$
and $c_1$ which does not contain the basepoint $z$ or any other
point $p_i\in\mathbf{a}$.
Then we can form a new pointed matched circle $\PMC'$ which agrees
everywhere with $\PMC$, except that $b_1$ is replaced by a new
distinguished point $b_1'$, which now is adjacent to $c_2$ and $b_1'$ is positioned so that
the orientation on the arc from $b_1$ to $c_1$ is opposite to the
orientation of the arc from $b_1'$ to~$c_2$. In this case, we say that
$\PMC'$ and $\PMC$ differ by an \emph{arc-slide of $b_1$ over
  $c_1$}. (See Figure~\ref{fig:ArcslideMatching} for two
examples.)

In this situation, there is a canonical
element in $\Hom(\PMC,\PMC')$, which we refer to as the
\emph{arc-slide diffeomorphism}; see Figure~\ref{fig:ArcSlide}.
\end{definition}

\begin{figure}
\begin{center}
\input{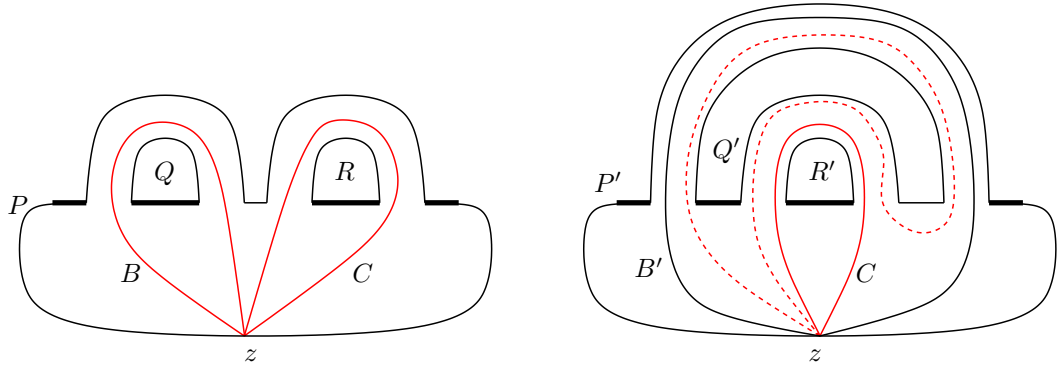}
\end{center}
\caption {{\bf The local case of an arc-slide diffeomorphism.}
\label{fig:ArcSlide}
Left: a pair of pants with boundary components labeled $P$, $Q$, and
$R$, and two distinguished curves $B$ and $C$. Right:
another pair of pants with boundary components $P'$, $Q'$, $R'$ and
distinguished curves $B'$ and $C$. The arc-slide diffeomorphism
carries $B$ to the dotted curve on the right, the curve labeled
$C$ on the left to the curve labeled $C$ on the right, and boundary
components $P$, $Q$, and $R$ to $P'$, $Q'$ and $R'$ respectively.
This diffeomorphism can be extended to a diffeomorphism between
surfaces associated to pointed matched circles: in such a surface
there are further handles attached along the four dark
intervals; however, our diffeomorphism carries the four dark intervals
on the left to the four dark intervals on the right and hence extends
to a diffeomorphism as stated. (This is only one of several possible
configurations of $B$ and $C$: they could also be nested or linked.)
This figure is~\cite[Figure~\ref*{HFa:fig:ArcSlide}]{LOT4}.}
\end{figure}

The diagrams in Figure~\ref{fig:ArcslideMatching} are shorthand for
bordered Heegaard diagrams for the mapping cylinders of the arc-slides. Such a bordered Heegaard
diagram for the second arc-slide in Figure~\ref{fig:ArcslideMatching}
is given in Figure~\ref{fig:HD-for-arcslide}.

\begin{figure}
  \centering
  \includegraphics{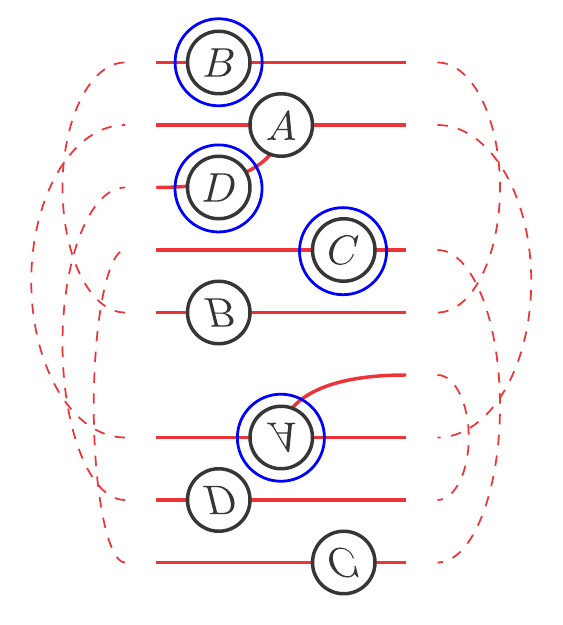}
  \caption{\textbf{Heegaard diagram for an arc-slide.} This diagram
    corresponds to the schematic on the right of Figure~\ref{fig:ArcslideMatching}.}
  \label{fig:HD-for-arcslide}
\end{figure}

\begin{lemma}\label{lem:arcslides-generate}
  The arc-slides generate the mapping class groupoid.
\end{lemma}
A proof can be found in~\cite{Bene08:ChordDiagrams}. It is perhaps a more familiar
fact that the mapping class group is generated by some finite,
preferred set of Dehn twists; see for
example~\cite{Humphries}. Lemma~\ref{lem:arcslides-generate} can be
deduced from this more familiar fact by explicitly factoring that particular
collection of Dehn twists into arcslides (see
Example~\ref{ex:FactorDehnTwist}).

\begin{example}
  \label{ex:FactorDehnTwist}
  Figure~\ref{fig:factor-Dehn} shows a factorization of a (particular) Dehn twist as a
  product of arc-slides.
\end{example}

\begin{figure}
  \centering
  \includegraphics[width=\textwidth]{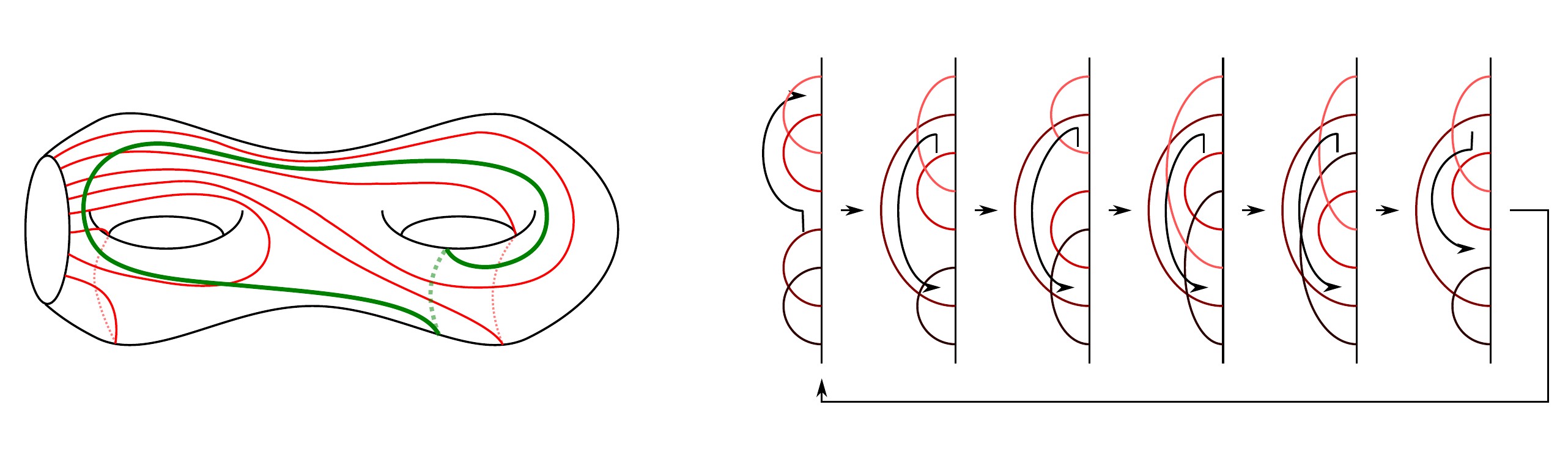}
  \caption{\textbf{Factoring a Dehn twist into arc-slides.} Left: a
    genus $2$ surface specified by a pointed matched circle, and a
    curve $\gamma$ (drawn in thick green) in it. Right: a sequence of arc-slides whose
    composition is a Dehn twist around $\gamma$. This
    is~\cite[Figure~\ref*{HFa:fig:factor-Dehn}]{LOT4}.}
  \label{fig:factor-Dehn}
\end{figure}

So, two steps remain to compute $\CFa$:
\begin{itemize}
\item Compute $\CFDa(\HD_k)$ for some Heegaard diagram $\HD_k$
  representing the genus $k$ handlebody.
\item Compute $\CFDDa(M_{\psi})$ for any arc-slide $\psi$.
\end{itemize}
We give these computations in Sections~\ref{sec:handlebody}
and~\ref{sec:underslides}, respectively. (As a warm-up before
computing the invariant of arc-slides we compute the type \DD\ module
associated to the identity cobordism.)

\begin{remark}
  The relations among arc-slides are also relatively easy to
  state; see~\cite{Bene08:ChordDiagrams}.
\end{remark}

\section{The invariant of a particular handlebody}\label{sec:handlebody}
Let $\PMC^1$ denote the (unique) pointed matched circle for the
torus, and let $\PMC^k$ denote the $k$-fold connect sum of $\PMC^1$
with itself,
i.e., the genus $k$ \emph{split pointed matched circle}. Label the
marked points in $\PMC^k$ as $a_1,\dots,a_{4k}$. So, in $\PMC^k$ the
matched pairs are $\{a_{4i-3},a_{4i-1}\}$ and $\{a_{4i},a_{4i-2}\}$.
 
The \emph{$0$-framed solid torus} $\HB^1=(H^1,\phi^1_0)$ is the solid
torus with boundary $-{F(\PMC^1)}$ in which the handle determined by
$\{a_1,a_3\}$ bounds a disk.
Let $\phi^1_0$ denote the preferred diffeomorphism $-{F(\PMC^1)}\to
\bdy H^1$.  The \emph{$0$-framed handlebody of genus $k$}
$\HB^k=(H^k,\phi^k_0)$ is a boundary connect sum of $k$ copies of
$\HB^1$. Our conventions are illustrated by the bordered Heegaard
diagram in Figure~\ref{fig:GenusTwoBorderedDiagram}.

\begin{figure}
\begin{center}
\input{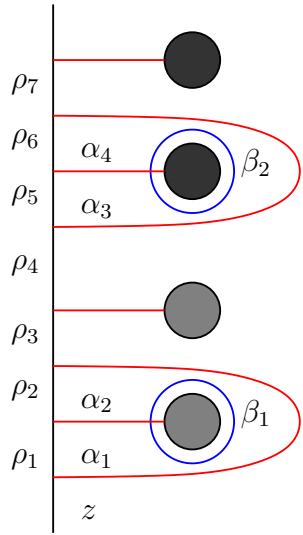}
\end{center}
\caption {{\bf Heegaard diagram for the $0$-framed genus two
    handlebody.} The lighter (respectively darker) shaded pair of
  circles indicates a handle attached to the diagram. This is~\cite[Figure~\ref*{HFa:fig:GenusTwoBorderedDiagram}]{LOT4}.
\label{fig:GenusTwoBorderedDiagram}}
\end{figure}

\begin{proposition}
  Let $\SetS=\{a_{4i-3},a_{4i-1}\}_{i=1}^k$. The module $\CFDa(\HB^k)$ is
  generated over the algebra by a single element $\x$ with $I(\SetS)\x=\x$,
  and is equipped with the differential determined by
  $$\bdy(\x)=\sum_{i=1}^k a(\xi_i)\x,$$ 
  where $\xi_i$ is the arc in $\PMC^k$ connecting $a_{4i-3}$ and
  $a_{4i-1}$.
\end{proposition}
\begin{proof}
  This is a simple computation from the definitions. Note that the
  domains of holomorphic curves contributing to the differential on
  $\CFDa(\HB^k)$ must be connected. It follows that the curves
  appearing here are simply copies of the curves occurring in the
  differential on $\CFDa(\HB^1)$. These, in turn, were already studied
  in Section~\ref{sec:surgery}.
\end{proof}

\section{The \DD\ identity}\label{sec:DD-id}
Let $\Id$ denote the identity arced cobordism of $F(\PMC)$.  As a
warm-up to computing the bimodules associated to arc-slides we compute
the bimodule $\CFDDa(\Id)$. The standard bordered Heegaard diagram
$\HD(\Id)$ for the identity cobordism (for a particular choice of
$\PMC$) is illustrated in
Figure~\ref{fig:Genus2Identity}. Inspecting the diagram, one has two
immediate observations:
\begin{enumerate}
\item Recall that indecomposable idempotents of $\Alg(\PMC)$
  correspond to subsets of the matched pairs in $\PMC$. There is an
  obvious bijection between matched pairs in $\PMC$ and matched pairs
  in $-\PMC$. With respect to this bijection, the generators of
  $\CFDDa(\Id)$ correspond one-to-one with pairs of indecomposable
  idempotents $I(\SetS)\otimes I(\SetT)\in
  \Alg(\PMC)\otimes\Alg(-\PMC)$ with $\SetS\cap \SetT=\emptyset$. We
  call such pairs \emph{complementary idempotents}.  (The set of
  complementary idempotents is also in bijection with the set of
  idempotents of $\Alg(\PMC)$, of course.)

  Given a pair of complementary idempotents $I\otimes
  I'$ let $\x_{I,I'}$ denote the corresponding generator of $\CFDDa(\Id)$.
\item\label{item:diagonal} Any domain in $\HD(\Id)$ has the same multiplicities at the two
  boundaries of $\HD(\Id)$.  Any basic element of
  $\Alg(\PMC)$ has an associated \emph{support} in $H_1(Z\setminus
  \{z\},\CircPts)$; let $\support{\xi}$ denote the support of
  $\xi$. It follows that if $(\xi\otimes \xi')\otimes \x_{J,J'}$ occurs in
  $\bdy(\x_{I,I'})$ then $\support{\xi}=\support{\xi'}$ (in the obvious
  sense).
\end{enumerate}
\begin{figure}
\begin{center}
\input{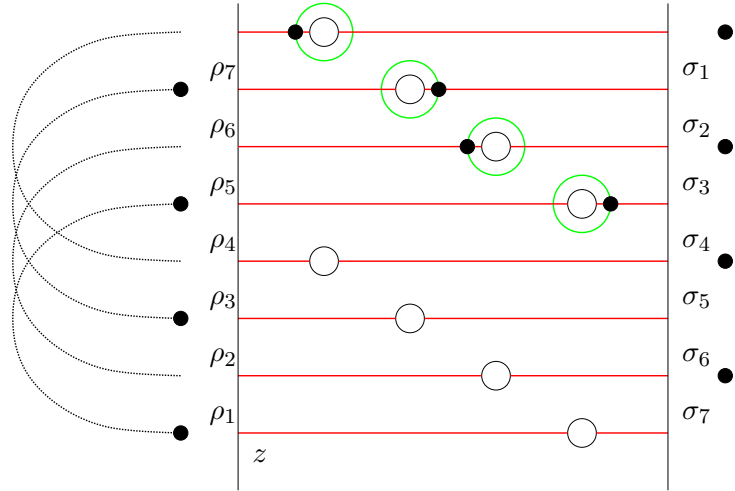}
\end{center}
\caption {{\bf Heegaard diagram for the identity map.}
\label{fig:Genus2Identity}
This is a Heegaard diagram for the identity cobordism of the genus two surface with antipodal 
matching, as indicated by the arcs to the left of the diagram. To the left and the right of the
diagram, we have also
indicated a pair of complementary idempotents, along with its unique extension into the
diagram as a generator for the complex. This figure is~\cite[Figure~\ref*{HFa:fig:Genus2Identity}]{LOT4}.}
\end{figure}

Formalizing the above, let the \emph{diagonal subalgebra} of
$\Alg(\PMC)\otimes \Alg(-\PMC)$ denote the subalgebra with basis
\[
\{(I\cdot \xi \cdot J)\otimes (I'\cdot \xi'\cdot J') \mid
\support{\xi}=\support{\xi'},\ (I,I')\text{ complementary},\ (J,J')\text{ complementary}\}.
\]

\begin{proposition}\label{prop:gr-on-diag}
  The diagonal subalgebra has a $\ZZ$-grading $\gr$ with the following
  properties:
  \begin{enumerate}
  \item\label{item:diag-dg} The grading $\gr$ respects the differential algebra structure,
    i.e., for homogeneous elements $a$ and $b$,
    $\gr(ab)=\gr(a)+\gr(b)$ and $\gr(d(a))=\gr(a)-1$.
  \item\label{item:diag-resp-DD} The differential on $\CFDDa(\Id)$ is homogeneous of degree
    $-1$ with respect to $\gr$.
  \item The standard basis elements for the diagonal subalgebra are
    homogeneous with respect to $\gr$.
  \item If $a\in \Alg$ is homogeneous then $\gr(a)\leq 0$.
  \item If $\gr(a)=0$ then $a$ is an idempotent.
  \item If $\gr(a)=-1$ then $a$ is a linear combination of
    \emph{chords}, i.e., elements of the form $a(\rho)\otimes
    a'(\rho)$ where $\rho$ is a single chord in $\PMC$. (Here,
    $a'(\rho)$ denotes the element of $\Alg(-\PMC)$ associated to the chord
    $\rho$.)
  \end{enumerate}
\end{proposition}
\begin{proof}[Sketch of proof.]
  There are at least two ways to go about this proof. One is to show
  that any element of the diagonal algebra can be factored as a
  product of chords, and the length of the factorization is
  unique. (This is the approach taken in~\cite[Section~\ref*{HFa:sec:DDforIdentity}]{LOT4}.)
  Another approach is to observe that there is a \dg algebra with
  properties~(\ref{item:diag-dg}) and~(\ref{item:diag-resp-DD})
  associated to any type \DD\ bimodule
  (or type $D$ module);
  we call this the
  \emph{coefficient
    algebra}~\cite[Sections~\ref*{HFa:sec:coeff-algebra} and~\ref*{HFa:sec:coeff-bimod}]{LOT4}. In
  the case of $\CFDDa(\Id)$, the coefficient algebra is exactly the
  diagonal subalgebra. Verifying the remaining properties above is
  then a fairly simple computation. (This is the approach taken for
  arc-slide bimodules in~\cite[Section~\ref*{HFa:sec:Arc-Slides}]{LOT4}.)
\end{proof}

\begin{corollary}\label{cor:DD-diff-only-chords}
  If $(a\otimes b)\otimes \x_{J,J'}$ occurs in $\bdy(\x_{I,I'})$ then $a\otimes
  b$ is a linear combination of chords $a(\rho_i)\otimes a(\rho'_i)$.
\end{corollary}

Let $\Chord(\PMC)$ denote the set of all chords for $\PMC$.

\begin{theorem}
  \label{thm:DDid} As a bimodule, $\CFDDa(\Id)$ is given by
  \[
  \CFDDa(\Id)=\bigoplus_{(I\otimes I')\text{ complementary}}(\Alg(\PMC)\cdot
  I)\otimes_\Field (\Alg(-\PMC)\cdot I') \otimes \x_{I,I'}.
  \]
  The differential of $\x_{I,I'}$ is given by
  \[
  \bdy(\x_{I,I'})=\sum_{(J,J')}\sum_{\rho\in\Chord(\PMC)}\bigl[(I\cdot
  a(\rho)\cdot J)\otimes (I'\cdot a'(\rho)\cdot J')\bigr]\otimes \x_{J,J'}.
  \]
  In other word, every term permitted by
  Corollary~\ref{cor:DD-diff-only-chords} to occur in $\bdy(\x_{I,I'})$
  does occur.
\end{theorem}
\begin{proof}[Sketch of proof]
  All that remains is to show that every term of the form $\bigl[(I\cdot
  a(\rho)\cdot J)\otimes (I'\cdot a'(\rho)\cdot J')\bigr]\otimes
  \x_{J,J'}$ does occur in $\bdy\x_{I,I'}$. The argument is by
  induction on the support to $\rho$. The base case is when $\rho$ has
  length $1$. In this case, the corresponding domain in $\HD(\Id)$ is
  a hexagon, so it follows from the Riemann mapping theorem that there
  is a holomorphic representative. 

  \begin{figure}
    \centering
    \includegraphics{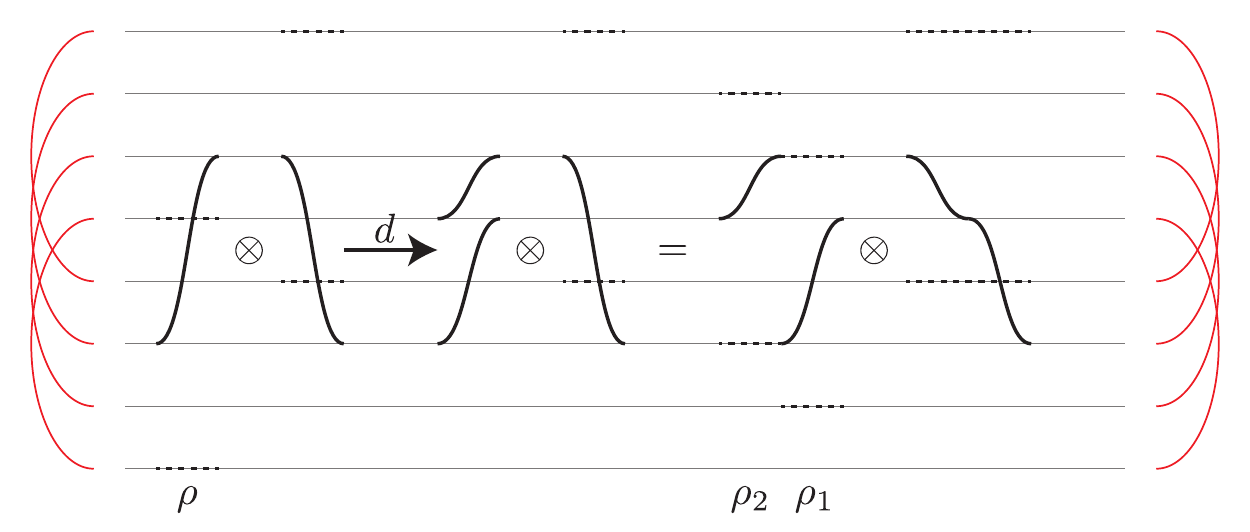}
    \caption{\textbf{Illustration of the inductive step in the proof
        of Theorem~\ref{thm:DDid}.} We want to show the term on the
      left occurs in $\bdy$ on $\CFDDa(\Id)$. The term on the far
      right occurs in $\bdy^2$, by induction on the length of the
      chords involved. The only other contribution to $\bdy^2$ which
      could cancel it is the differential of the term on the
      left. (The differential of the term on the left also has other
      terms, not shown.)}
    \label{fig:rho-exists}
  \end{figure}

  The rest of the induction argument is illustrated in
  Figure~\ref{fig:rho-exists}. In words, suppose $\rho$ has length
  bigger than $1$, and suppose there is a position $a\in\CircPts$ so
  that:
  \begin{itemize}
  \item $a$ lies in the interior of $\rho$ and
  \item the matched pair containing $a$ is in the idempotent $I$.
  \end{itemize}
  Let $\rho_1$ be the chord from the start of $\rho$ to the point $a$
  and let $\rho_2$ be the chord from $a$ to the end of $\rho$. By
  induction, $\bdy^2(\x_{I,I'})$ contains a term of the form
  $\bigl[(I\cdot a(\rho_2)a(\rho_1)J)\otimes (I'\cdot a'(\rho)\cdot
  J')\bigr]\otimes \x_{J,J'}$; this term comes from the sequence
  \begin{align*}
  \x_{I,I'}&\stackrel{\bdy}{\longrightarrow} 
  \bigl[(I\cdot a(\rho_2))\otimes (I'\cdot a'(\rho_2))\bigr]\otimes \x_{K,K'}\\
  &\stackrel{\bdy}{\longrightarrow} 
  \bigl[(I\cdot a(\rho_2)a(\rho_1))\otimes (I'\cdot
  a'(\rho_2)a'(\rho_1))\bigr]\otimes \x_{J,J'}\\
  &\qquad\qquad\qquad=
  \bigl[(I\cdot a(\rho_2)a(\rho_1))\otimes (I'\cdot a'(\rho))\bigr]\otimes \x_{J,J'}.
  \end{align*}
  The only term in $\bdy^2(\x_{I,I'})$ which could cancel this one is
  $\bigl[(I\cdot \bdy a(\rho)\cdot J)\otimes (I\cdot a'(\rho)\cdot
  J')\bigr]\otimes \x_{J,J'}$. Thus, since $\bdy^2=0$, the term
  $\bigl[(I\cdot a(\rho)\cdot J)\otimes (I\cdot a'(\rho)\cdot
  J')\bigr]\otimes \x_{J,J'}$ must occur in $\bdy(\x_{I,I'})$.

  If there is a position $a$ in the interior of $\rho$ occupied in the
  idempotent $I'$ then a similar argument, with the left and right
  sides reversed, gives the result. The only other case is that of
  length three chords in which both of the interior positions are
  matched to the endpoints. We call such chords \emph{special length
    $3$ chords} in~\cite{LOT4}. There are various ways to handle this
  case. A somewhat indirect argument is given in the proof
  of~\cite[Theorem~\ref*{HFa:thm:DDforIdentity}]{LOT4}. One can also prove the result in this
  case by a direct computation, as in the proof of~\cite[Proposition~\ref*{LOT2:prop:TypeDDTorus}]{LOT2}.
\end{proof}

\begin{remark}
  The bimodule $\CFDDa(\Id)$ exhibits a kind of duality between the
  algebras $\Alg(\PMC)$ and $\Alg(-\PMC)$, called \emph{Koszul
    duality}. See, for instance, \cite[Section~\ref*{HomPair:sec:koszul}]{LOTHomPair}.
\end{remark}

\section{Underslides}\label{sec:underslides}
To explain the bimodule $\CFDDa$ associated to an arc-slide we first
divide the arc-slides into two classes: underslides and
overslides. Specifically, with notation as in
Definition~\ref{def:arcslide}, $Z\setminus C$ has two connected
components. One of these components contains the basepoint $z$; call
that component $Z_z$. Then an arc-slide is an \emph{overslide} if
$b_1\in Z_z$, and is an \emph{underslide} if $b_1\not\in Z_z$. So, in
Figure~\ref{fig:ArcslideMatching}, the example on the left is an
overslide while the example on the right is an underslide.

It turns out that the bimodules for underslides are a little simpler,
so we will focus on this case, referring the reader to~\cite[Section~\ref*{HFa:subsec:Over-Slides}]{LOT4} for the overslide case. So, let $\psi\co \PMC\to\PMC'$ be an underslide and
$M_\psi$ the associated mapping cylinder. To describe $\CFDDa(M_\psi)$
we need two more pieces of terminology:
\begin{definition}
  There is an obvious bijection between matched pairs of $\PMC$ (i.e.,
  $1$-handles of $F(\PMC)$) and matched pairs of $\PMC'$ (i.e.,
  $1$-handles of $F(\PMC)$).  With notation as in
  Definition~\ref{def:arcslide}, a pair of indecomposable idempotents
  $I(\SetS)\otimes I(\SetS')\in\Alg(\PMC)\otimes \Alg(-\PMC')$ are called
  \emph{near-complementary} if either:
  \begin{itemize}
  \item $\SetS$ is complementary to $\SetS'$ or
  \item $\SetS\cap \SetT$ consists of the matched pair of the feet of
    $C$, while $\SetS\cup\SetT$ contains all the matched pairs except
    for the pair of feet of $B$.
  \end{itemize}
\end{definition}

\begin{figure}
    \begin{center}
      \input{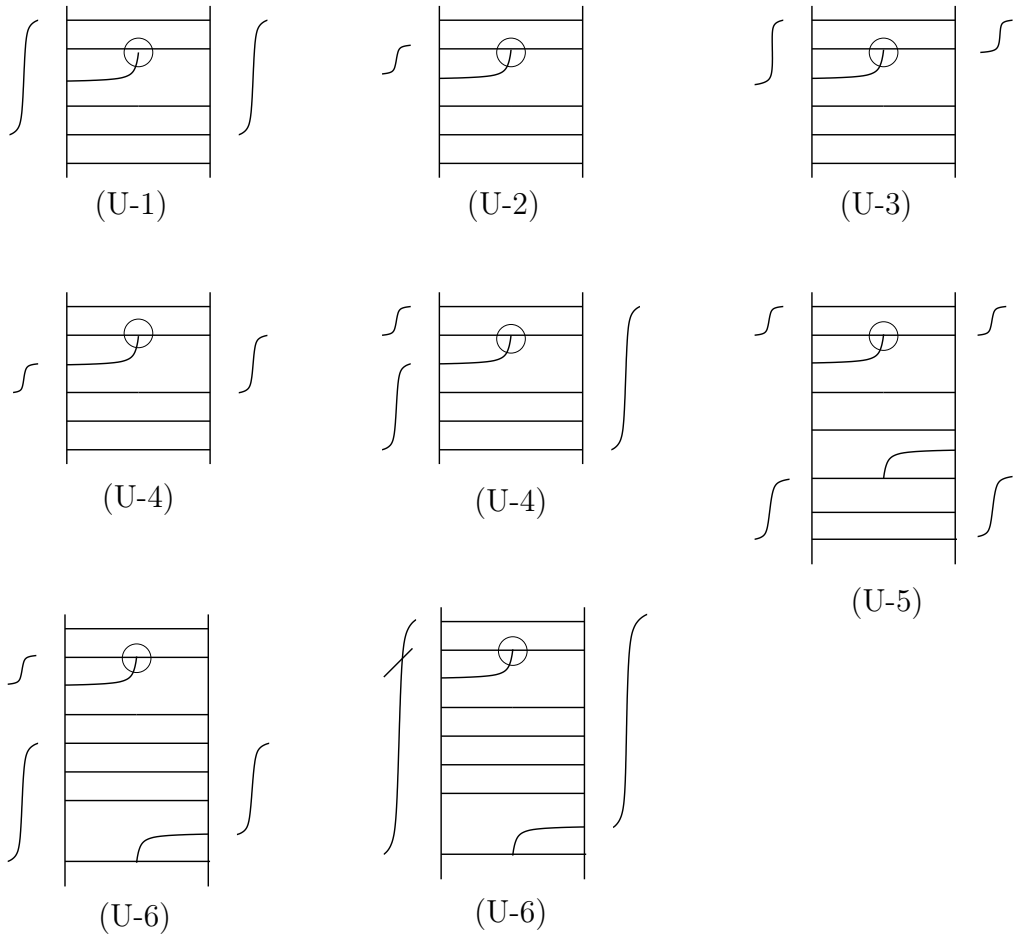}
    \end{center}
    \caption {{\bf Near-chords for under-slides.}
      \label{fig:UnderNearChords} 
  }
\end{figure}

\begin{definition}
  A \emph{near-chord} for the underslide $\psi$ is an algebra element
  of the form $a(\xi)\otimes a'(\xi')$ where $\xi$ (respectively
  $\xi'$) is a collection of chords in $\PMC$ (respectively $-\PMC'$)
  of one of the forms (U-1)--(U-6) shown in
  Figure~\ref{fig:UnderNearChords}.

  Let $\NearChord(\psi)$ denote the set of near-chords for $\psi$. 
\end{definition}
(See~\cite[Definition~\ref*{HFa:def:UnderNearChord}]{LOT4} for a more detailed description of
the types (U-1)--(U-6) of near-chords.)

\begin{theorem}\label{thm:DD-underslide}
  The bimodule $\CFDDa(M_\psi)$ has one generator $\x_{I,I'}$ for each
  near\hyp complementary pair of idempotents $I\otimes I'$; and
  $\x_{I,I'}=(I\otimes I')\cdot \x_{I,I'}$. (In other words, as a module
  $\CFDDa(M_\psi)\cong\bigoplus_{I\otimes I'\text{ near
      complementary}}\bigl(\Alg(\PMC)\cdot I\bigr) \otimes
  \bigl(\Alg(-\PMC')\cdot I'\bigr).$) The differential on
  $\CFDDa(M_\psi)$ is given by 
  \[
  \bdy(\x_{I,I'})=\sum_{\substack{(J,J')\\\text{
      near-complementary}}}\sum_{(\xi,\xi')\in\NearChord(\psi)} \bigl[(I\cdot
  a(\xi)\cdot J)\otimes (I'\cdot a'(\xi')\cdot J')\bigr]\otimes\x_{J,J'}.
  \]
\end{theorem}
\begin{proof}[Sketch of proof.]
  The proof is similar to, though more involved than, the proof of
  Theorem~\ref{thm:DDid}. There is an analogue of the diagonal
  algebra, called the \emph{near-diagonal algebra}, admitting a
  $\ZZ$-grading satisfying analogous properties to
  Proposition~\ref{prop:gr-on-diag}. In particular, the near-chords
  are exactly the basic elements in grading $-1$. So, it only remains
  to show that every near-chord occurs in the differential. This
  follows from an inductive argument similar to the proof of
  Theorem~\ref{thm:DDid}. For short near-chords---near chords of type
  (U-2) and minimal-length near-chords of types (U-1) and (U-4)---it
  follows from the Riemann mapping theorem that the chords occur in
  the differential. The existence of other near-chords follows by a
  (somewhat complicated) induction on the support, using only the fact
  that $\bdy^2=0$.
\end{proof}

We do not discuss the case of overslides, which are more complicated
than underslides.
At the heart of the complication is the fact that, for overslides,
the coefficient algebra contains non-idempotent elements in grading
$0$ (whereas in the underslide case, all non-idempotent elements have
negative grading). While in the underslide
case, every element in grading $-1$ appears as a coefficient in the
differential, in the overslide case which grading $-1$ elements appear
depends on a choice. Nonetheless, the index
zero elements can be used to induce maps between bimodules associated
to the various choices, and a somewhat weaker analogue of
Theorem~\ref{thm:DD-underslide} holds: the
overslide bimodule can be computed explicitly after some
combinatorial choices are made, and the homotopy type of the answer is
independent of those combinatorial choices. The interested reader is
referred to~\cite[Proposition~\ref*{HFa:prop:CalculateOver-Slide}]{LOTHomPair}.

\section{Exercises}

\begin{exercise}
  Verify the type \DD\ bimodule for the identity cobordism of the
  torus given in Exercise~\ref{ex:DD-id-torus} agrees with the answer
  given by Theorem~\ref{thm:DDid}.
\end{exercise}

\begin{exercise}
  Verify that the bimodules from Exercise~\ref{ex:torus-twists} agree
  with the bimodules given by Theorem~\ref{thm:DD-underslide}. (Note
  that one can view each of these Dehn twists as an underslide.)
\end{exercise}

\begin{exercise}
  Extend the algorithm above to compute $\CFDa(Y)$ for any bordered
  $3$-manifold $Y$.
\end{exercise}

\bibliographystyle{hamsalpha}
\bibliography{heegaardfloer}

\newcommand{\etalchar}[1]{$^{#1}$}
\providecommand{\noopsort}[1]{}
\providecommand{\bysame}{\leavevmode\hbox to3em{\hrulefill}\thinspace}
\providecommand{\href}[2]{#2}
\providecommand{\eprint}{\begingroup \urlstyle{rm}\Url}
\begin{thebibliography}{MO{\relax Sz}T07}

\bibitem[Aur10]{AurouxBordered}
Denis Auroux, \emph{Fukaya categories of symmetric products and bordered
  {H}eegaard-{F}loer homology}, J. G\"okova Geom. Topol. \textbf{4} (2010),
  1--54, \eprint{arXiv:1001.4323}.

\bibitem[BEH{\etalchar{+}}03]{BEHWZ03:CompactnessInSFT}
Frédéric Bourgeois, Yakov Eliashberg, Helmut Hofer, Kris Wysocki, and Eduard
  Zehnder, \emph{Compactness results in symplectic field theory}, Geom. Topol.
  \textbf{7} (2003), 799--888, \eprint{arXiv:math.SG/0308183}.

\bibitem[Ben10]{Bene08:ChordDiagrams}
Alex~James Bene, \emph{A chord diagrammatic presentation of the mapping class
  group of a once bordered surface}, Geom. Dedicata \textbf{144} (2010),
  171--190, \eprint{arXiv:0802.2747}.

\bibitem[BG12]{BaldwinGillam}
John~A. Baldwin and William~D. Gillam, \emph{Computations of {H}eegaard-{F}loer
  knot homology}, J. Knot Theory Ramifications \textbf{21} (2012), no.~8,
  1250075, \eprint{arXiv:math/0610167}.

\bibitem[BL94]{BernsteinLunts94:EquivariantSheaves}
Joseph Bernstein and Valery Lunts, \emph{Equivariant sheaves and functors},
  Lecture Notes in Mathematics, vol. 1578, Springer-Verlag, Berlin, 1994.

\bibitem[CGH12a]{ColinGhigginiHonda11:HF-ECH-1}
Vincent Colin, Paolo Ghiggini, and Ko~Honda, \emph{\noopsort{A}{T}he
  equivalence of {H}eegaard {F}loer homology and embedded contact homology via
  open book decompositions {I}}, 2012, \eprint{arXiv:1208.1074}.

\bibitem[CGH12b]{ColinGhigginiHonda11:HF-ECH-2}
\bysame, \emph{\noopsort{B}{T}he equivalence of {H}eegaard {F}loer homology and
  embedded contact homology via open book decompositions {II}}, 2012,
  \eprint{arXiv:1208.1077}.

\bibitem[CGH12c]{ColinGhigginiHonda11:HF-ECH-3}
\bysame, \emph{\noopsort{C}{T}he equivalence of {H}eegaard {F}loer homology and
  embedded contact homology {III}: from hat to plus}, 2012,
  \eprint{arXiv:1208.1526}.

\bibitem[Ghi08]{Ghiggini08:FiberedGenusOne}
Paolo Ghiggini, \emph{Knot {F}loer homology detects genus-one fibred knots},
  Amer. J. Math. \textbf{130} (2008), no.~5, 1151--1169.

\bibitem[GMM05]{OneOne}
Hiroshi Goda, Hiroshi Matsuda, and Takayuki Morifuji, \emph{Knot {F}loer
  homology of {$(1,1)$}-knots}, Geom. Dedicata \textbf{112} (2005), 197--214.

\bibitem[Hed05]{Hedden}
Matthew Hedden, \emph{On knot {F}loer homology and cabling}, Algebr. Geom.
  Topol. \textbf{5} (2005), 1197--1222, \eprint{arXiv:math.GT/0406402}.

\bibitem[Hom12]{Hom:bord-tau}
Jennifer Hom, \emph{Bordered {H}eegaard {F}loer homology and the tau-invariant
  of cable knots}, 2012, \eprint{arXiv:1202.1463}.

\bibitem[Hum79]{Humphries}
Stephen~P. Humphries, \emph{Generators for the mapping class group}, Topology
  of low-dimensional manifolds ({P}roc. {S}econd {S}ussex {C}onf., {C}helwood
  {G}ate, 1977), Lecture Notes in Math., vol. 722, Springer, Berlin, 1979,
  pp.~44--47.

\bibitem[JM08]{JabukaMark08:product}
Stanislav Jabuka and Thomas~E. Mark, \emph{Product formulae for
  {O}zsv\'ath-{S}zab\'o 4-manifold invariants}, Geom. Topol. \textbf{12}
  (2008), no.~3, 1557--1651.

\bibitem[JT12]{JT:Naturality}
Andr{\'a}s Juh{\'a}sz and Dylan~P. Thurston, \emph{Naturality and mapping class
  groups in {H}eegaard {F}loer homology}, 2012, \eprint{arXiv:1210.4996}.

\bibitem[Juh06]{Juhasz06:Sutured}
Andr{\'a}s Juh{\'a}sz, \emph{Holomorphic discs and sutured manifolds}, Algebr.
  Geom. Topol. \textbf{6} (2006), 1429--1457, \eprint{arXiv:math/0601443}.

\bibitem[Kho01]{Khovanov01:Nilcoxeter}
Mikhail Khovanov, \emph{Nilcoxeter algebras categorify the {W}eyl algebra},
  Comm. Algebra \textbf{29} (2001), no.~11, 5033--5052,
  \eprint{arXiv:math.RT/9906166}.

\bibitem[Kho10]{Khovanov10:gl12}
\bysame, \emph{How to categorify one-half of quantum $\gl(1\mid 2)$}, 2010,
  \eprint{arXiv:1007.3517}.

\bibitem[KLT10a]{KutluhanLeeTaubes:HFHMI}
{\c{C}}a\u{g}atay Kutluhan, Yi-Jen Lee, and Clifford~Henry Taubes,
  \emph{{HF=HM} {I}: {H}eegaard {F}loer homology and {S}eiberg--{W}itten
  {F}loer homology}, 2010, \eprint{arXiv:1007.1979}.

\bibitem[KLT10b]{KutluhanLeeTaubes:HFHMII}
\bysame, \emph{{HF=HM} {II}: {R}eeb orbits and holomorphic curves for the
  ech/{H}eegaard-{F}loer correspondence}, 2010, \eprint{1008.1595}.

\bibitem[KLT10c]{KutluhanLeeTaubes:HFHMIII}
\bysame, \emph{{HF=HM} {III}: {H}olomorphic curves and the differential for the
  ech/{H}eegaard {F}loer correspondence}, 2010, \eprint{arXiv:1010.3456}.

\bibitem[KLT11]{KutluhanLeeTaubes:HFHMIV}
\bysame, \emph{{HF=HM} {IV}: The {S}eiberg-{W}itten {F}loer homology and ech
  correspondence}, 2011, \eprint{arXiv:1107.2297}.

\bibitem[KLT12]{KutluhanLeeTaubes:HFHMV}
\bysame, \emph{{HF=HM} {V}: {S}eiberg-{W}itten {F}loer homology and handle
  additions}, 2012, \eprint{arXiv:1204.0115}.

\bibitem[Lev12a]{Levine12:doubling}
Adam~Simon Levine, \emph{Knot doubling operators and bordered {H}eegaard
  {F}loer homology}, J. Topol. \textbf{5} (2012), no.~3, 651--712,
  \eprint{arXiv:1008.3349}.

\bibitem[Lev12b]{Levine12:slicingBing}
\bysame, \emph{Slicing mixed {B}ing-{W}hitehead doubles}, J. Topol. \textbf{5}
  (2012), no.~3, 713--726, \eprint{arXiv:0912.5222}.

\bibitem[Lip06]{Lipshitz06:CylindricalHF}
Robert Lipshitz, \emph{A cylindrical reformulation of {H}eegaard {F}loer
  homology}, Geom. Topol. \textbf{10} (2006), 955--1097,
  \eprint{arXiv:math.SG/0502404}.

\bibitem[LOTa]{LOT:DCov2}
Robert Lipshitz, Peter~S. Ozsv{\'a}th, and Dylan~P. Thurston, \emph{Bordered
  {F}loer homology and the spectral sequence of a branched double cover {II}:
  the spectral sequences agree}, in preparation.

\bibitem[LOTb]{LOTCobordisms}
\bysame, \emph{Computing cobordism maps with bordered {F}loer homology}, in
  preparation.

\bibitem[LOT08]{LOT1}
\bysame, \emph{Bordered {H}eegaard {F}loer homology: {I}nvariance and pairing},
  2008, \eprint{arXiv:0810.0687v4}.

\bibitem[LOT09]{LOT0}
\bysame, \emph{Slicing planar grid diagrams: a gentle introduction to bordered
  {H}eegaard {F}loer homology}, Proceedings of {G}\"okova {G}eometry-{T}opology
  {C}onference 2008, G\"okova Geometry/Topology Conference (GGT), G\"okova,
  2009, pp.~91--119, \eprint{arXiv:0810.0695}.

\bibitem[LOT10a]{LOT2}
\bysame, \emph{Bimodules in bordered {H}eegaard {F}loer homology}, 2010,
  \eprint{arXiv:1003.0598v3}.

\bibitem[LOT10b]{LOT:DCov1}
\bysame, \emph{Bordered {F}loer homology and the spectral sequence of a
  branched double cover {I}}, 2010, \eprint{arXiv:1011.0499}.

\bibitem[LOT10c]{LOT4}
\bysame, \emph{Computing {$\HFa$} by factoring mapping classes}, 2010,
  \eprint{arXiv:1010.2550v3}.

\bibitem[LOT11a]{LOTHomPair}
Robert Lipshitz, Peter~S. Ozsv\'ath, and Dylan~P. Thurston, \emph{{H}eegaard
  {F}loer homology as morphism spaces}, Quantum Topology \textbf{2} (2011),
  no.~4, 384--449, \eprint{arXiv:1005.1248}.

\bibitem[LOT11b]{LOT11:tour}
Robert Lipshitz, Peter~S. Ozsv{\'a}th, and Dylan~P. Thurston, \emph{Tour of
  bordered {F}loer theory}, Proc. Natl. Acad. Sci. USA \textbf{108} (2011),
  no.~20, 8085--8092, \eprint{arXiv:1107.5621}.

\bibitem[LOT13]{LOT13:faith}
\bysame, \emph{A faithful linear-categorical action of the mapping class group
  of a surface with boundary}, J. Eur. Math. Soc. (JEMS) \textbf{15} (2013),
  no.~4, 1279--1307, \eprint{arXiv:1012.1032}.

\bibitem[MO10]{ManolescuOzsvath:surgery}
Ciprian Manolescu and Peter Ozsv{\'a}th, \emph{Heegaard {F}loer homology and
  integer surgeries on links}, 2010, \eprint{arXiv:1011.1317}.

\bibitem[MOS09]{MOS06:CombinatorialDescrip}
Ciprian Manolescu, Peter Ozsv{\'a}th, and Sucharit Sarkar, \emph{A
  combinatorial description of knot {F}loer homology}, Ann. of Math. (2)
  \textbf{169} (2009), no.~2, 633--660, \eprint{arXiv:math.GT/0607691}.

\bibitem[MO{\relax Sz}T07]{MOST07:CombinatorialLink}
Ciprian Manolescu, Peter~S. Ozsv{\'a}th, Zolt{\'a}n {\relax Sz}ab{\'o}, and
  Dylan~P. Thurston, \emph{On combinatorial link {F}loer homology}, Geom.
  Topol. \textbf{11} (2007), 2339--2412, \eprint{arXiv:math.GT/0610559}.

\bibitem[MOT09]{MOT:grid}
Ciprian Manolescu, Peter Ozsv{\'a}th, and Dylan Thurston, \emph{Grid diagrams
  and {H}eegaard {F}loer invariants}, 2009, \eprint{arXiv:0910.0078}.

\bibitem[Ni09]{Ni09:FiberedMfld}
Yi~Ni, \emph{Heegaard {F}loer homology and fibred 3-manifolds}, Amer. J. Math.
  \textbf{131} (2009), no.~4, 1047--1063.

\bibitem[OS{\relax Sz}09]{OSS09:singular}
Peter Ozsv{\'a}th, Andr{\'a}s Stipsicz, and Zolt{\'a}n {\relax Sz}ab{\'o},
  \emph{Floer homology and singular knots}, J. Topol. \textbf{2} (2009), no.~2,
  380--404.

\bibitem[O{\relax Sz}03]{AltKnots}
Peter~S. Ozsv{\'a}th and Zolt{\'a}n {\relax Sz}ab{\'o}, \emph{Heegaard {F}loer
  homology and alternating knots}, Geom. Topol. \textbf{7} (2003), 225--254
  (electronic).

\bibitem[O{\relax Sz}04a]{OS04:ThurstonNorm}
\bysame, \emph{Holomorphic disks and genus bounds}, Geom. Topol. \textbf{8}
  (2004), 311--334.

\bibitem[O{\relax Sz}04b]{OS04:Knots}
\bysame, \emph{Holomorphic disks and knot invariants}, Adv. Math. \textbf{186}
  (2004), no.~1, 58--116, \eprint{arXiv:math.GT/0209056}.

\bibitem[O{\relax Sz}04c]{OS04:HolDiskProperties}
\bysame, \emph{Holomorphic disks and three-manifold invariants: properties and
  applications}, Ann. of Math. (2) \textbf{159} (2004), no.~3, 1159--1245,
  \eprint{arXiv:math.SG/0105202}.

\bibitem[O{\relax Sz}04d]{OS04:HolomorphicDisks}
\bysame, \emph{Holomorphic disks and topological invariants for closed
  three-manifolds}, Ann. of Math. (2) \textbf{159} (2004), no.~3, 1027--1158,
  \eprint{arXiv:math.SG/0101206}.

\bibitem[O{\relax Sz}04e]{OS04:symplectic}
\bysame, \emph{Holomorphic triangle invariants and the topology of symplectic
  four-manifolds}, Duke Math. J. \textbf{121} (2004), no.~1, 1--34.

\bibitem[O{\relax Sz}05a]{OS05:surgeries}
\bysame, \emph{On knot {F}loer homology and lens space surgeries}, Topology
  \textbf{44} (2005), no.~6, 1281--1300, \eprint{arXiv:math/0303017}.

\bibitem[O{\relax Sz}05b]{BrDCov}
\bysame, \emph{On the {H}eegaard {F}loer homology of branched double-covers},
  Adv. Math. \textbf{194} (2005), no.~1, 1--33, \eprint{arXiv:math.GT/0309170}.

\bibitem[O{\relax Sz}06]{OS06:HolDiskFour}
\bysame, \emph{Holomorphic triangles and invariants for smooth four-manifolds},
  Adv. Math. \textbf{202} (2006), no.~2, 326--400,
  \eprint{arXiv:math.SG/0110169}.

\bibitem[O{\relax Sz}08]{IntSurg}
\bysame, \emph{Knot {F}loer homology and integer surgeries}, Algebr. Geom.
  Topol. \textbf{8} (2008), no.~1, 101--153.

\bibitem[O{\relax Sz}09]{OS09:cube}
Peter Ozsv{\'a}th and Zolt{\'a}n {\relax Sz}ab{\'o}, \emph{A cube of
  resolutions for knot {F}loer homology}, J. Topol. \textbf{2} (2009), no.~4,
  865--910.

\bibitem[Per08]{Perutz07:handleslides}
Timothy Perutz, \emph{Hamiltonian handleslides for {H}eegaard {F}loer
  homology}, Proceedings of {G}\"okova {G}eometry-{T}opology {C}onference 2007,
  G\"okova Geometry/Topology Conference (GGT), G\"okova, 2008, pp.~15--35.

\bibitem[Pet09]{Petkova:cableofthin}
Ina Petkova, \emph{Cables of thin knots and bordered {H}eegaard {F}loer
  homology}, 2009, \eprint{arXiv:0911.2679}.

\bibitem[Ras03]{Rasmussen03:Knots}
Jacob Rasmussen, \emph{Floer homology and knot complements}, Ph.D. thesis,
  Harvard University, Cambridge, MA, 2003, \eprint{arXiv:math.GT/0306378}.

\bibitem[Rob08]{Roberts08:blow-downs}
Lawrence~P. Roberts, \emph{Rational blow-downs in {H}eegaard-{F}loer homology},
  Commun. Contemp. Math. \textbf{10} (2008), no.~4, 491--522.

\bibitem[SW10]{SarkarWang07:ComputingHFhat}
Sucharit Sarkar and Jiajun Wang, \emph{An algorithm for computing some
  {H}eegaard {F}loer homologies}, Ann. of Math. (2) \textbf{171} (2010), no.~2,
  1213--1236, \eprint{arXiv:math/0607777}.

\bibitem[Tau10a]{Taubes10:SW-ECH-I}
Clifford~Henry Taubes, \emph{Embedded contact homology and {S}eiberg-{W}itten
  {F}loer cohomology {I}}, Geom. Topol. \textbf{14} (2010), no.~5, 2497--2581.

\bibitem[Tau10b]{Taubes10:SW-ECH-II}
\bysame, \emph{Embedded contact homology and {S}eiberg-{W}itten {F}loer
  cohomology {II}}, Geom. Topol. \textbf{14} (2010), no.~5, 2583--2720.

\bibitem[Tau10c]{Taubes10:SW-ECH-III}
\bysame, \emph{Embedded contact homology and {S}eiberg-{W}itten {F}loer
  cohomology {III}}, Geom. Topol. \textbf{14} (2010), no.~5, 2721--2817.

\bibitem[Tau10d]{Taubes10:SW-ECH-IV}
\bysame, \emph{Embedded contact homology and {S}eiberg-{W}itten {F}loer
  cohomology {IV}}, Geom. Topol. \textbf{14} (2010), no.~5, 2819--2960.

\bibitem[Tau10e]{Taubes10:SW-ECH-V}
\bysame, \emph{Embedded contact homology and {S}eiberg-{W}itten {F}loer
  cohomology {V}}, Geom. Topol. \textbf{14} (2010), no.~5, 2961--3000.

\bibitem[Zar09]{Zarev09:BorSut}
Rumen Zarev, \emph{Bordered {F}loer homology for sutured manifolds}, 2009,
  \eprint{arXiv:0908.1106}.

\end{thebibliography}
\end{document}